\documentclass[a4paper]{amsart}
\usepackage[top=1.25in,bottom=1.25in,left=1.25in,right=1.25in]{geometry}

\usepackage[utf8]{inputenc}
\usepackage[T1]{fontenc}
\usepackage[english]{babel}

\usepackage{hyperref}
\hypersetup{
colorlinks=true,
}

\usepackage{enumerate}
\usepackage{graphicx}
\usepackage{tikz-cd}
\usepackage{amsmath}
\usepackage{amsfonts}
\usepackage{amssymb}
\usepackage{amsthm}
\usepackage{mathrsfs}
\usepackage{bm}
\usepackage{widebar}
\allowdisplaybreaks

\newtheorem{theorem}{Theorem}[section]

\newtheorem{conjecture}[theorem]{Conjecture}
\newtheorem{corollary}[theorem]{Corollary}
\newtheorem{lemma}[theorem]{Lemma}

\newtheorem{proposition}[theorem]{Proposition}

\theoremstyle{definition}

\newtheorem{definition}[theorem]{Definition}
\newtheorem{example}[theorem]{Example}

\theoremstyle{remark}
\newtheorem{remark}[theorem]{Remark}

\numberwithin{equation}{section}

\let\emptyset\varnothing
\let\leq\leqslant \let\geq\geqslant

\def\bydef{\mathrel{\mathop:}=}
\def\dif{\mathop{}\!\textnormal{d}}
\def\ddif{\mathop{}\!\textnormal{dd}}

\def\dual{\raisebox{.5ex}{$\scriptscriptstyle \vee$}}
\def\mltp{\raisebox{.5ex}{$\scriptscriptstyle \times$}}

\DeclareMathOperator{\Bs}{Bs} 
\DeclareMathOperator{\Pluc}{Pluc} 
\DeclareMathOperator{\Bl}{Bl} 
\DeclareMathOperator{\Gr}{Gr}

\DeclareMathOperator{\Proj}{\textbf{Proj}} 
\DeclareMathOperator{\pr}{pr} 
\DeclareMathOperator{\res}{res}

\DeclareMathOperator{\Span}{Span} 
\DeclareMathOperator{\Sym}{Sym} 
\DeclareMathOperator{\Supp}{Supp} 
 
\DeclareMathOperator{\rank}{rank}

\DeclareMathOperator{\reg}{reg} 
\DeclareMathOperator{\loc}{loc} 
\DeclareMathOperator{\sing}{sing} 
\DeclareMathOperator{\Sing}{Sing} 
\DeclareMathOperator{\Ind}{Ind}

\begin{document}

\title{Lemma on logarithmic derivative over directed manifolds}
\author{Peiqiang Lin}
\address{Graduate School of Mathematics, Nagoya University}
\email{lpq.fj@hotmail.com}
\maketitle

\begin{abstract}
In this paper, we generalize Ahlfors' lemma on logarithmic derivative to holomorphic tangent curves of directed projective manifolds intersecting closed subschemes.
As a consequence, we obtain Algebro-Geometric Ahlfors' Lemma on Logarithmic Derivative (AALD for short) and General form of Algebro-Geometric Version of Ahlfors' Lemma on Logarithmic Derivative (GAALD for short) for holomorphic tangent curves of directed projective manifolds.
We also get a transform of AALD and GAALD with respect to a linear system.
Finally, we get the Second Main Theorem type results for holomorphic curves as the applications of GAALD and its transform.
\end{abstract}

\tableofcontents

\setcounter{section}{-1}
\section{Introduction}

\subsection{Background}
In 1925, Nevanlinna \cite{Nevanlinna1925} evolved a theory that carries his name, through his groundbreaking work on meromorphic functions.
It consists of two fundamental pillars: the First Main Theorem (FMT for short) and the Second Main Theorem (SMT for short).

The First Main Theorem can be viewed as the integral version of the fundamental theorem of algebra.
The fundamental theorem of algebra indicates the equality between the degree of polynomial functions and the number of their roots, and FMT extends this concept to meromorphic functions by taking the double integral operator $\displaystyle \int_0^r \frac{dt}{t} \int_{\bigtriangleup(t)}$ (see Section \ref{ssc:FMT} for details).
However, a boundary term (known as the proximity function) arises due to the non-compactness of the disk $\bigtriangleup(r)$.
It is evident that the proximity function is always non-negative.
Hence achieving the upper bound of the proximity function is crucial to understand the value distribution of meromorphic functions on non-compact domains $\bigtriangleup(r)$.
The Second Main Theorem, which stands as Nevanlinna's most profound achievement, addresses this very issue.
It gives an upper bound for the sum of proximity functions for arbitrary finite many values (including the infinity).
Moreover, the upper bound given by SMT also controls the ramification term, which implies that SMT can be viewed as the integral version of Riemann-Hurwitz theorem on non-compact domains $\bigtriangleup(r)$.
Because his work focuses on how functions distribute their values, it is also referred to as value distribution theory.

Afterwards, Cartan \cite{Cartan1933} extended Nevanlinna theory to linear non-degenerate holomorphic curves in (complex) projective space $\mathbb{P}^n$ intersecting hyperplanes in general position.
Notice that $\mathbb{P}^1 \cong \mathbb{C} \cup \{\infty\}$ and meromorphic functions can be regarded as holomorphic curves in $\mathbb{P}^1$.
Thus Cartan's work is a natural extension of Nevanlinna theory to higher dimensions.

The key lemma in establishing both Nevanlinna's and Cartan's results is known as (Nevanlinna's) Lemma on Logarithmic Derivative (LLD for short) for meromorphic functions (see also Chapter \uppercase\expandafter{\romannumeral1}, Section 6 of \cite{Lang1990} or Section 1.2 of \cite{Noguchi2014}).

\begin{lemma}[Nevanlinna's LLD]\cite{Nevanlinna1929}
Let $f \colon \mathbb{C} \to \mathbb{C} \bigcup \{\infty\}$ be a non-zero meromorphic function. Then
\[
m\left(r, \frac{f'}{f}\right) \leq O(\log r + \log T(r,f))
\]
holds for all $r \in \mathbb{R}_{\geq 0}$ outside a Borel subset $E$ of finite Lebesgue measure,
where
\[
m\left(r, \frac{f'}{f}\right) = \int_0^{2\pi} \!\log^{+} \left|\frac{f'}{f}(re^{i\theta})\right| \frac{\dif \theta}{2\pi}
,\]
and $\log^{+} t \bydef \log \max \{t, 1\}$ for any $t \in \mathbb{R}_{+}$.
Especially, if $f$ is a rational function, then
\[
m\left(r, \frac{f'}{f}\right) \leq O(1) \quad (r \to \infty)
.\]
\end{lemma}
Here $m\left(r, \dfrac{f'}{f}\right) = m_{\frac{f'}{f}}(r, \infty)$ is indeed the proximity function defined in Section \ref{ssc:NTnotation}, if we think of $\dfrac{f'}{f}$ as a holomorphic curve in $\mathbb{P}^1$ (more precisely, the analytic extension of $[f' : f]$) and $\infty$ as the divisor $\{[z_0 : z_1] \mid z_1 = 0\}$ in $\mathbb{P}^1$.

Following \cite{Weyl1938}, Ahlfors \cite{Ahlfors1941} generalized Cartan's work \cite{Cartan1933} to the case of intersecting linear subspaces in general position.
Thereupon Ahlfors' and the related works is called Weyl-Ahlfors theory, or Ahlfors-Weyl theory in some literature (see \cite{Weyl1943} and \cite{Stoll1983}).
In his proof, Ahlfors established an innovative estimate as follows, which is the key to Weyl-Ahlfors theory and a lot of papers on value distribution theory for holomorphic curves in projective manifolds.
\begin{lemma}[=Lemma \ref{lem:ALLD_1^1}, Ahlfors' LLD]
Let $H$ be a hyperplane of $\mathbb{P}^{n}$, and let $f \colon \mathbb{C} \to \mathbb{P}^n$ be a non-constant holomorphic curve such that $f(\mathbb{C}) \not\subset H$.
Then for any $0 < \varepsilon < 1$, one has
\[
\varepsilon \mathcal{N}\left( \frac{\phi_1(H)}{\phi_0(H)^{1-\varepsilon}} \Omega_0, r \right) \leq (1+\varepsilon) T_f(r, \mathcal{O}_{\mathbb{P}^{n}}(1)) + O(1)
,\]
where $\phi_k(H)$ is the $k$-th contact function of $f$ for $H$ and $\Omega_0 = f^{\ast} \omega_0$ is the pullback of the Kähler form $\omega_0$ associated to the Fubini-Study metric on $\mathbb{P}^{n}$.
\end{lemma}
Because this lemma plays a similar role as Nevanlinna's LLD in \cite{Nevanlinna1925}, we hereinafter prefer to call it Ahlfors' Lemma on Logarithmic Derivative (Ahlfors' LLD for short) for holomorphic curves in projective spaces.

Weyl-Ahlfors theory not only provides a geometric approach to the value distribution theory for holomorphic curves in projective spaces, but also inspires the development of value distribution theory for holomorphic curves with more general domains and targets.
For instance, Stoll \cite{Stoll1953, Stoll1954} developed Weyl-Ahlfors theory into the theory of meromorphic maps from parabolic spaces into projective spaces.

The significance of Weyl-Ahlfors theory extends far beyond its original scope, influencing various branches of mathematics, such as differential geometry, hyperbolic geometry, and number theory.
Chern investigated the value distribution of holomorphic curves from a differential geometry perspective and published several papers on this subject.
One of his works, joint with Bott \cite{Bott1965}, introduced a “connection” to value distribution theory and established a relation with the characteristic classes.

Inspired by Ahlfors' LLD, Carlson and Griffiths \cite{Carlson1972} obtained an analogous estimate for (non-degenerate) equidimensional holomorphic maps from $\mathbb{C}^n$ to projective manifolds $X$, and established value distribution theory for equidimensional holomorphic maps.
And Griffiths and King \cite{Griffiths1973} extended the theory to the case where $n \geq \dim X$.
Furthermore, this result in turn led to a SMT type conjecture for holomorphic curves in projective manifolds $X$.

\begin{conjecture}[Griffiths' Second Main Conjecture]\cite{Griffiths1972}
Let $X$ be a projective manifold, and let $D$ be an effective divisor on $X$ with possibly simple normal crossings singularities.
Let $f \colon \mathbb{C} \longrightarrow X$ be a algebraically non-degenerate holomorphic curve.
For arbitrary small $\varepsilon > 0$, the inequality
\[
m_f(r, D) + T_f(r, K_X) \leq \varepsilon T_f(r, A)
\]
holds for all $r \in \mathbb{R}_{\geq 0}$ outside a Borel subset $E$ of finite Lebesgue measure, where $K_X \bydef \det T_X^{\dual}$ is the canonical bundle of $X$ and $A$ is an ample divisor.
\end{conjecture}
Obviously, the well-known Green-Griffiths conjecture \cite{Green1980} is a special case of Griffiths' Second Main Conjecture.
\begin{remark}
The original statement of Griffiths' Second Main Conjecture is formulated in terms of Nevanlinna defect.
We reformulate it as a SMT type conjecture so that it is more convenient to compare with Lang's Second Main Conjecture.
This statement is a bit stronger than the original one, but it is still weaker than Lang's Second Main Conjecture.
\end{remark}

On the other hand, Nochka \cite{Nochka1983} introduced Nochka weights, and solved Cartan's conjecture on linear degenerate holomorphic curves in projective spaces, now called Cartan-Nochka theorem (cf. \cite{Chen1990} or Section 2.4 of \cite{Fujimoto1993} for a simplification of Nochka's proof).
Afterwards, using the trick of Nochka weights, many SMT type results for non-degenerate holomorphic curves or meromorphic maps were extended to the degenerate case.

Brody's theorem \cite{Brody1978} says that Brody hyperbolicity and Kobayashi hyperbolicity are equivalent for compact manifolds.
Therefore, value distribution theory for degenerate holomorphic curves offers the best approach to the hyperbolicity question.

In addition, Lang \cite{Lang1986} proposed a conjecture of a higher dimensional version of Mordell's conjecture (cf. \cite{Mordell1922}) on rational points, which is known as Lang's conjecture.
It states that there are at most finitely many rational points on a Kobayashi hyperbolic algebraic manifold over a number field.

What's more, Vojta \cite{Vojta1987} introduced the correspondence between Nevanlinna theory and Diophantine approximation in the case of one variable, and suggested this correspondence can be extended to several variables.
Hence Lang proposed a stronger version of SMT type conjecture, inspired by a counterpart in Diophantine approximation (see Conjecture 5.1 in Chapter \uppercase\expandafter{\romannumeral8} of \cite{Lang1991}).

\begin{conjecture}[Lang's Second Main Conjecture]
Let $X$ and $D$ be as in Griffiths' Second Main Conjecture.
There is a proper closed algebraic subset $Z \subsetneq X$ such that for any holomorphic curve $f \colon \mathbb{C} \longrightarrow X$ with $f(\mathbb{C}) \not\subset Z$, the inequality
\[
m_f (r, D) + N_{f, \textnormal{ram}}(r) + T_f (r, K_X) \leq O(\log r + \log^{+} T_f(r, A))
\]
holds for all $r \in \mathbb{R}_{\geq 0}$ outside a Borel subset $E$ of finite Lebesgue measure,
where $N_{f, \textnormal{ram}}(r) = \mathcal{N}([D_{f, \textnormal{ram}}], r)$, and $D_{f, \textnormal{ram}}$ is the ramification divisor of $f$.
\end{conjecture}
In this conjecture, holomorphic curve $f$ need not be algebraically non-degenerate.
Instead, it suffices that $f(\mathbb{C}) \not\subset Z$ for some proper algebraic subset $Z$ independent of $f$.
So far, these two conjectures are still open, except for some special cases.

\subsection{Value Distribution Theory over Directed Manifolds}

Ochiai \cite{Ochiai1977} revived Bloch's paper \cite{Bloch1926} and established the Bloch-Ochiai Theorem.
In his work, Ochiai proved a lemma on holomorphic differentials, which is an analogue to Nevanlinna's LLD in Nevanlinna theory.
Shortly afterwards, Noguchi \cite{Noguchi1977} discovered its relation to Deligne's logarithmic forms \cite{Deligne1971, Deligne1974}, and established Lemma on Logarithmic Differentials for holomorphic curves in compact Kähler manifolds.
Noguchi's approach unifies the Bloch-Ochiai Theorem and the classical theorem of Borel.
Moreover, Noguchi \cite{Noguchi1985} extended Lemma on Logarithmic Differentials to the case of meromorphic maps in projective manifolds.

On the other hand, Kobayashi \cite{Kobayashi1996} proposed another geometric formulation of Nevanlinna's LLD by Radon transform and SMT for one variable (i.e., Nevanlinna's original version).
Inspired Kobayashi's proposal, Yamanoi \cite{Yamanoi2004} established a geometric formulation of Nevanlinna's LLD for meromorphic maps, using Lemma on Logarithmic Differentials for meromorphic maps.
And he called it Algebro-Geometric Version of Nevanlinna's Lemma on Logarithmic Derivative and took the initials ANLD for short.
However, Yamanoi's result does not completely coincide with the proposal in \cite{Kobayashi1996}.
Kobayashi considered the theorems over Demailly-Semple jet towers over directed manifolds $(X, T_X)$, while Yamanoi constructed the compactiﬁcation of jet space by defining each fiber at $x$ as $\bar{J}_kX_x \bydef \textnormal{P}(J_kX_x \times \mathbb{C})$.
Recently, Brotbek \cite{Brotbek2017} constructed Wronskians on Demailly-Semple jet towers, and gave a simple proof of Kobayashi conjecture (cf. \cite{Kobayashi1970, Kobayashi1998}) for sufficiently large degree of the hypersurfaces.
Brotbek's approach is much more simple than Siu's proof in \cite{Siu2015}, and Wronskians on Demailly-Semple jet towers provides the last missing piece of the puzzle for the implication between Kobayashi's proposal and Cartan's (or Ahlfors') SMT.

More generally, we propose that either Griffiths's or Lang's Second Main Conjecture can be extended to the category of directed manifolds.

\begin{conjecture}
Let $(X, V)$ be a directed projective manifold, and let $D$ be an effective divisor  $X$ with possibly simple normal crossings singularities.
For any algebraically non-degenerate holomorphic curve $f \colon (\mathbb{C}, T_\mathbb{C}) \to (X, V)$, we have
\[
m_f(r, D) + T_f(r, K_V) \leq S_f(r)
,\]
where $K_V \bydef \det V^{\dual}$ and $S_f(r)$ is given by Definition \ref{def:Sf}.
\end{conjecture}

\begin{conjecture}
Let $(X, V)$ be a directed projective manifold, and let $D$ be an effective divisor  $X$ with possibly simple normal crossings singularities.
There is a closed algebraic subset $Z \subsetneq X$ such that for any holomorphic curve $f \colon (\mathbb{C}, T_\mathbb{C}) \to (X, V)$ with $f(\mathbb{C}) \not\subset Z$, we have
\[
m_f (r, D) + N_{f, \textnormal{ram}}(r) + T_f (r, K_V) \leq S_f(r)
.\]

\end{conjecture}

\subsection{Structure}
The paper is organized as follows.

Section \ref{sec:DV} gives a brief introduction to the construction of directed varieties, the morphisms between directed varieties, and the blow-up of directed varieties.
In Section \ref{sec:jets}, we introduce the concept of the (classic) jet spaces and Demailly-Semple jet towers, which are essential for understanding the higher-order derivatives of holomorphic curves.
And then we recall the Green-Griffiths jet differentials and invariant jet differentials, the latter one has better properties and will play an important role in the whole paper.
Section \ref{sec:Wronskians} introduces the Wronskians sections on Demailly-Semple jet towers, and generalizes the Wronskians ideal sheaves with respect to the linear systems.

In Section \ref{sec:logjets}, we give a brief introduction to logarithmic Demailly-Semple jet towers as a generalization of Demailly-Semple jet towers to the logarithmic case, which is introduced by Dethloff and Lu \cite{Dethloff2001}.
Section \ref{sec:logWronskians} constructs the logarithmic Wronskian ideal sheaves on logarithmic Demailly-Semple jet towers as an analogue of the Wronskian ideal sheaves in Section \ref{sec:Wronskians}.
Furthermore, we show that the the logarithmic Wronskian ideal sheaves corresponds to a subsheaves of the Wronskian ideal sheaves, which helps us to understand the jets of closed subschemes over directed manifolds.

In Section \ref{sec:NT}, we introduce the necessary preliminaries from Nevanlinna theory, including Weil functions, the first main theorem.
And we delve Crofton's formula for Weil functions and proximity functions, and slightly improve the typical application of Calculus Lemma.

In Section \ref{sec:ALLD}, we recall derived curves (or associated curves) of holomorphic curves in projective spaces, and present Ahlfors' LLD.
Then using the maps to the Grassmannian, we indicate a way to generalize Ahlfors' LLD to the case of holomorphic tangent curves of directed projective manifolds.
We also define the jets of closed subschemes over directed manifolds, and obtain Ahlfors' Lemma on Logarithmic Derivative over Directed Manifolds (ALLDD for short), which is one of the central result of this paper.
In addition, we prove General form of Ahlfors' Lemma on Logarithmic Derivative over Directed Manifolds (GALLDD for short) as a generalization of ALLDD to the case of several closed subschemes, from which we can obtain SMT type results for holomorphic curves more conveniently.

In Section \ref{sec:AALD}, we prove Algebro-Geometric Version of Ahlfors' Lemma on Logarithmic Derivative (AALD for short) and General form of Algebro-Geometric Version of Ahlfors' Lemma on Logarithmic Derivative (GAALD for short) for holomorphic tangent curves of directed projective manifolds.
What's more, we give various corollaries of AALD and GAALD.
Finally, Section \ref{sec:Applications} shows some SMT type results as the applications of GAALD.
For more details on the implication between these lemmas and theorems, we refer to the diagram below.

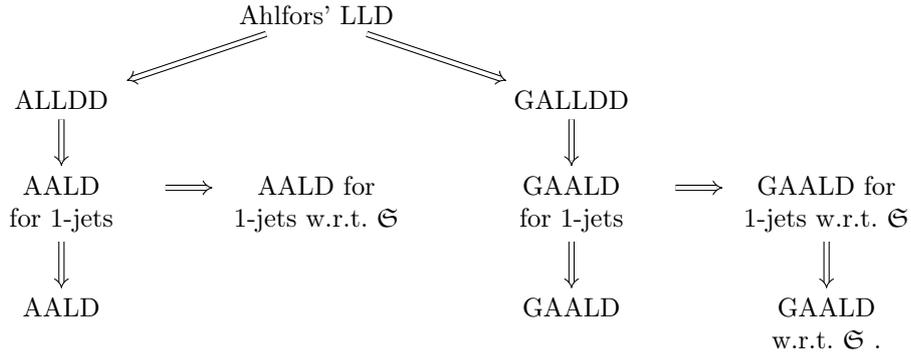
\begin{equation*}
\begin{tikzcd}[column sep=scriptsize, math mode=false, cells={text width=16ex, align=center}]
& Ahlfors' LLD \arrow[ld, Rightarrow] \arrow[rd, Rightarrow] & & \\%\arrow[rrddd, Rightarrow, bend left]
ALLDD \arrow[d, Rightarrow] & & GALLDD \arrow[d, Rightarrow] & \\
AALD for $1$-jets \arrow[r, Rightarrow] \arrow[d, Rightarrow] & AALD for $1$-jets w.r.t. \!$\mathfrak{S}$ & GAALD for $1$-jets \arrow[r, Rightarrow] \arrow[d, Rightarrow] & GAALD for $1$-jets w.r.t. \!$\mathfrak{S}$ \arrow[d, Rightarrow] \\
AALD & & GAALD & GAALD w.r.t. \!$\mathfrak{S}$
.\end{tikzcd}
\end{equation*}

Using the language of Weil functions for closed subschemes of Demailly-Semple jet towers, we generalize and modernize Ahlfors' LLD for $1$-jets as follows.
\begin{lemma}[= Lemma \ref{lem:ALLDZ^(1)}, ALLDD]
Let $(X, V)$ be a directed projective manifold, and let $Z$ be a closed subscheme of $X$.
Let $f \colon (\mathbb{C}, T_{\mathbb{C}}) \to (X, V)$ be a non-constant holomorphic curve such that $f(\mathbb{C}) \not\subset \Supp Z$.
Choose a linear system $\mathfrak{S} \subset |L|$ on $X$ satisfying that $\mathfrak{S}$ separates $1$-jets at every point of $X$ and $\mathcal{I}_Z = \mathfrak{s}_{\supset Z}$.
Then there exists a constant $c > 0$ such that the inequality
\[
\varepsilon \mathcal{N}\left(\exp\left(2(1-\varepsilon)\lambda_Z \circ f - 2\lambda_{Z^{(1)}} \circ f_{[1]}\right) f^{\ast} \ddif^c \varphi_{\mathfrak S}, r\right) \leq c T_f(r, L) + O(1)
\]
holds for any $0 < \varepsilon < 1$, where $\mathcal{N}(\Phi, r)$ is defined as in Definition \ref{def:currentcount} for $\Phi \in \mathcal{D}'^{(1, 1)}(\mathbb{C})$.
\end{lemma}
For real-valued functions $g, h \in L^2(U)$ on a subset $U \subset \mathbb{C}$, we define the ``inner product'' of $g$ and $h$ with respect to the measure $\gamma_{\mathfrak S} \ddif^c |z|^2$ as $\langle g, h \rangle_{\gamma_{\mathfrak S}, U} = \int_{U} gh \gamma_{\mathfrak S} \ddif^c |z|^2$.
Then $\mathcal{N}\left( gh f^{\ast} \ddif^c \varphi_{\mathfrak S}, r\right) = \int_{0}^{r} \langle g, h \rangle_{\gamma_{\mathfrak S}, \bigtriangleup(t)} \frac{\dif t}{t}$ can be regarded as the growth of the ``inner product'' of $g$ and $h$ with respect to the measure $\gamma_{\mathfrak S} \ddif^c |z|^2 (= f^{\ast} \ddif^c \varphi_{\mathfrak S})$.
Dividing $\varepsilon$ on both sides of the inequality in ALLDD, we see that the left-hand side of the inequality above can be interpreted as the growth of the square of the ``$L^2$-norm'' of $\exp\left((1-\varepsilon)\lambda_Z \circ f - \lambda_{Z^{(1)}} \circ f_{[1]}\right)$ with respect to the measure $\gamma_{\mathfrak S} \ddif^c |z|^2$,
and the right-hand side can be interpreted as the growth of the square of the ``$L^2$-norm'' of the constant $\left(\dfrac{c}{\varepsilon}\right)^{\frac{1}{2}}$ with respect to the same measure.
Observe that $\exp\left(\lambda_Z \circ f - \lambda_{Z^{(1)}} \circ f_{[1]}\right)$ takes the poles only at the points where $f$ comes close to $Z$ but $f_{[1]}$ does not approach $Z^{(1)}$.

Therefore, ALLDD gives a somewhat upper bound for the growth of the number of points (with multiplicity) where $f$ comes close to $Z$ but $f_{[1]}$ does not approach $Z^{(1)}$ in the disk $\bigtriangleup(r)$.
More concretely, $\exp\left((1-\varepsilon)\lambda_Z \circ f - \lambda_{Z^{(1)}} \circ f_{[1]}\right)$ is an approximating evaluation of the exceptional cases when $f$ comes close to $Z$ but $f_{[1]}$ does not approach $Z^{(1)}$, and the growth of its ``$L^2$-norm'' with respect to the measure $\gamma_{\mathfrak S} \ddif^c |z|^2$ is controlled by the growth of the ``$L^2$-norm'' of the constant.

Moreover, in order to get a better estimation in the case where $f$ intersects the closed subschemes with singularities, we generalize Ahlfors' Lemma on Logarithmic Derivative over Directed Manifolds to the case of several closed subschemes.
\begin{lemma}[= Lemma \ref{lem:GALLDD}, GALLDD]
Let $(X, V)$ be a directed projective manifold, and let $Z_1, \dots, Z_q$ be closed subschemes of $X$. Let $f \colon (\mathbb{C}, T_{\mathbb{C}}) \to (X, V)$ be a non-constant holomorphic curve such that $f(\mathbb{C}) \not\subset \Supp (Z_1 + \dots + Z_q)$. For any linear system $\mathfrak{S} \subset |L|$ on $X$ satisfying that $\mathfrak{S}$ separates $1$-jets at every point of $X$ and $\mathcal{I}_Z = \mathfrak{s}_{\supset Z}$, there exists a constant $c > 0$ such that the inequality
\[
\varepsilon \mathcal{N}\left(\exp\left(2(1-\varepsilon)\mathop{\max}_{1 \leq i \leq q} \lambda_{Z_i} \circ f - 2\mathop{\max}_{1 \leq i \leq q} \lambda_{Z_i^{(1)}} \circ f_{[1]}\right) f^{\ast} \ddif^c \varphi_{\mathfrak S}, r\right) \leq c T_f(r, L) + O(1)
\]
holds for any $0 < \varepsilon < 1$.
\end{lemma}

Taking logarithms on both sides of the inequality in ALLDD, and using the Green-Jensen formula and Calculus Lemma, we conclude AALD for $1$-jets.
\begin{theorem}[= Theorem \ref{thm:Z^(1)}, AALD for $1$-jets]
Let $(X, V)$ be a directed projective manifold, and let $Z$ be a closed subscheme of $X$.
Let $f \colon (\mathbb{C}, T_{\mathbb{C}}) \to (X, V)$ be a non-constant holomorphic curve such that $f(\mathbb{C}) \not\subset \Supp Z$.
Then we have
\[
m_f(r, Z) + T_{f_{[1]}}(r, \mathcal{O}_{X_1}(1)) + \mathcal{N}([Z_{f'}], r) \leq m_{f_{[1]}}(r, Z^{(1)}) + S_f(r)
.\]
\end{theorem}
This implies that not only the difference between the proximity functions for $Z$ and $Z^{(1)}$ but also the sum of the characteristic function for $\mathcal{O}_{X_1}(1)$ and the counting function for $[Z_{f'}]$ is controlled by the growth of $S_f(r)$.
The controlling of the characteristic function for $\mathcal{O}_{X_1}(1)$ is essential to deal with the hyperbolicity questions.

Viewing $(X_k, V_k)$ as the Demailly-Semple $1$-jet tower of $(X_{k-1}, V_{k-1})$, we immediately get AALD by induction.
\begin{theorem}[= Theorem \ref{thm:Z^(k)}, AALD]
Let $(X, V)$ be a directed projective manifold, and let $Z$ be a closed subscheme of $X$.
Let $f \colon (\mathbb{C}, T_{\mathbb{C}}) \to (X, V)$ be a non-constant holomorphic curve such that $f(\mathbb{C}) \not\subset \Supp Z$.
For each $k \in \mathbb{Z}_{+}$, we have
\[
m_{f_{[k-1]}}(r, Z^{(k-1)}) + T_{f_{[k]}}(r, \mathcal{O}_{X_{k}}(1)) + \mathcal{N}([Z_{f_{[k-1]}'}], r) \leq m_{f_{[k]}}(r, Z^{(k)}) + S_f(r)
.\]
\end{theorem}

Furthermore, in order to prove SMT type results more conveniently, we need the following transform of AALD.
\begin{theorem}[= Theorem \ref{thm:Z_S^1}, AALD for $1$-jets w.r.t. \!$\mathfrak{S}$]
Let $(X, V)$ be a directed projective manifold, and let $Z$ be a closed subscheme of $X$.
Choose a linear system $\mathfrak{S} \subset |L|$ on $X$ satisfying that $\mathcal{I}_Z = \mathfrak{s}_{\supset Z}$.
Let $f \colon (\mathbb{C}, T_{\mathbb{C}}) \to (X, V)$ be a non-constant holomorphic curve such that $f_{[1]}(\mathbb{C}) \not\subset \Bs(\mathfrak{W}(X_1, \mathfrak{S}))$.
Suppose that $f(\mathbb{C}) \not\subset \Supp Z$.
Then,
\[
m_f(r, Z) + T_{f_{[1]}}(r, \mathcal{O}_{X_1}(1)) + \mathcal{N}([Z_{f'}], r) \leq m_{f_{[1]}}(r, Z_{\mathfrak S}^{[1]}) - m_f(r, Z_{\mathfrak s}) + S_f(r)
,\]
where $Z_{\mathfrak{S}}^{[k]}$ is the closed subscheme of $X_k$ with the ideal sheaf $\widebar{\mathfrak w}(X_k, Z, \mathfrak{S})$.
\end{theorem}

Since $Z_1^{(1)} + Z_2^{(1)} \subsetneqq (Z_1 + Z_2)^{(1)}$ and $(Z_1)_{\mathfrak S}^{(1)} + (Z_2)_{\mathfrak S}^{(1)} \subsetneqq (Z_1 + Z_2)_{\mathfrak S}^{(1)}$ in general, we refine AALD for $1$-jets and obtain the following two theorems in order to deal with SMT type questions better, especially in the case of intersecting closed subscheme with singularities.
\begin{theorem}[= Theorem \ref{thm:Z^(1)veesum}, GAALD for $1$-jets]
Let $(X, V)$ be a directed projective manifold, and let $Z_1, \dots, Z_q$ be closed subschemes of $X$.
Take an integer $1 \leq \ell \leq q$.
Let $f \colon (\mathbb{C}, T_{\mathbb{C}}) \to (X, V)$ be a non-constant holomorphic curve such that $f(\mathbb{C}) \not\subset \Supp (Z_1 + \dots + Z_q)$. Then,
\[
m_f(r, \bigvee_{\substack{I \subset J \\ |I| = \ell}} \sum_{i \in I} Z_i) + T_{f_{[1]}}(r, \mathcal{O}_{X_1}(\ell)) + \ell \mathcal{N}([Z_{f'}], r)
\leq m_{f_{[1]}}(r, \bigvee_{\substack{I \subset J \\ |I| = \ell}} \sum_{i \in I} Z_i^{(1)}) + S_f(r)
,\]
where $\displaystyle m_f(r, \bigvee_{\substack{I \subset J \\ |I| = \ell}} \sum_{i \in I} Z_i)$ and $\displaystyle m_{f_{[1]}}(r, \bigvee_{\substack{I \subset J \\ |I| = \ell}} \sum_{i \in I} Z_i^{(1)})$ are defined as \eqref{eqn:proximityvee}.
\end{theorem}

\begin{theorem}[= Theorem \ref{thm:Z_S^1veesum}, GAALD for $1$-jets w.r.t. \!$\mathfrak{S}$]
Let $(X, V)$ be a directed projective manifold, and let $Z_1, \dots, Z_q$ be closed subschemes of $X$.
Choose a linear system $\mathfrak{S}$ on $X$ satisfying that $\mathcal{I}_{Z_i} = \mathfrak{s}_{\supset Z_i}$ for each $1 \leq i \leq q$.
Take an integer $1 \leq \ell \leq q$.
Let $f \colon (\mathbb{C}, T_{\mathbb{C}}) \to (X, V)$ be a non-constant holomorphic curve such that $f_{[1]}(\mathbb{C}) \not\subset \Bs(\mathfrak{W}(X_1, \mathfrak{S}))$.
Suppose that $f(\mathbb{C}) \not\subset \Supp (Z_1 + \dots + Z_q)$. Then we have
\begin{align*}
&m_f(r, \!\bigvee_{\substack{I \subset J \\ |I| = \ell}} \!\sum_{i \in I} Z_i) + T_{f_{[1]}}(r, \mathcal{O}_{X_1}(\ell)) + \ell \mathcal{N}([Z_{f'}], r) \\
\leq \, &m_{f_{[1]}}\big(r, \!\bigvee_{\substack{I \subset J \\ |I| = \ell}} \!\sum_{i \in I} (Z_i)_{\mathfrak S}^{[1]}\big) - \ell m_f(r, Z_{\mathfrak s}) + S_f(r)
.\end{align*}
\end{theorem}

Finally, we obtain GAALD and GAALD w.r.t. \!$\mathfrak{S}$ for high jets and the combination of closed subschemes.
\begin{theorem}[= Theorem \ref{thm:Z^(k)sum}, GAALD]
Let $(X, V)$ be a directed projective manifold, and let $Z_1, \dots, Z_q$ be closed subschemes of $X$. Set
\[
\ell_k \bydef \max \{|I| \mid I \subset J, \, \bigcap_{i \in  I} Z_i^{(k)} \neq \emptyset\} 
.\]
For each $k \in \mathbb{Z}_{+}$, let $f \colon (\mathbb{C}, T_{\mathbb{C}}) \to (X, V)$ be a non-constant holomorphic curve such that $f(\mathbb{C}) \not\subset \Supp (Z_1 + \dots + Z_q)$.
Then,
\[
\sum_{i =1}^q m_f(r, Z_i) + T_{f_{[k]}}(r, \mathcal{O}_{X_k}(\bm{\ell})) + \sum_{j=0}^{k-1} \ell_j \mathcal{N}([Z_{f_{[j]}'}], r)
\leq m_{f_{[k]}}(r, \bigvee_{\substack{I \subset J \\ |I| = \ell_k}} \sum_{i \in I} Z_i^{(k)}) + S_f(r)
,\]
where $\bm{\ell} = (\ell_0, \dots, \ell_{k-1}) \in \mathbb{Z}_{+}^k$.
\end{theorem}

\begin{theorem}[= Theorem \ref{thm:Z_S^kveesum}, GAALD w.r.t. \!$\mathfrak{S}$]
Let $(X, V)$ be a directed projective manifold, and let $Z_1, \dots, Z_q$ be closed subschemes of $X$.
Choose a linear system $\mathfrak{S}$ on $X$ satisfying that $\mathcal{I}_{Z_i} = \mathfrak{s}_{\supset Z_i}$ for each $1 \leq i \leq q$. Set
\[
\ell_{k, \mathfrak{S}} \bydef \max \{ |I| \mid I \subset J, \, \bigcap_{i \in I} (Z_i)_{\mathfrak S}^{[k]} \supsetneqq Z_{\widebar{\mathfrak w}(X_k, \mathfrak{S})} \} 
.\]
Let $k_{\mathfrak S}$ be the minimum of $k$ such that $\ell_{k, \mathfrak{S}} = 0$.
Let $f \colon (\mathbb{C}, T_{\mathbb{C}}) \to (X, V)$ be a non-constant holomorphic curve such that $f_{[k_{\mathfrak S}]}(\mathbb{C}) \not\subset \Bs(\mathfrak{W}(X_{k_{\mathfrak S}}, \mathfrak{S}))$.
Suppose that $f(\mathbb{C}) \not\subset \Supp (Z_1 + \dots + Z_q)$.
Then,
\begin{align*}
&\sum_{i =1}^q m_f(r, Z_i) + T_{f_{[k_{\mathfrak S}]}}(r, \mathcal{O}_{X_{k_{\mathfrak S}}}(\bm{\ell}_{\mathfrak S})) + \sum_{j=0}^{k_{\mathfrak S}-1} \ell_{j, \mathfrak{S}} \mathcal{N}([Z_{f_{[j]}'}], r) \\
\leq \, & \sum_{j=0}^{k_{\mathfrak S}} (\ell_{j-1, \mathfrak{S}} - 2\ell_{j, \mathfrak{S}} + \ell_{j+1, \mathfrak{S}}) m_{f_{[j]}}(r, Z_{\widebar{\mathfrak w}(X_{j}, \mathfrak{S})}) + S_f(r)
,\end{align*}
where $\bm{\ell}_{\mathfrak S} = (\ell_{0, \mathfrak{S}}, \dots, \ell_{k_{\mathfrak S}-1, \mathfrak{S}}) \in \mathbb{Z}_{+}^{k_{\mathfrak S}}$ and we make the conventions that $\ell_{-1, \mathfrak{S}} = q$, and $\widebar{\mathfrak w}(X_0, \mathfrak{S}) = \mathfrak{s}$.
\end{theorem}

\section{Directed varieties and Blow-up}\label{sec:DV}

\subsection{Directed varieties}

Let us consider a pair $(X, V)$ consisting of a complex manifold $X$ of dimension $n$ equipped with a linear subspace $V \subset T_X$.
More precisely, if $X$ is connected, the total space of $V$ is by definition an irreducible closed analytic subspace of the total space of $T_X$ such that each fiber $V_x = V \cap T_{X,x}$ is a vector subspace of $T_{X,x}$.
If $X$ is not connected, $V$ should be irreducible merely over each connected component of $X$, but we will typically assume that the varieties are connected.

Let $\mathcal{W} \subset \mathcal{T}_X^{\dual}$ be the sheaf of $1$-forms vanishing on $V$.
Notice that $\mathcal{W}$ is indeed the direct image sheaf $(\pi_1)_{\ast} \mathcal{O}_{\textnormal{P}(V)}(1)$, where $\pi_1 \colon \textnormal{P}(V) \subset \textnormal{P}(T_X) \hookrightarrow X$ is a proper morphism.
It follows from the direct image theorem (cf. \cite{Grauert1960}) that $\mathcal{W}$ is coherent.
Take a small open subset $U \subset X$, and choose some local sections $\alpha_1, \dots, \alpha_q \in \mathcal{W}(U)$ generating $\mathcal{W}$ at every point of $U$.
Then $V$ is locally defined by
\vspace{-.5ex}
\[
V_x = \{\xi \in T_{X,x} \mid \alpha_j(x)\cdot\xi=0,\, \forall 1 \leq j \leq q\}
,\vspace{-.5ex}\]
for every point $x \in U$.
We can associate to $V$ a coherent sheaf $\mathcal{V} \bydef \mathcal{W}^{\perp} = \mathcal{H}om_{\mathcal{O}_X}(\mathcal{T}_X^{\dual}/\mathcal{W}, \mathcal{O}_X) \subset \mathcal{T}_X$,
which is a saturated subsheaf of $\mathcal{T}_X$, i.e., such that $\mathcal{T}_X/\mathcal{V}$ has no torsion.
Consequently, $\mathcal{V}$ is also reflexive, namely, $\mathcal{V}^{\dual\dual}=\mathcal{V}$.
We refer to such a pair as a (complex) directed variety.
In general, we will think of $V$ as a linear space, rather than considering the associated saturated subsheaf $\mathcal{V} \subset \mathcal{T}_X$.
A morphism $\varphi \colon (X, V) \to (Y, W)$ in the category of complex directed varieties is a holomorphic map $\varphi \colon X \to Y$ such that $\dif \varphi(V) \subset W$, where $\dif \varphi$ is the differential of $\varphi$.

The rank at each point $x \in X$ is given by a function $x \mapsto \dim_\mathbb{C} V_x$, which is Zariski lower semi-continuous and may a priori jump.
The rank $r \bydef \rank V$ of $V$ is defined as the dimension of $V_x$ at a generic point.
The dimension may be larger at non-generic points; this happens e.g., on $X=\mathbb{C}^n$ for the rank $1$ linear space $V$ generated by the Euler vector field: $V_z = \mathbb{C}\sum_{1 \leq j\le n} z_j\frac{\partial}{\partial z_j}$ for $z \neq 0$, and $V_0 = \mathbb{C}^n$.

We believe that directed varieties are useful even when studying the "absolute case" (where $V=T_X$).
By \eqref{eqn:functorial}, any natural embedding $(X, V) \hookrightarrow (X, T_X)$ induces a certain functorial constructions in the category of directed varieties.
We think of directed varieties as a kind of ``relative situation'', covering, e.g., the case when $V$ is the relative tangent space to a holomorphic map $X \to S$.
It is important to notice that the local sections of $\mathcal{V}$ need not generate the fibers of $V$ at singular points, as one sees already in the case of the Euler vector field when $n\ge 2$.
We also want to emphasize that no assumption need be made about the Lie bracket tensor $[\bullet, \bullet] \colon \mathcal{V} \times \mathcal{V}\to \mathcal{T}_X/\mathcal{V}$,
which amounts to saying that we do not assume any kind of integrability for $\mathcal{V}$ or $\mathcal{W}$.

The singular set $\Sing(V)$ is defined by the set of points where $\mathcal{V}$ is not locally free.
It can also be characterized as the indeterminacy set of the (meromorphic) classifying map to the Grassmannian of $r$-dimensional subspaces of $T_X$ as follows:
\begin{equation}\label{eqn:alpha}
\begin{aligned}
\alpha \colon X &\dashrightarrow \Gr(r, T_X) \\
z &\longmapsto V_z
.\end{aligned}
\end{equation}
We thus have $V_{\restriction X \setminus \Sing(V)} = \alpha^{\ast} S$, where $S \to \Gr(r, T_X)$ is the tautological subbundle of $\Gr(r, T_X)$.
The singular set $\Sing(V)$ is an analytic subset of $X$ of codimension at least $2$, and hence $V$ is always a holomorphic subbundle outside an analytic subset of codimension at least $2$.
Thanks to this remark, one can most often treat linear spaces as vector bundles (possibly modulo passing to the Zariski closure along $\Sing(V)$).

In addition, we call a (complex) directed variety $(X, V)$ a (complex) directed manifold, if $V$ is a holomorphic subbundle of $T_X$.
And we say that a directed manifold $(X, V)$ is a directed projective manifold, in the case where $X$ is a complex projective manifold.

\subsection{Blow-up of directed varieties}

For a complex variety $X$ and an ideal sheaf $\mathcal{I} \subset \mathcal{O}_X$, we have the blow-up of $X$ with respect to $\mathcal{I}$
\[
\upsilon \colon \widehat{X} = \Bl_{\mathcal{I}}(X) \longrightarrow X
,\]
where $\Bl_{\mathcal{I}}(X) \bydef \Proj \bigoplus_{d \geq 0} \mathcal{I}^d$ (see Chapter \uppercase\expandafter{\romannumeral2}, Section 7 of Hartshorne's book \cite{Hartshorne1977} for more details).
If $Z_{\mathcal I}$ is the closed subscheme of $X$ with the ideal sheaf $\mathcal{I}$, then we also call $\upsilon \colon \widehat{X} \to X$ the blow-up of $X$ with center $Z_{\mathcal I}$.
Moreover, we see that $\upsilon^{-1} \mathcal{I} \cdot \mathcal{O}_{\widehat{X}}$ associates to an effective Cartier divisor $E$ on $\widehat{X}$, called the exceptional divisor of $\upsilon$.
It amounts to saying that
\[
\mathcal{O}_{\widehat{X}}(-E) = \upsilon^{-1} \mathcal{I} \cdot \mathcal{O}_{\widehat{X}}
.\]
Apparently, $\upsilon_{\restriction U} \colon U \to X$ is an open embedding, where $U = \widehat{X} \setminus E$.

Now, we are going to generalize the blow-ups to complex directed varieties.
Let $(X, V)$ be a complex directed variety with $V$ being possibly singular, and let $\upsilon$ be the blow-up of $X$ with respect to $\mathcal{I}$ as above. We thus get a directed variety $(\widehat{X}, \widehat{V})$ by taking $\widehat{V}$ to be the Zariski closure of $(\dif \upsilon)^{-1}(V')$,
where $V' = V_{\restriction X'}$ is the restriction of $V$ over a Zariski open subset $X' \subset X \setminus (\Supp(\mathcal{O}_X / \mathcal{I}) \cup \Sing(V))$.
This construction is independent of the choice of $X'$, thanks to the Zariski density of $X'$.
Therefore, we say the morphism of directed varieties
\[
\upsilon \colon (\widehat{X}, \widehat{V}) \longrightarrow (X, V)
\]
is the blow-up of $(X, V)$ with respect to $\mathcal{I}$.

Let $f \colon \bigtriangleup(R) \to X$ be a (holomorphic) tangent curve of the directed variety $(X, V)$, in other words, $f'(z)\in V_{f(z)}$ for every $z \in \bigtriangleup(R)$,
where $\bigtriangleup(R)$ is the disk of center $0$ and radius $R$ in the complex plane $\mathbb{C}$.
Suppose that $f$ is non-constant and $f(\bigtriangleup(R)) \not\subset \Supp(\mathcal{O}_X / \mathcal{I}) \cup \Sing(V)$.
Obviously, there exists a unique curve $\hat{f} \colon \bigtriangleup(R) \to \widehat{X}$ such that its graph is the same to the (Euclidean) closure of the graph of $\upsilon_{\restriction U}^{-1}(f)$, where $U = \widehat{X} \setminus E$.
Observe that the following diagram
\begin{equation*}
\begin{tikzcd}
& \widehat{X} \arrow[d, "v"] \\
\Delta(R) \arrow[ur, dashed, "\hat{f}"] \arrow[r, "f"] & X
\end{tikzcd}
\end{equation*}
commutes, and $\hat{f}$ is a non-constant tangent curve of $(\widehat{X}, \widehat{V})$.
We thus call $\hat{f}$ the lifting of $f$ to $(\widehat{X}, \widehat{V})$.

Moreover, we claim that $\Sing(\widehat{V}) \subset \upsilon^{-1}(\Sing(V))$.
Actually, the composition $\alpha \circ \upsilon$ gives a rational map $\widehat{\alpha} \colon \widehat{X} \dashrightarrow \Gr(r, T_X)$, where $\alpha \colon X \dashrightarrow \Gr(r, T_X)$ is defined by \eqref{eqn:alpha}.
We easily check that $\widehat{\alpha}$ coincides with the meromorphic extension of $\alpha \circ \upsilon_{\restriction U'}$.
Hence $\Sing(\widehat{V}) = \Ind(\widehat{\alpha}) \subset \upsilon^{-1}(\Ind(\alpha)) = \upsilon^{-1}(\Sing(V))$.

\begin{remark}
In Section 2.C of \cite{Demailly2020}, Demailly also introduced the blow-up in the category of directed varieties.
Notice that $\upsilon$ restricted on $\upsilon^{-1}(X')$ is a (holomorphic) isomorphism, namely, biholomorphism for any Zariski open subset $X' \subset X \setminus (\Supp(\mathcal{O}_X / \mathcal{I}) \cup \Sing(V))$.
Therefore, our construction of the blow-up is identified with the construction in \cite{Demailly2020}.
\end{remark}

\section{Jets and jet differentials}\label{sec:jets}

\subsection{Jet spaces}
Let us first recall the classic jet space.

Let $X$ be a complex manifold of dimension $n$.
Given a point $x \in X$, a germ of holomorphic curves $f \colon (\mathbb{C}, 0) \to (X, x)$ is an equivalence class of pairs $(U, f_U)$,
where $U \subset \mathbb{C}$ is an open neighborhood of the origin $0$, $f_U$ is a holomorphic curve from $U$ to $X$ such that $f_U(0)=x$, and where $(U_1, f_{U_1})$ and $(U_2, f_{U_2})$ are equivalent if $f_{U_1}$ and $f_{U_2}$ agree on $U_1 \cap U_2$.
Moreover, for the germs of holomorphic curves $f, g \colon (\mathbb{C}, 0) \to (X, x)$,
one defines the equivalence relation $f \overset{k}{\sim} g$ if and only if $f^{(\ell)}(0)= g^{(\ell)}(0)$ for all $0 \leq \ell \leq k$,
where the derivatives are computed in some (equivalently, any) local coordinate system near $x$.
The equivalence class of $f$ is denoted by $j_kf$.
Set
\begin{align*}
J_kX_x &\bydef \{j_kf \mid f: (\mathbb{C}, 0)\to (X, x)\}, \\
J_kX &\bydef \bigcup_{x \in X} J_kX_x
.\end{align*}
One immediately gets a natural projection
\[
p_k \colon J_kX \longrightarrow X
\]
defined by $p_k(j_kf)= f(0)$.
Apparently, $J_kX$ is a complex manifold, and $p_k$ gives a $\mathbb{C}^{nk}$-fiber bundle over $X$, which is called $k$-jet bundle over $X$.
In the case when $k=1$, $J_1X = T_X$ is holomorphic tangent bundle over $X$.
However, for $k \geq 2$, there is no such linear structure in $k$-jet bundle.

Furthermore, following Definition 5.1 in \cite{Demailly1997}, one has the jet spaces over a complex directed manifold $(X, V)$.
Let $(X, V)$ be a complex directed manifold such that $V$ is a holomorphic subbundle of $T_X$.
One denotes by $J_kV$ the $k$-jet bundle of $X$ tangent to $V$, consisting of all $k$-jets $j_kf$ such that $f'(z) \in V_{f(z)}$ for all $z$ in some neighborhood of $0$ in the complex plane.
Apparently, $J_kV$ is a subbundle of $J_kX$. For shortness, one hereinafter writes
\[
p_{k, V} \bydef p_{k \restriction J_kV} \colon J_kV \longrightarrow X
.\]
The following construction indicates that $J_kV$ is indeed a $\mathbb{C}^{rk}$-fiber bundle over $X$ under this projection $p_{k, V}$, where $r=\rank V$.

Given any open subset $U \subset X$, for each $0 \leq \ell \leq k-1$, one can define a $\mathbb{C}$-linear map
\begin{equation}\label{eqn:difform}
\dif_{U, V}^{[\ell]} \colon \mathcal{V}^{\dual}(U) \longrightarrow \mathcal{O}(p_{k, V}^{-1}(U))
\end{equation}
by $\dif_{U, V}^{[\ell]}\omega(j_kf) = A^{(\ell)}(0)$ for every holomorphic 1-form $\omega \in \mathcal{V}^{\dual}(U)$ and $j_kf \in p_{k, V}^{-1}(U) \subset J_kV$,
where $A$ is a germ of functions on $(\mathbb{C}, 0)$ such that $f^{\ast}\omega = A(z) \dif z$.
Obviously, for any $j_kf \in J_kV$, one has $\dif_{U, V}^{[\ell]} (\res_{V} \omega)(j_kf) = \dif_{U, T_X}^{[\ell]} \omega(j_kf)$,
where $\res_V \colon \mathcal{T}_X^{\dual} \to \mathcal{V}^{\dual}$ is a quotient morphism induced by the restriction morphism $\mathcal{O}(T_X) \to \mathcal{O}(V)$.
Take an open subset $U \subset X$ such that there exist $r$ local sections $\omega_1, \dots, \omega_r \in \mathcal{V}^{\dual}(U)$ generating $V^{\dual}$ at every point of $U$.
Thus one obtains a local trivialization of $p_{k, V}^{-1}(U)$ given by
\begin{equation}\label{eqn:localizationJkV}
\begin{aligned}
p_{k, V} \times (\dif_{U, V}^{[\ell]}\omega_i)_{\substack{1 \leq i \leq r \quad \\ 0 \leq \ell \leq k-1}} \colon p_{k, V}^{-1}(U) &\longrightarrow U \times \mathbb{C}^{rk} \\
j_kf &\longmapsto \left(f(0), (\dif_{U, V}^{[\ell]}\omega_i(j_kf))_{\substack{1 \leq i \leq r \quad \\ 0 \leq \ell \leq k-1}}\right)
.\end{aligned}
\end{equation}
In the particular case where $\omega = \res_V \dif \sigma$ for some $\sigma \in \mathcal{O}(U)$, slightly extending the notation, one writes

\begin{equation}\label{eqn:diffun}
\dif_{U, V}^{[\ell]}\sigma \bydef
\begin{cases}
p_{k, V}^{\ast} \sigma & \ell = 0, \\[.5eM]
\dif_{U, V}^{[\ell-1]}\omega & 1 \leq \ell \leq k.
\end{cases}
\end{equation}
It is evident that $\dif_{U, V}^{[\ell]}\sigma(j_kf) = (\sigma \circ f)^{(\ell)}(0)$ for every $\sigma \in \mathcal{O}(U)$ and $j_kf \in p_{k, V}^{-1}(U) \subset J_kV$.
When no confusion can arise, one simply writes $\dif^{[\ell]}$ for brevity.

\subsection{Demailly-Semple jet towers}
Let us recall the Demailly-Semple jet tower (see Section 4 and 5 of \cite{Demailly1997} for more details).
Let us start from $(X_0, V_0) = (X, V)$, for a given directed manifold $(X, V)$.
In the spirit of \cite{Semple1954}, Demailly \cite{Demailly1997} inductively constructed a sequence of directed manifolds $(X_{\bullet}, V_{\bullet})$ as follows.

Let $X_k = \textnormal{P}(V_{k-1})$ with its natural projection $\pi_k$ to $X_{k-1}$, for each $k \in \mathbb{Z}_{+}$.
For every point $(x, [v]) \in X_k = \textnormal{P}(V_{k-1})$ associated with a tangent vector $v \in (V_{k-1})_x \setminus \{0\}$ for $x \in X_{k-1}$, one defines the tangent subbundle $V_k \subset T_{X_k}$ by
\[
(V_k)_{(x, [v])} = \{ \xi \in (T_{X_k})_{(x, [v])} \mid \dif \pi_k (\xi) \in \mathbb{C}\,v \}, \quad \mathbb{C}\,v \subset (V_{k-1})_x
,\]
where $\dif \pi_k \colon T_{X_k} \to \pi_k^{\ast} T_{X_{k-1}}$ is the differential of $\pi_k$.
One refers to the directed manifold $(X_k, V_k)$ as the Demailly-Semple $k$-jet tower of $(X, V)$.
\begin{remark}
Technically, the projective manifold $\textnormal{P}(V)$ of lines of the vector bundle $V$ is identified with the projective bundle $\mathbb{P}(V^{\dual}) \bydef \Proj \bigoplus_{d\geq 0} \Sym^d V^{\dual}$.
\end{remark}
On $X_k = \textnormal{P}(V_{k-1})$, one has the tautological line bundle $\mathcal{O}_{X_k}(-1) \subset \pi_k^{\ast} V_{k-1}$ such that $\mathcal{O}_{X_k}(-1)_{(x, [v])} = \mathbb{C}\,v$.
Actually, the definition of $V_k$ gives a exact sequence
\begin{equation}\label{eqn:defexact}
\begin{tikzcd}[column sep=scriptsize]
0 \arrow[r] & T_{X_k/X_{k-1}} \arrow[r] & V_k \arrow[r, "(\pi_k)_{\ast}"] &[.5em] \mathcal{O}_{X_k}(-1) \arrow[r] & 0
,\end{tikzcd}
\end{equation}
where $T_{X_k/X_{k-1}} \bydef \ker(\dif \pi_k)$ denotes the relative tangent bundle of $\pi_k$.
Moreover, thinking of $X_k$ as a projective bundle $X_k = \mathbb{P}(V_{k-1}^{\dual})$, one gets the Euler exact sequence
\begin{equation}\label{eqn:Eulexact}
0 \longrightarrow \mathcal{O}_{X_k} \longrightarrow \pi_k^{\ast} V_{k-1} \otimes \mathcal{O}_{X_k}(1) \longrightarrow T_{X_k/X_{k-1}} \longrightarrow 0
.\end{equation}
It is evident that
\[
\dim X_k = n+k(r-1), \qquad
\rank V_k = r
,\]
where $r \bydef \rank V$.
Hence, taking the determinants of \eqref{eqn:Eulexact}, one obtains
\begin{align*}
\det T_{X_k/X_{k-1}} &= \pi_k^{\ast} \det V_{k-1} \otimes \mathcal{O}_{X_k}(r), \\
\det V_k &= \pi_k^{\ast} \det V_{k-1} \otimes \mathcal{O}_{X_k}(r-1)
.\end{align*}

Obviously, $\pi_k$ induces a morphism of directed manifolds
\[
\pi_k \colon (X_k, V_k) \longrightarrow (X_{k-1}, V_{k-1})
,\]
for each $k \in \mathbb{Z}_{+}$.
One sees that this construction is functorial.
That is to say, for every morphism of directed manifolds $\varphi \colon (X, V) \to (Y, W)$, there is a commutative diagram
\begin{equation}\label{eqn:functorial}
\begin{tikzcd}[column sep=scriptsize]
(X_1, V_1) \arrow[r, "\pi_1"] \arrow[d, dashed, "\varphi_1"] & (X, V) \arrow[d, "\varphi"]  \\
(Y_1, W_1) \arrow[r, "\pi_1"] & (Y, W)
,\end{tikzcd}
\end{equation}
where $\varphi_1$ is induced by the differential $\dif \varphi$ restricted on $V$.
Notice that $\varphi_1$ is holomorphic if $\dif \varphi_{\restriction V} \colon V \to \varphi^{\ast} W$ is injective.

Observe that $\mathcal{O}_{X_k}(-1) \hookrightarrow \pi_k^{\ast} V_{k-1}$ is a canonical injection on $X_k = \mathbb{P}(V_{k-1}^{\dual})$,
and the composition with $(\pi_{k-1})_{\ast}$ (analogue for order $k-1$ of the arrow $(\pi_k)_{\ast}$ in sequence \eqref{eqn:defexact}) gives a line bundle morphism over $X_k$
\begin{equation}\label{eqn:Gamma}
\begin{tikzcd}[column sep=scriptsize]
\mathcal{O}_{X_k}(-1) \arrow[r, hook] & \pi_k^{\ast} V_{k-1} \arrow[r, "\pi_k^{\ast} (\pi_{k-1})_{\ast}"] &[2em] \pi_k^{\ast} \mathcal{O}_{X_{k-1}}(-1)
,\end{tikzcd}
\end{equation}
for any integer $k \geq 2$, which admits $\varGamma_k \bydef \textnormal{P}(T_{X_{k-1}/X_{k-2}}) \subset X_k$ as its zero divisor.
Immediately, one infers that
\begin{equation}\label{eqn:O(Gamma)}
\mathcal{O}_{X_k}(1) = \pi_k^{\ast} \mathcal{O}_{X_{k-1}}(1) \otimes \mathcal{O}(\varGamma_k)
.\end{equation}

Now, for any integer $0 \leq j \leq k-1$, one defines the composition of projections
\begin{equation}\label{eqn:pi}
\pi_{j, k} = \pi_{j+1} \circ \pi_{j+2} \circ \dots \circ \pi_k \colon X_k \longrightarrow X_j
.\end{equation}
For the sake of convenience, one writes
\begin{equation}\label{eqn:Oa}
\mathcal{O}_{X_k}(\bm{a}) \bydef \pi_{1, k}^{\ast} \mathcal{O}_{X_1}(a_1) \otimes \pi_{2, k}^{\ast} \mathcal{O}_{X_2}(a_2) \otimes \dots \otimes \mathcal{O}_{X_k}(a_k)
,\end{equation}
for every $\bm{a}=(a_1, \ldots, a_k) \in \mathbb{Z}^k$.
It follows from \eqref{eqn:O(Gamma)} that
\[
\pi_{j, k}^{\ast} \mathcal{O}_{X_j}(1) = \mathcal{O}_{X_k}(1) \otimes \mathcal{O}_{X_k}(-\pi_{j+1, k}^{\ast} \varGamma_{j+1} - \dots - \varGamma_k)
.\]
One thus gets
\begin{equation}\label{eqn:Oab}
\mathcal{O}_{X_k}(\bm{a}) = \mathcal{O}_{X_k}(b_k) \otimes \mathcal{O}_{X_k}(-\bm{b} \cdot \varGamma)
,\end{equation}
where
\begin{align*}
\bm{b} &= (b_1, \dots, b_k) \in \mathbb{Z}^k, \quad b_j = a_1+\dots+a_j, \\
\bm{b} &\cdot \varGamma = \sum_{1\leq j \leq k-1} b_j \pi_{j+1, k}^{\ast} \varGamma_{j+1}
.\end{align*}

\subsection{Lifting of curves to Demailly-Semple jet towers}
Let $f \colon \bigtriangleup(R)\to X$ be a (holomorphic) tangent curve of the directed manifold $(X, V)$, i.e., $f'(z)\in V_{f(z)}$ for every $z \in \bigtriangleup(R)$.
Suppose that $f$ is non-constant. Then we can write $f'(z)=(z-z_0)^m u(z)$ with $m \in \mathbb{N}$ and $u(z_0)\neq 0$, for every $z_0 \in \bigtriangleup(R)$. Here $m>0$ only when $z_0$ is a stationary point. Obviously, $[u(z_0)]$ gives a well-defined and unique tangent line in $\textnormal{P}(V_{f(z_0)})$, even at stationary points, and the map
\begin{align*}
f_{[1]} \colon \bigtriangleup(R) &\longrightarrow X_1 \\
z &\longmapsto f_{[1]}(z) = (f(z),[u(z)])
\end{align*}
is holomorphic.
By definition $f'(z) \in \mathcal{O}_{X_1}(-1)_{f_{[1]}(z)} = \mathbb{C}\,u(z)$, hence the derivative $f'$ defines a section
\[
f' \colon T_{\bigtriangleup(R)} \longrightarrow f_{[1]}^{\ast} \mathcal{O}_{X_1}(-1)
.\]
Moreover, seeing that
\[
\pi_1 \circ f_{[1]}(z) = f(z), \quad
(\pi_1)_{\ast} f_{[1]}'(z) = f'(z) \in \mathbb{C}\,u(z)
,\]
we have $f_{[1]}'(z) \in (V_1)_{(f(z), [u(z)])} = (V_1)_{f_{[1]}(z)}$, which implies that $f_{[1]}$ is a non-constant tangent curve of $(X_1, V_1)$. We thus call $f_{[1]}$ the canonical lifting of $f$ to $(X_1, V_1)$.

Furthermore, every non-constant tangent curve $f$ of $(X, V)$ inductively lifts to a unique non-constant tangent curve $f_{[k]}$ of $(X_k, V_k)$, for each $k \in \mathbb{Z}_{+}$. And the derivative $f_{[k-1]}'$ gives rise to a section
\begin{equation}\label{eqn:fk-1'}
f_{[k-1]}' \colon T_{\bigtriangleup(R)} \longrightarrow f_{[k]}^{\ast} \mathcal{O}_{X_k}(-1)
.\end{equation}

Now, let us investigate the multiplicity of tangent curves and their liftings. We define the multiplicity $m(f, z_0)$ of a curves $f \colon \bigtriangleup(R) \to X$ at a point $z_0$ to be the smallest positive integer $m \in \mathbb{Z}_{+}$ such that $f^{(m)}(z_0) \neq 0$. That is, we can write $f'(z)=(z-z_0)^{m-1} u(z)$ such that $u(z_0)\neq 0$. One easily sees that $m(f_{[k]}, z)$ is non-increasing with $k$. What's more, Demailly verified the following proposition.
\begin{proposition}[Proposition 5.11, \cite{Demailly1997}]
Let $f \colon (\mathbb{C}, 0) \to (X, x)$ be a germ of non-constant curves tangent to $V$. Then we have $m(f_{[j-2]}, 0) \geq m(f_{[j-1]}, 0)$ for any integer $j \geq 2$, and the inequality is strict if and only if $f_{[j]}(0) \in \varGamma_j$.
Conversely, if $x_k \in X_k$ is an arbitrary point and $m_0, m_1, \dots, m_{k-1}$ is a non-increasing sequence of positive integers satisfying
\[
m_{j-2} > m_{j-1} \quad \text{if and only if } \pi_{j, k}(x_k) \in \varGamma_j
,\]
for all integer $2 \leq j \leq k$, then there exist a germ of curves $f \colon (\mathbb{C}, 0) \to (X, x)$ tangent to $V$ such that $f_{[k]}(0)=x_k$ and $m(f_{[j]}, 0)=m_j$ for all integer $0 \leq j \leq k-1$, where $x = \pi_{0, k} (x_k)$.
\end{proposition}

Observe that we can always take $m_{k-1}=1$. Thus we have
\begin{corollary}[Corollary 5.12, \cite{Demailly1997}]\label{cor:regular}
For any $x_k \in X_k$, there is a germ of curves $f \colon (\mathbb{C}, 0) \to (X, x)$ such that $f_{[k]}(0)=x_k$ and $f'_{[k-1]}(0) \neq 0$. Moreover, if $x_k$ is taken in a sufficiently small neighborhood, then the germ $f = f_{x_k}$ can be taken to depend holomorphically on $x_k$.
\end{corollary}

Therefore, $f \colon (\mathbb{C}, 0) \to (X, x)$ is regular at the origin (i.e., $f'(0)\neq 0$) if and only if $x_k=f_{[k]}(0) \in X_k$ satisfies $\pi_{j, k}(x_k) \not\in \varGamma_j$ for all $2\leq j \leq k$. Hence, we naturally define
\begin{align*}
X_k^{\reg} &= \bigcap_{2\leq j \leq k} \pi_{j, k}^{-1}(X_j \setminus \varGamma_j), \\
X_k^{\sing} &= \bigcup_{2\leq j \leq k} \pi_{j, k}^{-1}(\varGamma_j) = X_k \setminus X_k^{\reg} 
.\end{align*}

\subsection{Green-Griffiths jet differentials}
We now introduce the concept of jet differentials in the sense of Green-Griffiths \cite{Green1980}.

Let $\mathbb{G}_k$ be the group of germs of $k$-jets of biholomorphisms of $(\mathbb{C}, 0)$, i.e., the group of germs of biholomorphic maps
\begin{align*}
\varphi \colon (\mathbb{C}, 0) &\longrightarrow (\mathbb{C}, 0) \\
z &\longmapsto \varphi(z) = a_1z+a_2z^2+\dots+a_kz^k
,\end{align*}
where $a_1 \in \mathbb{C}^{\mltp}$, $a_j \in \mathbb{C}$, $j\geq 2$, and the composition law is taken modulo $z^{k+1}$.
Apparently, $\mathbb{G}_k$ gives a natural reparametrization action on $J_kV$, $\varphi \cdot j_kf \mapsto j_k(f \circ \varphi)$.
Observe that the commutator subgroup $\mathbb{G}_k' \bydef [\mathbb{G}_k, \mathbb{G}_k]$ is the group of $k$-jets of biholomorphisms tangent to the identity. Thus there is an exact sequence of groups
\[
1 \longrightarrow \mathbb{G}_k' \longrightarrow \mathbb{G}_k \longrightarrow \mathbb{C}^{\mltp} \longrightarrow 1
,\]
where $\mathbb{G}_k \to \mathbb{C}^{\mltp}$ is the morphism $\varphi \mapsto \varphi'(0)$, and $\mathbb{C}^{\mltp}$ is interpreted as the group of homotheties $\varphi(z) = \lambda z$.
Suppose $j_kf$ is parametrized by $(f', f'', \ldots, f^{(k)})$.
Thereupon a reparametrization action $\lambda \in \mathbb{C}^{\mltp}$ on $J_kV$ is described by
 \[
\lambda \cdot (f', f'', \ldots, f^{(k)}) = (\lambda f', \lambda^2 f'', \ldots, \lambda^k f^{(k)})
.\]

Following \cite{Green1980}, we consider homogeneous polynomials $Q$ on $J_kV$ of weighted degree $m$ with respect to the reparametrization action of $\mathbb{C}^{\mltp}$, which amounts to satisfying that
\[
Q(\lambda f', \lambda^2 f'', \ldots, \lambda^k f^{(k)}) = \lambda^m Q(f', f'', \ldots, f^{(k)})
\]
for all $\lambda \in \mathbb{C}^{\mltp}$ and $(f', f'', \ldots, f^{(k)}) \in J_kV$.
Notice that the concept of polynomial on the fibers of $J_kV$ makes sense, for any coordinate changes $z \mapsto w=\Psi(z)$ on $X$ induce a polynomial transition automorphism on the fibers of $J_kV$, given by
\begin{equation}\label{eqn:Psi}
(\Psi\circ f)^{(\ell)}=\Psi'(f) \cdot f^{(\ell)} + \sum_{s=2}^{\ell} \sum_{\substack{\bm{j} \in \mathbb{Z}_{+}^s \\ \|\bm{j}\|=\ell}}c_{\bm{j}} \Psi^{(s)}(f) \cdot (f^{(j_1)} \cdots f^{(j_s)})
\end{equation}
with suitable integer constants $c_{\bm{j}}$ (these are easily computed by induction on $s$), where $\|\bm{j}\| \bydef j_1+\dots+j_s$.

Green-Griffiths jet differentials sheaf $\mathcal{E}_{k, m}^{GG}V^{\dual}$ of order $k$ and weighted degree $m$, is a locally free sheaf consisting of these polynomials $Q$, i.e., for any open subset $U \subset X$,
\[
\mathcal{E}_{k, m}^{GG}V^{\dual}(U) \bydef \{Q \in \mathcal{O}(p_{k, V}^{-1}(U)) \mid Q(\lambda \cdot j_kf) = \lambda^m Q(j_kf), \forall j_kf \in p_{k, V}^{-1}(U), \forall \lambda \in \mathbb{C}^{\mltp}\}
.\]
The corresponding vector bundle is denoted by $E_{k, m}^{GG}V^{\dual}$, and hereby called Green-Griffiths jet differentials bundle of order $k$ and weighted degree $m$.

For a holomorphic subbundle $V \subset T_X$, there is a natural inclusion $J_kV \hookrightarrow J_kX$. Hence, the restriction induces a surjective morphism $\mathcal{E}_{k, m}^{GG}T_X^{\dual} \twoheadrightarrow \mathcal{E}_{k, m}^{GG}V^{\dual}$.
This implies that $\mathcal{E}_{k, m}^{GG}V^{\dual}$ is a quotient of $\mathcal{E}_{k, m}^{GG}T_X^{\dual}$.

For any Green-Griffiths jet differentials $Q \in \mathcal{E}_{k, m}^{GG}V^{\dual}$, it can be decomposed into multihomogeneous components of multidegree $\bm{\ell} \bydef (\ell_1, \ell_2, \dots, \ell_k)  \in \mathbb{N}^k$ in $f', f'', \dots, \allowbreak f^{(k)}$ (the decomposition is certainly dependent on the choice of coordinate).
It is evident that the coordinate change must keep the weighted degree constant, that is, $|\bm{\ell}|_k \bydef \ell_1 + 2\ell_2 + \dots + k\ell_k = m$.
Moreover, investigate the partial weighted degree $|\bm{\ell}|_s \bydef \ell_1 + 2\ell_2 + \dots + s\ell_s$ of order $s$.
For $1 \leq s \leq k$, it follows from \eqref{eqn:Psi} that a coordinate change $f \mapsto \Psi \circ f$ transforms every monomial $(f^{(\bullet)})^{\bm{\ell}} \bydef (f')^{\ell_1} (f'')^{\ell_2} \cdots (f^{(k)})^{\ell_k}$ into a polynomial $((\Psi \circ f)^{(\bullet)})^{\bm{\ell}}$ whose non-zero monomials have a larger or equal partial weighted degree of order $s$.
In particular, if $\ell_{s+1} = \dots = \ell_k = 0$, it keeps the partial weighted degree of order $s$ the same.

Therefore, for each integer $1 \leq s \leq k-1$, one gets a well-defined decreasing filtration $F_s^\bullet$ on $\mathcal{E}_{k, m}^{GG}V^{\dual}$ as follows.
For every positive integer $p \leq m$, the sheaf $F_s^p(\mathcal{E}_{k,m}^{GG}V^{\dual})$ is the set of polynomials $Q(f', f'', \dots, f^{(k)}) \in \mathcal{E}_{k, m}^{GG}V^{\dual}$ involving only monomials $(f^{(\bullet)})^{\bm{\ell}}$ with $|\ell|_s \geq p$, regardless of the choice of coordinate.

One denotes the graded terms associated with the filtration $F_{k-1}^p(\mathcal{E}_{k, m}^{GG}V^{\dual})$ by
\[
\Gr_{k-1}^p(\mathcal{E}_{k, m}^{GG}V^{\dual}) = F_{k-1}^p(\mathcal{E}_{k, m}^{GG}V^{\dual}) / F_{k-1}^{p+1}(\mathcal{E}_{k, m}^{GG}V^{\dual})
.\]
Technically, $\Gr_{k-1}^p(\mathcal{E}_{k, m}^{GG}V^{\dual})$ are precisely the homogeneous polynomials $Q$ whose monomials $(f^{\bullet})^{\bm{\ell}}$ all have partial weighted degree $|\ell|_{k-1}=p$.
Obviously, the degree $\ell_k$ in $f^{(k)}$ of each monomial needs to satisfy $m-p=k\ell_k$, and $\Gr_{k-1}^p(\mathcal{E}_{k, m}^{GG}V^{\dual})=0$ unless $k|m-p$.
The transition automorphisms of the graded bundle are induced by coordinate changes $f \mapsto \Psi \circ f$, and they are described by substituting the arguments of $Q(f', f'', \dots, f^{(k)})$ according to formula \eqref{eqn:Psi}, namely,
$f^{(j)} \mapsto (\Psi \circ f)^{(j)}$ for $j<k$, and $f^{(k)} \mapsto \Psi'(f) \circ f^{(k)}$ for $j=k$ (when $j=k$, the other terms fall in the next stage $F^{p+1}_{k-1}$ of the filtration).
Hence $f^{(k)}$ behaves as an element of $V \subset T_X$ under coordinate changes. One infers
\[
\Gr_{k-1}^{m-k\ell_k}(\mathcal{E}_{k, m}^{GG}V^{\dual}) \cong \mathcal{E}_{k-1,m-k\ell_k}^{GG}V^{\dual} \otimes \Sym^{\ell_k}V^{\dual}
.\]
Combining all filtrations $F_s^{\bullet}$ together, one inductively finds a filtration $F^{\bullet}$ on $\mathcal{E}_{k, m}^{GG}V^{\dual}$ such that the graded terms are
\begin{equation}\label{eqn:Gr}
\Gr^{\bm{\ell}}(\mathcal{E}_{k, m}^{GG}V^{\dual}) \cong \Sym^{\ell_1}V^{\dual} \otimes \Sym^{\ell_2}V^{\dual} \otimes \dots \otimes \Sym^{\ell_k}V^{\dual}
,\end{equation}
for $\bm{\ell} \in \mathbb{N}^k$ with $|\bm{\ell}|_k = m$.

Associated with the graded algebra sheaf $\mathcal{E}_{k, \bullet}^{GG}V^{\dual} \bydef \bigoplus_{m \geq 0} \mathcal{E}_{k, m}^{GG}V^{\dual}$, one has an analytic fiber bundle
\begin{equation}
X_k^{GG} \bydef \Proj \mathcal{E}_{k, \bullet}^{GG}V^{\dual} = (J_kV \setminus (X \times \{0\}))/\mathbb{C}^{\mltp}
\end{equation}
over $X$, where $J_kV \setminus (X \times \{0\})$ is the set of non-constant $k$-jets. By construction, one easily sees that $X_k^{GG}$ has weighted projective spaces $\mathbb{P}(1^{[r]}, 2^{[r]}, \ldots, k^{[r]})$ as fibers. Moreover, under the natural projection $\pi_k^{GG} \colon X_k^{GG} \to X$, one obtains the direct image formula
\begin{equation}
(\pi_k^{GG})_{\ast} \mathcal{O}_{X_k^{GG}}(m) \cong \mathcal{E}_{k, m}^{GG}V^{\dual}
\end{equation}
for all positive integers $k$ and $m$.

\subsection{Invariant jet differentials}
In geometric context, $(J_kV \setminus (X \times \{0\}))/\mathbb{C}^{\mltp}$ is unwieldy, and we prefer $(J_kV \setminus \{0\})/\mathbb{G}_k$ if such quotients would exist.
We will see that $X_k$ plays the role of such a quotient.
First, we recall a canonical vector subbundle $E_{k, m}V^{\dual}$ of $E_{k, m}^{GG}V^{\dual}$ introduced by Demailly \cite{Demailly1997}.
The corresponding sheaf $\mathcal{E}_{k, m}V^{\dual}$ is locally defined by
\[
\mathcal{E}_{k, m}V^{\dual}(U) \bydef \{Q \in \mathcal{O}(p_{k, V}^{-1}(U)) \mid Q(\varphi \cdot j_kf) = \varphi'(0)^m Q(j_kf), \forall j_kf \in p_{k, V}^{-1}(U), \forall \varphi \in \mathbb{G}_k\}
,\]
for any open subset $U \subset X$. In other words, $E_{k, m}V^{\dual} = (E_{k, m}^{GG}V^{\dual})^{\mathbb{G}_k'}$ is the set of invariants of $E_{k, m}^{GG}V^{\dual}$ with respect to the reparametrization action of $\mathbb{G}_k'$. We call $\mathcal{E}_{k, m}V^{\dual}$ invariant jet differentials sheaf of order $k$ and weighted degree $m$, and $E_{k, m}V^{\dual}$ invariant jet differentials bundle of order $k$ and weighted degree $m$.

Seeing that $\mathcal{E}_{k, m}V^{\dual} \subset \mathcal{E}_{k, m}^{GG}V^{\dual}$, there are natural induced filtrations
\[
F_s^p(\mathcal{E}_{k, m}V^{\dual}) = \mathcal{E}_{k, m}V^{\dual} \bigcap F_s^p(\mathcal{E}_{k, m}^{GG}V^{\dual})
\]
on $\mathcal{E}_{k, m}V^{\dual}$, for $1 \leq s \leq k-1$, $0 \leq p \leq m$, and the graded terms are
\[
\Gr^{\bm{\ell}}(\mathcal{E}_{k, m}V^{\dual}) = \mathcal{E}_{k, m}V^{\dual} \bigcap \Gr^{\bm{\ell}}(\mathcal{E}_{k, m}^{GG}V^{\dual})
,\]
for $\bm{\ell} \in \mathbb{N}^k$ such that $|\bm{\ell}|_k = m$.

\begin{theorem}[Theorem 6.8, \cite{Demailly1997}]\label{thm:DS}
Suppose that $(X, V)$ is a directed manifold such that $V$ has rank $r \geq 2$. Let $\pi_{0, k} \colon X_k \to X$ be a natural projection defined by \eqref{eqn:pi}, and let $J_kV^{\reg}$ be the bundle of regular $k$-jets of maps $f \colon (\mathbb{C},0) \to X$, that is, $k$-jets $j_kf$ such that $f'(0)\neq 0$.
\begin{enumerate}[i)]
\item The quotient $J_kV^{\reg} / \mathbb{G}_k$ has the structure of a locally trivial bundle over $X$, and there is a holomorphic embedding $J_kV^{\reg} / \mathbb{G}_k \hookrightarrow X_k$ over $X$, which identifies $J_kV^{\reg} / \mathbb{G}_k$ with $X_k^{\reg}$. That is to say,
\[
J_kV^{\reg} / \mathbb{G}_k \cong X_k^{\reg}
.\]
\item The direct image sheaf $(\pi_{0,k})_{\ast}\mathcal{O}_{X_k}(m)$ is identified with the sheaf of invariant jet differentials $\mathcal{E}_{k, m}V^{\dual}$, namely,
\begin{equation}\label{eqn:EV}
(\pi_{0,k})_{\ast}\mathcal{O}_{X_k}(m) \cong \mathcal{E}_{k,m}V^{\dual}
.\end{equation}
\item\label{itm:relbs} For every $m>0$, the relative base locus of the linear system $|\mathcal{O}_{X_k}(m)|$ is equal to the set $X_k^{\sing}$. Moreover, $\mathcal{O}_{X_k}(1)$ is relatively big over $X$.
\end{enumerate}
\end{theorem}

Let $\bm{a} \in \mathbb{Z}^k$ such that $\|\bm{a}\| = m$.
We then have a subsheaf $\widebar{F}^{\bm{a}}\mathcal{E}_{k, m}V^{\dual} \subset \mathcal{E}_{k, m}V^{\dual}$, which consists of polynomials $Q(f', \dots, f^{(k)}) \in \mathcal{E}_{k, m}V^{\dual}$ involving only monomials $(f^{(\bullet)})^{\bm{\ell}}$ such that
\[
|\bm{\ell}|_{>s} \bydef \ell_{s+1} + 2\ell_{s+2} + \dots + (k-s)\ell_k \leq a_{s+1} + \dots + a_k
,\]
regardless of the choice of coordinate, for each integer $0 \leq s \leq k-1$.
Apparently, $\widebar{F}^{\bm{a}}\mathcal{E}_{k, m}V^{\dual}$ is a locally free subsheaf of $\mathcal{E}_{k, m}V^{\dual}$ over $X$, and the corresponding vector bundle $\widebar{F}^{\bm{a}}E_{k, m}V^{\dual}$ is also a subbundle of $E_{k, m}V^{\dual}$.
Likewise, there are natural induced filtrations
\[
F_s^p(\widebar{F}^{\bm{a}}\mathcal{E}_{k, m}V^{\dual}) = \widebar{F}^{\bm{a}}\mathcal{E}_{k, m}V^{\dual} \bigcap F_s^p(\mathcal{E}_{k, m}^{GG}V^{\dual})
\]
on $\widebar{F}^{\bm{a}}\mathcal{E}_{k, m}V^{\dual}$, for $1 \leq s \leq k-1$, $0 \leq p \leq m$, and the graded terms are
\[
\Gr^{\bm{\ell}}(\widebar{F}^{\bm{a}}\mathcal{E}_{k, m}V^{\dual}) = \widebar{F}^{\bm{a}}\mathcal{E}_{k, m}V^{\dual} \bigcap \Gr^{\bm{\ell}}(\mathcal{E}_{k, m}^{GG}V^{\dual})
,\]
for $\bm{\ell} \in \mathbb{N}^k$ such that $|\bm{\ell}|_k = m$.

According to Proposition 6.16 in \cite{Demailly1997}, for any $k, m \geq 0$ and any $\bm{a}$ as above, one also has
\begin{equation}\label{eqn:FEV}
(\pi_{0,k})_{\ast}\mathcal{O}_{X_k}(\bm{a}) \cong \widebar{F}^{\bm{a}}\mathcal{E}_{k,m}V^{\dual}
.\end{equation}

\section{Wronskians on Demailly-Semple jet towers}\label{sec:Wronskians}

\subsection{Wronskian sections}
In this subsection, we are going to recall the Wronskian sections on the Demailly-Semple jet tower (see \cite{Brotbek2017}).
Let us start with a local construction of Wronskians.

Let $(X, V)$ be a complex directed manifold, and let $U$ be an open subset of $X$. Brotbek \cite{Brotbek2017} constructed the Wronskian of $k+1$ local holomorphic functions $\sigma_0, \sigma_1, \dots, \sigma_k \in \mathcal{O}(U)$ as follows:
\[
\renewcommand{\arraystretch}{1.5}
W_U^V(\sigma_0, \dots, \sigma_k) \bydef
\begin{vmatrix}
\dif^{[0]}\sigma_0 & \cdots & \dif^{[0]}\sigma_k \\
\vdots & \ddots & \vdots \\
\dif^{[k]}\sigma_0 & \cdots & \dif^{[k]}\sigma_k
\end{vmatrix}
\in \mathcal{O}(p_{k, V}^{-1}(U))
,\]
where $\dif^{[\ell]} = \dif_{U, V}^{[\ell]}$ defined by \eqref{eqn:diffun}.

For any positive integer $k$, set
\begin{align*}
k' &\bydef 1+2+\dots+k = \frac{k(k+1)}{2}, \\
\bm{a}^k &\bydef (k, k-1, \dots, 1) \in \mathbb{Z}^k
.\end{align*}
Applying Faà Di Bruno's formula, Brotbek \cite{Brotbek2017} verified that $W_U^V(\sigma_0, \dots, \sigma_k) \in \mathcal{E}_{k, k'}V^{\dual}(U)$.
More concretely, one easily checks that $W_U^V(\sigma_0, \dots, \sigma_k)(j_kf)$ involves only monomials $(f^{(\bullet)})^{\bm{\ell}}$ such that $|\bm{\ell}|_{>s} \leq a_{s+1}^k + \dots + a_k^k = \displaystyle\frac{(k-s)(k-s+1)}{2}$, for each integer $0 \leq s \leq k-1$.
In other words,
\[
W_U^V(\sigma_0, \dots, \sigma_k) \in \widebar{F}^{\bm{a}^k}\!\mathcal{E}_{k, k'}V^{\dual}(U) \subset \mathcal{E}_{k, k'}V^{\dual}(U)
.\]
In addition, thinking of the variant of $W_U^V$ as a vector $\begin{pmatrix} \sigma_0 & \sigma_1 & \dots & \sigma_k \end{pmatrix}$ in $\mathcal{O}(U)^{k+1}$, we easily see that $W_U^V$ is invariant under the action of $\mathrm{SL}(k+1, \mathbb{C})$ on $\mathcal{O}(U)^{k+1}$.

Observe that the restriction morphism $\mathcal{O}(J_kX) \to \mathcal{O}(J_kV)$ induces a quotient morphism $\res_{k, V} \colon \mathcal{E}_{k, \bullet}^{GG}T_X^{\dual} \to \mathcal{E}_{k, \bullet}^{GG}V^{\dual}$.
Obviously, $\res_{k, V}$ also maps $\overline{F}^{\bm{a}}\mathcal{E}_{k, m}T_X^{\dual}$ into $\overline{F}^{\bm{a}}\mathcal{E}_{k, m}V^{\dual}$, and we have
\[
W_U^V(\sigma_0, \dots, \sigma_k) = \res_{k, V} \left(W_U^{T_X}(\sigma_0, \dots, \sigma_k)\right)
.\]
When no confusion can arise, we hereinafter prefer to write $W_U(\sigma_0, \dots, \sigma_k)$ rather than $W_U^V(\sigma_0, \dots, \sigma_k)$.

Now, we will globalize this construction.
For a holomorphic line bundle $L$ on $X$, take an open subset $U$, on which $L$ can be trivialized. 
Take a trivialization of $L_{\restriction U}$, which associates $\sigma_U \in \mathcal{O}(U)$ to a global section $\sigma \in H^0(X, L)$.
Given global sections $\sigma_0, \dots, \sigma_k \in H^0(X, L)$, we immediately get a locally defined invariant jet differential
\[
W_U(\sigma_0, \dots, \sigma_k) \bydef W_U(\sigma_{0, U}, \dots, \sigma_{k, U}) \in \widebar{F}^{\bm{a}^k}\!\mathcal{E}_{k, k'}V^{\dual}(U) \subset \mathcal{E}_{k, k'}V^{\dual}(U)
.\]
Moreover, they glue together into a global section
\[
W(\sigma_0, \dots, \sigma_k) \in H^0(X, \widebar{F}^{\bm{a}^k}\!E_{k, k'}V^{\dual} \!\otimes L^{k+1}) \subset H^0(X, E_{k, k'}V^{\dual} \!\otimes L^{k+1})
.\]

Thanks to \eqref{eqn:EV}, there is a global section
\[
\omega(\sigma_0, \dots, \sigma_k) \in H^0(X_k, \mathcal{O}_{X_k}(k') \otimes \pi_{0, k}^{\ast} L^{k+1})
\]
corresponding to $W(\sigma_0, \dots, \sigma_k)$, that is, they satisfy the relation
\[
(\pi_{0, k})_{\ast}\omega(\sigma_0, \dots, \sigma_k) = W(\sigma_0, \dots, \sigma_k)
.\]
More precisely, following the proof of Theorem 6.8 in \cite{Demailly1997},

\begin{equation}\label{eqn:wf'}
W(\sigma_0, \dots, \sigma_k)(j_kf_{\!+z_0}) = \omega(\sigma_0, \dots, \sigma_k)(f_{[k]}(z_0)) \cdot (f'_{[k-1]}(z_0))^{k'}
\end{equation}
holds for any tangent curve $f$ of $(X, V)$ and every regular point $z_0$ of $f$, where $f_{\!+z_0}(z) \bydef f(z + z_0)$ and $f_{[k-1]}'$ is defined by \eqref{eqn:fk-1'}.
Likewise, there is also a global section
\[
\widebar{\omega}(\sigma_0, \dots, \sigma_k) \in H^0(X_k, \mathcal{O}_{X_k}(\bm{a}^k) \otimes \pi_{0, k}^{\ast} L^{k+1})
\]
corresponding to $W(\sigma_0, \dots, \sigma_k)$ under the isomorphism \eqref{eqn:FEV}.

According to \eqref{eqn:Oab}, one has
\[
\mathcal{O}_{X_k}(k') = \mathcal{O}_{X_k}(\bm{a}^k) \otimes \mathcal{O}_{X_k}(\bm{b}^k \cdot \varGamma)
,\]
where $\bm{b}^k = (b_1, \dots, b_k) \in \mathbb{Z}^k$ such that $b_j = k+(k-1)+\dots+(k-j+1) =\dfrac{j(2k-j+1)}{2}$ for each $1 \leq j \leq k$,
and $\bm{b}^k \cdot \varGamma = \sum_{1\leq j \leq k-1} b_j \pi_{j+1, k}^{\ast} \varGamma_{j+1}$.
We conclude from \eqref{eqn:EV} and \eqref{eqn:FEV} that the inclusion
\[
\widebar{F}^{\bm{a}^k}\!\mathcal{E}_{k, k'}V^{\dual} \lhook\joinrel\longrightarrow \mathcal{E}_{k, k'}V^{\dual}
\]
induces the inclusion
\[
\mathcal{O}_{X_k}(\bm{a}^k) \lhook\joinrel\longrightarrow \mathcal{O}_{X_k}(k') 
.\]
More concretely, the latter inclusion is written as
\begin{align*}
\sigma_{\bm{b}^k \cdot \varGamma} \otimes \colon \mathcal{O}_{X_k}(\bm{a}^k) &\lhook\joinrel\longrightarrow \mathcal{O}_{X_k}(k') \\
\xi &\longmapsto \sigma_{\bm{b}^k \cdot \varGamma}(\pi(\xi)) \otimes \xi
,\end{align*}
for a canonical section $\sigma_{\bm{b}^k \cdot \varGamma}$ of $\mathcal{O}_{X_k}(\bm{b}^k \cdot \varGamma)$, i.e., $(\sigma_{\bm{b}^k \cdot \varGamma}) = \bm{b}^k \cdot \varGamma$,
where $\pi \colon \mathcal{O}_{X_k}(\bm{a}^k) \to X_k$ is the bundle projection.
It amounts to saying that
\begin{equation}\label{eqn:bkGamma}
\omega(\sigma_0, \dots, \sigma_k) = \sigma_{\bm{b}^k \cdot \varGamma} \cdot \widebar{\omega}(\sigma_0, \dots, \sigma_k)
,\end{equation}
for any global sections $\sigma_0, \dots, \sigma_k \in H^0(X, L)$.

\subsection{Wronskian ideal sheaf}
In this subsection, we are going to generalize the Wronskian ideal sheaf in \cite{Brotbek2017} to linear systems.
Let us recall the linear system and the related terminology.

Let $X$ be a projective manifold, and choose a divisor $D_0$ on $X$.
A complete linear system $|D_0|$ on $X$ is defined as the set (maybe empty) of all effective divisors linearly equivalent to $D_0$, namely,
\[
|D_0| \bydef \{D \geq 0 \mid D \sim D_0\}
.\]
Let $L = \mathcal{O}_X(D_0)$ be a holomorphic line bundle on $X$.
It is evident that $|D_0|$ is in one-to-one corresponding with $(H^0(X, L) \setminus \{0\}) / \mathbb{C}^{\mltp}$, that is, $|D_0| \cong \textnormal{P}(H^0(X, L))$.
This implies that $|D_0|$ merely depends on the line bundle $L$ associated to $D$.
Thus $|L| \bydef |D_0|$ is well-defined.

Furthermore, a subset $\mathfrak{S} \subset |D_0|$ is called a linear system if it is a linear subspace for the projective space structure of $|D_0|$,
which amounts to saying that $\mathfrak{S}$ corresponds to a linear subspace $\mathbb{S} \subset H^0(X, L)$, where $\mathbb{S} = \{\sigma \in H^0(X, L) \mid (\sigma) \in \mathfrak{S}\} \cup \{0\}$.
Obviously, $\mathfrak{S} \cong \textnormal{P}(\mathbb{S})$.

A point $x \in X$ is said to be a base point of a linear system $\mathfrak{S}$ if $x \in \Supp D$ for all $D \in \mathfrak{S}$.
The set of all the base points of $\mathfrak{S}$ 
\[
\Bs(\mathfrak{S}) \bydef \{x \in X \mid x \in \Supp D, \, \forall \, D \in \mathfrak{S}\}
\]
is referred to as the base locus of $\mathfrak{S}$.
What's more, a linear system $\mathfrak{S}$ is said to be base point free, if $\Bs(\mathfrak{S}) = \emptyset$, i.e., for every $x \in X$, there exists $D \in \mathfrak{S}$ such that $x \notin \Supp D$.

Let us recall the base ideal of a linear system $\mathfrak{S} \subset |L|$ on $X$.
Consider the evaluation map
\[
\mathrm{ev} \colon \mathbb{S} \longrightarrow L
,\]
where $\mathbb{S} \subset H^0(X, L)$ is the linear subspace corresponding to $\mathfrak{S}$.
Then
\[
\mathfrak{s} \bydef \mathrm{im}(\mathrm{ev}) \otimes L^{-1} \subset \mathcal{O}_X
\]
is called the base ideal of $\mathfrak{S}$.
It is evident that $\mathfrak{s} = \mathcal{O}_X$ if and only if $\mathfrak{S}$ is base point free.

Take a directed projective manifold $(X, V)$, and choose a non-empty linear system $\mathfrak{S} \subset |L|$ on $X$ with the corresponding subspace $\mathbb{S} \subset H^0(X, L)$.
Set
\[
\mathbb{W}(X_k, \mathfrak{S}) \bydef \Span \{\omega(\sigma_0, \dots, \sigma_k) \mid \sigma_0, \dots, \sigma_k \in \mathbb{S}\}
,\]
which is a linear subspace of $H^0(X_k, \mathcal{O}_{X_k}(k') \otimes \pi_{0, k}^{\ast} L^{k+1})$.
We denote by $\mathfrak{W}(X_k, \mathfrak{S})$ the corresponding linear system, i.e.,
\[
\mathfrak{W}(X_k, \mathfrak{S}) \bydef \{(\omega) \mid \omega \neq 0, \, \omega \in \mathbb{W}(X_k, \mathfrak{S})\}
.\]
Therefore, we define the $k$-th Wronskian ideal sheaf of $\mathfrak{S}$ to be the base ideal of $\mathfrak{W}(X_k, \mathfrak{S})$, and denote it by $\mathfrak{w}(X_k, \mathfrak{S})$.

By convention, if $\mathfrak{S} = |L|$ is a complete linear system, we simply write $\mathbb{W}(X_k, L)$, $\mathfrak{W}(X_k, L)$ and $\mathfrak{w}(X_k, L)$ rather than $\mathbb{W}(X_k, \mathfrak{S})$, $\mathfrak{W}(X_k, \mathfrak{S})$ and $\mathfrak{w}(X_k, \mathfrak{S})$, respectively.
Obviously, our definition of $\mathfrak{w}(X_k, \mathfrak{S})$ generalizes the $k$-th Wronskian ideal sheaf of $L$, which is introduced by Brotbek \cite{Brotbek2017}.

In addition, set
\[
\widebar{\mathbb W}(X_k, \mathfrak{S}) \bydef \Span \{\widebar{\omega}(\sigma_0, \dots, \sigma_k) \mid \sigma_0, \dots, \sigma_k \in \mathbb{S}\}
\]
as a linear subspace of $H^0(X_k, \mathcal{O}_{X_k}(\bm{a}^k) \otimes \pi_{0, k}^{\ast} L^{k+1})$.
Likewise, we denote the corresponding linear system by $\widebar{\mathfrak W}(X_k, \mathfrak{S})$, and denote the base ideal of $\widebar{\mathfrak W}(X_k, \mathfrak{S})$ by $\widebar{\mathfrak w}(X_k, \mathfrak{S})$.
It follows from \eqref{eqn:bkGamma} that
\begin{align*}
\widebar{\mathfrak W}(X_k, \mathfrak{S}) &= \{D \mid D + \bm{b}^k \cdot \varGamma \in \mathfrak{W}(X_k, \mathfrak{S})\}, \\
\widebar{\mathfrak w}(X_k, \mathfrak{S}) &= \mathfrak{w}(X_k, \mathfrak{S}) \cdot \mathcal{O}_{X_k}(\bm{b}^k \cdot \varGamma)
.\end{align*}
In particular, $\widebar{\mathfrak W}(X_1, \mathfrak{S}) = \mathfrak{W}(X_1, \mathfrak{S})$ and $\widebar{\mathfrak w}(X_1, \mathfrak{S}) = \mathfrak{w}(X_1, \mathfrak{S})$.

We say that $\mathfrak{S}$ separates $k$-jets at a point $x \in X$ if the evaluation map
\[
\mathbb{S} \longrightarrow L \otimes \mathcal{O}_{X, x}/\mathfrak{m}_{X, x}^{k+1}
\]
is surjective.
Analogue to Lemma 2.4 in \cite{Brotbek2017}, we have the following lemma.
\begin{lemma}\label{lem:wseparate}
If $\mathfrak{S}$ separates $k$-jets at each point of $X$, then
\begin{equation*}
\Supp (\mathcal{O}_{X_k}/\mathfrak{w}(X_k, \mathfrak{S})) = X_k^{\sing}
.\end{equation*}
In other words, the base locus $\Bs(\mathfrak{W}(X_k, \mathfrak{S})) = X_k^{\sing}$.
\end{lemma}

What's more, analogue to Lemma 2.8 in \cite{Brotbek2017}, we see that $\mathfrak{w}(X_k, \mathfrak{S})$ is independent of the choice of the linear system $\mathfrak{S}$, provided that $\mathfrak{S}$ separates $k$-jets at each point of $X$.
\begin{lemma}
Let $\mathfrak{S}$ be a linear system that separates $k$-jets at each point of $X$.
For any $x \in X$ and any $x_k \in X_k$ such that $\pi_{0, k}(x_k)=x$, one has
\[
\mathfrak{w}(X_k, \mathfrak{S})_{x_k} = \big(\omega_{x_k}(f_0, \dots, f_k)\big)_{f_0, \dots, f_k \in \mathcal{O}_{X,x}} \subset \mathcal{O}_{X_k, x_k}
,\]
where the right hand side denotes the ideal whose generators are $\omega_{x_k}(f_0, \dots, f_k)$ for all $f_0, \dots, f_k \in \mathcal{O}_{X,x}$.
\end{lemma}

Therefore, we can define the asymptotic Wronskian ideal sheaf of $X_k$ by
\[
\mathfrak{w}_{\infty}(X_k) \bydef \mathfrak{w}(X_k, \mathfrak{S}) \subset \mathcal{O}_{X_k}
,\]
for any linear system $\mathfrak{S}$ that separates $k$-jets at each point of $X$.

\section{Logarithmic jets and jet differentials}\label{sec:logjets}

\subsection{Logarithmic jet spaces}

Let $X$ be a complex manifold of dimension $n$ with a normal crossing divisor $D$.
This amounts to saying that for each point $x \in X$, one can take an open subset $U \subset X$ containing $x$ with local coordinates $z_1, \dots, z_n$ centered at $x$ such that $D$ is defined by $z_1 z_2 \cdots z_{\ell} = 0$ in $U$ for some $\ell \leq n$ dependent on $x$.
Then such a pair $(X, D)$ is called a logarithmic manifold.
The logarithmic cotangent sheaf $\mathcal{T}_X^{\dual}(\log D)$ of $(X, D)$ is a locally free subsheaf of the sheaf of meromorphic $1$-form on $X$, whose localization at any point $x \in X$ is given by
\[
\mathcal{T}_{X, x}^{\dual}(\log D) \bydef \sum_{i=1}^{\ell} \mathcal{O}_{X, x} \frac{\dif z_i}{z_i} + \sum_{j=\ell+1}^{n} \mathcal{O}_{X, x} \dif z_j
.\]
Seeing that $\ell = 0$ when $x \notin D$, this local expression implies that the restriction of $\mathcal{T}_X^{\dual}(\log D)$ on $X \setminus D$ is identified with $\mathcal{T}_{X \setminus D}^{\dual}$.
The vector bundle associated to $\mathcal{T}_X^{\dual}(\log D)$ is denoted by $T_X^{\dual}(\log D)$ and called the logarithmic cotangent bundle along $D$.
And its dual is denoted by $T_X(-\log D)$, and called the logarithmic tangent bundle along $D$.
The corresponding sheaf $\mathcal{T}_X(-\log D)$ is a local free subsheaf of the holomorphic tangent sheaf $\mathcal{T}_X$, whose localization at any point $x \in X$ is given by
\[
\mathcal{T}_{X, x}(-\log D) \bydef \sum_{i=1}^{\ell} \mathcal{O}_{X, x} z_i \frac{\partial}{\partial z_i} + \sum_{j=\ell+1}^{n} \mathcal{O}_{X, x} \frac{\partial}{\partial z_j}
.\]
Obviously, its restriction on $X \setminus D$ is identified with $\mathcal{T}_{X \setminus D}$.

Given an open subset $U \subset X$, now consider holomorphic sections $\alpha \in H^0(U, J_kX)$.
One says that $\alpha$ is a logarithmic $k$-jet field along $D$ if $\dif^{[j]} \omega(\alpha_{\restriction U'})$ are holomorphic for all $\omega \in \mathcal{T}_X^{\dual}(\log D)(U')$ and $0 \leq j \leq k-1$, for all open subset $U' \subset U$.
Here $\dif^{[j]} = \dif_{U', T_X}^{[j]}$ defined by \eqref{eqn:difform}, which is also well-defined for meromorphic $1$-forms $\omega$.
One denotes by $\mathcal{J}_k(X, -\log D)$ the sheaf, which locally consists of germs of these logarithmic $k$-jet fields $\alpha$, and calls it logarithmic $k$-jet sheaf along $D$.
By \cite{Noguchi1986} (see also Section 4.6.3 of \cite{Noguchi2014}), $\mathcal{J}_k(X, -\log D)$ is the sheaf of sections of a holomorphic fiber bundle $J_k(X, -\log D)$ over $X$.
More precisely, there is a fiber map
\[
\tau \colon J_k(X, -\log D) \longrightarrow J_kX
,\]
such that the induced map of sections
\[
\tau_{\ast} \colon H^0(U, J_k(X, -\log D)) \longrightarrow \mathcal{J}_k(X, -\log D)(U)
\]
is isomorphism for any open subset $U \subset X$.
One refers to $J_k(X, -\log D)$ as the logarithmic $k$-jet bundle of $(X, D)$.
Likewise, its restriction on $X \setminus D$ is identified with $J_k(X \setminus D)$.

Following \cite{Dethloff2001}, a logarithmic directed manifold is a triple $(X, D, V(-\log D))$ consisting of a logarithmic manifold $(X, D)$ equipped with a holomorphic subbundle $V(-\log D) \subset T_X(-\log D)$.
It is evident that $V^{\circ} \bydef V(-\log D)_{\restriction X \setminus D}$ is a holomorphic subbundle of $T_{X \setminus D}$.
Seeing that $J_kV^{\circ} \subset J_k(X \setminus D) \cong J_k(X, -\log D)_{\restriction X \setminus D} \subset J_k(X, -\log D)$, $J_kV^{\circ}$ can be think of as a subset of $J_k(X, -\log D)$.
Thereupon one defines logarithmic jet spaces $J_kV(-\log D)$ over $(X, D, V(-\log D))$ to be the (Euclidean) closure of $J_kV^{\circ}$ in $J_k(X, -\log D)$.

Indeed, the natural projection map
\begin{equation}\label{eqn:logjets}
p_{k, V\!(\!-\!\log \!D)} \colon J_kV(-\log D) \longrightarrow X
\end{equation}
gives a $\mathbb{C}^{rk}$-fiber bundle over $X$, where $r = \rank (V(-\log D))$.
Let $V \subset T_X$ be the closure of $V^{\circ}$ in $T_X$, then one sees that $V$ is also a holomorphic subbundle of $T_X$.
Hence $\dif_{U, V}^{[j]}$ is well-defined for every local sections $\omega \in \mathcal{V}^{\dual}\!(\log D)$, where $\mathcal{V}^{\dual}(\log D)$ is the dual of $\mathcal{V}(-\log D)$.
Moreover, for any local section $\alpha \in \mathcal{J}_kV(-\log D)$, one has $\dif_{U, V}^{[\ell]} (\res_{V\!(\!-\!\log \!D)} \omega)(\alpha) = \dif_{U, T_X}^{[\ell]} \omega(\alpha)$, where $\res_V \colon \mathcal{T}_X^{\dual}(\log D) \to \mathcal{V}^{\dual}(\log D)$ is a quotient morphism induced by the restriction morphism $\mathcal{O}(T_X(-\log D)) \to \mathcal{O}(V(-\log D))$.
Take an open subset $U \subset X$ such that there exist $r$ local sections $\omega_1, \dots, \omega_r \in \mathcal{V}^{\dual}\!(\log D)(U)$ generating $V^{\dual}\!(\log D) \subset T_X^{\dual}(\log D)$ at every point of $U$.
Analogue to \eqref{eqn:localizationJkV}, one has a local trivialization of $p_{k, V\!(\!-\!\log \!D)}^{-1}(U)$ given by
\begin{equation}
p_{k, V\!(\!-\!\log \!D)} \times (\dif_{U, V}^{[\ell]}\omega_i)_{\substack{1 \leq i \leq r \quad \\ 0 \leq \ell \leq k-1}} \colon p_{k, V\!(\!-\!\log \!D)}^{-1}(U) \longrightarrow U \times \mathbb{C}^{rk}
.
\end{equation}

\begin{remark}
In the absolute case $V(-\log D) = T_X(-\log D)$, one has a canonical choice of $\omega_i$ as follows:
\[
\omega_1 = \dif \log z_1, \, \dots, \, \omega_{\ell} = \dif \log z_{\ell}, \, \omega_{\ell+1} = \dif z_{\ell+1}, \, \dots, \, \omega_n = \dif z_n
,\]
where $z_1, \cdots, z_n$ is the local coordinate of $U$, given by the definition of the normal crossing divisor $D$.
In general, given a holomorphic subbundle $V(-\log D) \subset T_X(-\log D)$ with rank $r$, after permuting $z_1, \dots, z_{\ell}$ and $z_{\ell+1}, \dots, z_n$, respectively,
one has a system of generators $\omega_1, \dots, \omega_p, \omega_{\ell+1}, \dots, \omega_{\ell+q} \in \mathcal{V}^{\dual}\!(\log D)(U)$,
where $0 \leq p \leq \ell$, $0 \leq q \leq n-\ell$ and $p+q=r$.
\end{remark}

\subsection{Logarithmic Demailly-Semple jet towers}

In \cite{Dethloff2001}, Dethloff and Lu generalized Demailly-Semple jet tower to the logarithmic case.
Given a logarithmic directed manifold $(X, D, V(-\log D))$, let us start from $(X_0(D), D_0, V_0(-\log D_0)) = (X, D, V)$.
A sequence of logarithmic directed manifolds $(X_{\bullet}(D), D_{\bullet}, V_{\bullet}(-\log D_{\bullet}))$ is inductively constructed as follows.

Let $X_k(D) = \textnormal{P}(V_{k-1}(-\log D_{k-1})) \allowbreak \cong \mathbb{P}(V_{k-1}^{\dual}\!(\log D_{k-1}))$ with its natural projection $\bar{\pi}_k$ to $X_{k-1}(D)$, and let $D_k = \bar{\pi}_k^{-1} (D_{k-1})$, for each $k \in \mathbb{Z}_{+}$.
Observe that the differential of $\bar{\pi}_k$, $\dif \bar{\pi}_k \colon T_{X_k(D)} \to \bar{\pi}_k^{\ast} T_{X_{k-1}(D)}$ behaves well as a differential map $T_{X_k(D)}(-\log D_k) \to \bar{\pi}_k^{\ast} T_{X_{k-1}(D)}(-\log D_{k-1})$.
One defines
\[
V_k(-\log D_k) \bydef (\dif \bar{\pi}_k)^{-1} \mathcal{O}_{X_k(D)}(-1)
,\]
where $\mathcal{O}_{X_k(D)}(-1)$ is the tautological line bundle of $\textnormal{P}(V_{k-1}(-\log D_{k-1}))$ and regarded as the subbundle of $\bar{\pi}_k^{\ast} V_{k-1}(-\log D_{k-1})$.
More concretely, for every point $(x, [v]) \in X_k(D) = \textnormal{P}(V_{k-1}(-\log D_{k-1}))$ associated with a logarithmic tangent vector $v \in (V_{k-1}(-\log D_{k-1}))_x \setminus \{0\}$ for $x \in X_{k-1}$, one defines the logarithmic tangent subbundle $V_k(-\log D_k) \subset T_{X_k(D)}(-\log D_k)$ by
\[
(V_k(-\log D_k))_{(x, [v])} = \{ \xi \in (T_{X_k(D)}(-\log D_k))_{(x, [v])} \mid \dif \bar{\pi}_k (\xi) \in \mathbb{C}\,v \}
,\]
where $\mathbb{C}\,v \subset (V_{k-1}(-\log D_{k-1}))_x$.
Then the logarithmic directed manifold $(X_k(D), \allowbreak D_k, V_k(-\log D_k))$ is called the logarithmic Demailly-Semple $k$-jet tower of $(X, D, \allowbreak V(-\log D))$.
When no confusion can arise, one simply writes $(\widebar{X}_k, D_k, \widebar{V}_k)$ in what follows.

Next, we are going to show that the logarithmic Demailly-Semple jet tower has similar properties to the Demailly-Semple jet tower.

The bundle $\widebar{V}_k$ is characterized by the two exact sequences:
\vspace{-1.5ex}
\begin{equation*}
\begin{tikzcd}[column sep=scriptsize]
0 \arrow[r] & T_{\widebar{X}_k/\widebar{X}_{k-1}} \arrow[r] & \widebar{V}_k \arrow[r, "(\bar{\pi}_k)_{\ast}"] &[.5em] \mathcal{O}_{\widebar{X}_k}(-1) \arrow[r] & 0
,\end{tikzcd}
\end{equation*}
\vspace{-3ex}
\begin{equation*}
\begin{tikzcd}[column sep=scriptsize]
0 \arrow[r] & \mathcal{O}_{\widebar{X}_k} \arrow[r] & \bar{\pi}_k^{\ast} \widebar{V}_{k-1} \otimes \mathcal{O}_{\widebar{X}_k}(1) \arrow[r] & T_{\widebar{X}_k/\widebar{X}_{k-1}} \arrow[r] & 0
.\end{tikzcd}
\vspace{-.5ex}
\end{equation*}
Moreover, for each $k \geq 2$, the composition of vector bundle morphisms over $\widebar{X}_k$ 
\vspace{-.5ex}
\begin{equation*}
\begin{tikzcd}[column sep=scriptsize]
\mathcal{O}_{\widebar{X}_k}(-1) \arrow[r, hook] & \bar{\pi}_k^{\ast} \widebar{V}_{k-1} \arrow[r, "\bar{\pi}_k^{\ast} (\bar{\pi}_{k-1})_{\ast}"] &[2em] \bar{\pi}_k^{\ast} \mathcal{O}_{\widebar{X}_{k-1}}(-1)
\end{tikzcd}
\vspace{-.5ex}
\end{equation*}
admits $\widebar{\varGamma}_k \subset \widebar{X}_k$ as its zero divisor,
which is effective and satisfies that
\begin{equation}\label{eqn:O(barGamma)}
\mathcal{O}_{\widebar{X}_k}(1) = \bar{\pi}_k^{\ast} \mathcal{O}_{\widebar{X}_{k-1}}(1) \otimes \mathcal{O}(\widebar{\varGamma}_k)
.\end{equation}

For the sake of convenience, for each $0 \leq j \leq k-1$, one writes
\begin{equation}\label{eqn:barpi}
\bar{\pi}_{j, k} = \bar{\pi}_{j+1} \circ \bar{\pi}_j \circ \dots \circ \bar{\pi}_k \colon \widebar{X}_k \longrightarrow \widebar{X}_j
,\end{equation}
and
\[
\mathcal{O}_{\widebar{X}_k}(\bm{a}) \bydef \bar{\pi}_{1, k}^{\ast} \mathcal{O}_{\widebar{X}_1}(a_1) \otimes \bar{\pi}_{2, k}^{\ast} \mathcal{O}_{\widebar{X}_2}(a_2) \otimes \dots \otimes \mathcal{O}_{\widebar{X}_k}(a_k)
,\]
for every $\bm{a}=(a_1, \ldots, a_k) \in \mathbb{Z}^k$.
Likewise, it follows from \eqref{eqn:O(barGamma)} that
\begin{equation}\label{eqn:Obarab}
\mathcal{O}_{\widebar{X}_k}(\bm{a}) = \mathcal{O}_{\widebar{X}_k}(b_k) \otimes \mathcal{O}_{\widebar{X}_k}(-\bm{b} \cdot \widebar{\varGamma})
,\end{equation}
where
\begin{align*}
\bm{b} &= (b_1, \dots, b_k) \in \mathbb{Z}^k, \quad b_j = a_1+\dots+a_j, \\
\bm{b} &\cdot \widebar{\varGamma} = \sum_{1\leq j \leq k-1} b_j \bar{\pi}_{j+1, k}^{\ast} \widebar{\varGamma}_{j+1}
.\end{align*}

Furthermore, one defines
\begin{align*}
\widebar{X}_k^{\reg} &= \bigcap_{2\leq j \leq k} \bar{\pi}_{j, k}^{-1}(\widebar{X}_j \setminus \widebar{\varGamma}_j), \\
\widebar{X}_k^{\sing} &= \bigcup_{2\leq j \leq k} \bar{\pi}_{j, k}^{-1}(\widebar{\varGamma}_j) = \widebar{X}_k \setminus \widebar{X}_k^{\reg} 
.\end{align*}

\subsection{Logarithmic jet differentials}

Let $(X, D, V(-\log D))$ be a logarithmic directed manifold, and take an open subset $U \subset X$. A local meromorphic $k$-jet differential $Q$ on $U$ is called a logarithmic $k$-jet differential,
if $Q(\alpha)$ is holomorphic for any logarithmic $k$-jet field $\alpha \in \mathcal{J}_kV(-\log D)(U)$,
where $\mathcal{J}_kV(-\log D) \subset \mathcal{J}_k(X, -\log D)$ is the sheaf of sections of $J_kV(-\log D)$ over $X$.
It amounts to saying that $\tau^{\ast} Q$ is a homogeneous polynomial on $p_{k, V\!(\!-\!\log \!D)}^{-1} (U)$.
One denotes the sheaf of logarithmic $k$-jet differential by $\mathcal{E}_{k, \bullet}^{GG}V^{\dual}\!(\log D)$.
There is a natural splitting
\[
\mathcal{E}_{k, \bullet}^{GG}V^{\dual}\!(\log D) = \bigoplus_{m \geq 0} \mathcal{E}_{k, m}^{GG}V^{\dual}\!(\log D)
,\]
where $\mathcal{E}_{k, m}^{GG}V^{\dual}\!(\log D)$ is locally defined by
\[
\mathcal{E}_{k, m}^{GG}V^{\dual}\!(\log D)(U) \bydef \{Q \in \mathcal{E}_{k, \bullet}^{GG}V^{\dual}\!(\log D)(U) \mid Q_{\restriction U \setminus D} \in \mathcal{E}_{k, m}^{GG}V^{\dual}(U \setminus D)\}
,\]
called logarithmic Green-Griffiths jet differential sheaf of order $k$ and weighted degree $m$.
Likewise, logarithmic invariant jet differential sheaf $\mathcal{E}_{k, m}V^{\dual}\!(\log D)$ of order $k$ and weighted degree $m$ is locally defined by
\[
\mathcal{E}_{k, m}V^{\dual}\!(\log D)(U) \bydef \{Q \in \mathcal{E}_{k, \bullet}^{GG}V^{\dual}\!(\log D)(U) \mid Q_{\restriction U \setminus D} \in \mathcal{E}_{k, m}V^{\dual}(U \setminus D)\}
.\]
More precisely, $\mathcal{E}_{k, m}V^{\dual}\!(\log D)$ is a locally free sheaf, which consists of logarithmic $k$-jet differentials $Q \in \mathcal{E}_{k, \bullet}^{GG}V^{\dual}\!(\log D)(U)$ such that
\[
Q(\varphi \cdot j_kf) = \varphi'(0)^m Q(j_kf)
,\]
for any $\varphi \in \mathbb{G}_k$ and any $j_kf \in p_{k, V^{\circ}}^{-1}(U \setminus D) \cap J_kV^{\circ \reg}$.

What's more, we define $J_kV(-\log D)^{\sing}$ the closure of $J_kV^{\circ\sing}$ in $J_kV(-\log D)$, and $J_kV(-\log D)^{\reg} = J_kV(-\log D) \setminus J_kV(-\log D)^{\sing}$, where $J_kV^{\circ\sing} \bydef J_kV^{\circ} \setminus J_kV^{\circ\reg}$ is the set of singular $k$-jets over $(X \setminus D, V^{\circ})$.

\begin{theorem} 
Suppose that $(X, D, V(-\log D))$ is a logarithmic directed manifold such that $V(-\log D)$ has rank $r \geq 2$. Let $\bar{\pi}_{0, k} \colon X_k(D) \to X$ be a natural projection defined by \eqref{eqn:barpi}.
\begin{enumerate}[i)]
\item The quotient $J_kV(-\log D)^{\reg} / \mathbb{G}_k$ has the structure of a locally trivial bundle over $X$, and there is a holomorphic embedding $J_kV(-\log D)^{\reg} / \mathbb{G}_k \hookrightarrow X_k(D)$ over $X$, which identifies $J_kV(-\log D)^{\reg} / \mathbb{G}_k$ with $X_k(D)^{\reg}$. That is to say,
\[
J_kV(-\log D)^{\reg} / \mathbb{G}_k \cong X_k(D)^{\reg}
.\]
\item The direct image sheaf $(\bar{\pi}_{0,k})_{\ast}\mathcal{O}_{X_k(D)}(m)$ is identified with the sheaf of logarithmic invariant jet differentials $\mathcal{E}_{k, m}V^{\dual}\!(\log D)$, namely,
\begin{equation}\label{eqn:EVlogD}
(\bar{\pi}_{0,k})_{\ast}\mathcal{O}_{X_k(D)}(m) \cong \mathcal{E}_{k,m}V^{\dual}\!(\log D)
.\end{equation}
\item For every $m>0$, the relative base locus of the linear system $|\mathcal{O}_{X_k(D)}(m)|$ lies in the set $X_k(D)^{\sing} \cup \bar{\pi}_{0,k}^{-1}(D)$. Moreover, $\mathcal{O}_{X_k(D)}(1)$ is relatively big over $X$.
\end{enumerate}
\end{theorem}

\begin{proof}
\romannumeral1) See Proposition 3.8, c) in \cite{Dethloff2001}.\\
\romannumeral2) See Proposition 3.9 in \cite{Dethloff2001}.\\
\romannumeral3) Let $U$ be an affine open subset of $X$.
By definition, every $Q \in \mathcal{E}_{k,m}V^{\dual}\!(\log D)(U)$ is a local meromorphic section of $E_{k,m}V^{\dual}(U)$, which is holomorphic on $U \setminus D$.
Conversely, for each $Q \in \mathcal{E}_{k,m}V^{\dual}(U)$, we see that $\tau^{\ast} Q$ is indeed a local section of $E_{k,m}V^{\dual}\!(\log D)$, which vanishes on $D$.
Seeing that $X_k(D)^{\sing} \cap \bar{\pi}_{0,k}^{-1}(X \setminus D) = (X \setminus D)_k^{\sing}$, the result follows from Theorem \ref{thm:DS}, \ref{itm:relbs}).
\end{proof}

\section{Logarithmic Wronskians}\label{sec:logWronskians}
In \cite{Brotbek2019}, Brotbek and Deng generalize the Wronskian sections and the Wronskian ideal sheaf to the logarithmic case.
In this section, we are going to further generalize these concepts.

\subsection{Logarithmic Wronskian sections}
Let $(X, D, V(-\log D))$ be a logarithmic directed manifold, and take an open subset $U \subset X$.
For each $k \geq 0$, Brotbek and Deng \cite{Brotbek2019} define a $\mathbb{C}$-linear map $\nabla_D^{[k]} \colon \mathcal{O}_X(D)(U) \to \mathcal{E}_{k, k}^{GG}T_X^{\dual}\!(\log D)(U)$ given by
\[
\nabla_D^{[\ell]}\sigma = \sigma_D \cdot \dif_{U, T_X}^{[\ell]}\!\left(\frac{\sigma}{\sigma_D\!}\right)
,\]
where $\sigma_D$ is a canonical section of $\mathcal{O}_X(D)$.
Hence we construct the logarithmic Wronskian of a local holomorphic function $\sigma_0 \in \mathcal{O}(U)$ and $k$ local holomorphic sections $\sigma_1, \dots, \sigma_k \in H^0(U, \mathcal{O}_X(D))$ as follows:
\begin{align*}
&W_{D, U}^{V\!(\!-\!\log \!D)}(\sigma_0; \sigma_1, \dots, \sigma_k) \\
\bydef &
\renewcommand{\arraystretch}{1.5}
\begin{vmatrix}
\dif_{U, V}^{[0]}\sigma_0 & \nabla_D^{[0]}\sigma_1 & \cdots & \nabla_D^{[0]}\sigma_k \\
\dif_{U, V}^{[1]}\sigma_0 & \nabla_D^{[1]}\sigma_1 & \cdots & \nabla_D^{[1]}\sigma_k \\
\vdots & \vdots & \ddots & \vdots \\
\dif_{U, V}^{[k]}\sigma_0 & \nabla_D^{[k]}\sigma_1 & \cdots & \nabla_D^{[k]}\sigma_k
\end{vmatrix}
\in \mathcal{E}_{k, k'}V^{\dual}\!(\log D) \otimes \mathcal{O}_X(kD)(U)
,\end{align*}
where $V \subset T_X$ is the (Euclidean) closure of $V^{\circ}$ in $T_X$ and we abusively write
\[
\nabla_D^{[\ell]} \bydef \res_{k, V\!(\!-\!\log \!D)} \circ \nabla_{D \restriction U}^{[\ell]} = \sigma_D \cdot \dif_{U, V}^{[\ell]}\!\left(\frac{\sigma}{\sigma_D\!}\right)
.\]
Here $\res_{k, V\!(\!-\!\log \!D)} \colon \mathcal{E}_{k, \bullet}^{GG}T_X^{\dual}\!(\log D) \to \mathcal{E}_{k, \bullet}^{GG}V^{\dual}\!(\log D)$ is a quotient morphism induced by the restriction morphism $\mathcal{O}(J_k(X, -\log D)) \to \mathcal{O}(J_kV(-\log D))$.
Moreover, we see that
\[
W_{D, U}^{V\!(\!-\!\log \!D)}(\sigma_0; \sigma_1, \dots, \sigma_k) = \res_{k, V\!(\!-\!\log \!D)} \left(W_{D, U}^{T\!_X\!(\!-\!\log \!D)}(\sigma_0; \sigma_1, \dots, \sigma_k)\right)
.\]
When no confusion can arise, we hereinafter prefer to write $W_{D, U}(\sigma_0; \sigma_1, \dots, \sigma_k)$ rather than $W_{D, U}^{V\!(\!-\!\log \!D)} \allowbreak (\sigma_0; \sigma_1, \dots, \sigma_k)$.

Furthermore, our construction of the logarithmic Wronskian is related to the Wronskian in the following way.
\begin{lemma}\label{lem:sigmaW_D}
Let $(X, D, V(-\log D))$ be a logarithmic directed manifold, and take an open subset $U \subset X$.
For any $\sigma_D \in \mathcal{O}(U)$ and any $\sigma_0, \dots, \sigma_k \in H^0(U, \mathcal{O}_X(D))$, we have
\[
W_U^V(\sigma_D \cdot \sigma_0, \sigma_1, \dots, \sigma_k) = \sigma_D \cdot W_{D, U}^{V\!(\!-\!\log \!D)}(\sigma_0; \sigma_1, \dots, \sigma_k)
,\]
where $V \subset T_X$ is the closure of $V^{\circ}$ in $T_X$.
\end{lemma}
\begin{proof}
For any $\sigma \in H^0(U, \mathcal{O}_X(D))$, using the Leibniz rule to $\sigma_D \cdot \dfrac{\sigma}{\sigma_D\!}$, we get
\[
\dif_{U, V}^{[\ell]} \sigma = \sum_{i=0}^{\ell} \binom{k}{i} \frac{\dif_{U, V}^{[\ell-i]} \sigma_D}{\sigma_D\!} \nabla_D^{[i]} \sigma
.\]
Then it is evident that

\begin{align*}
W_U^V(\sigma_D \cdot \sigma_0, \sigma_1, \dots, \sigma_k) &=
\renewcommand{\arraystretch}{1.5}
\begin{vmatrix}
\nabla_D^{[0]}(\sigma_D \cdot \sigma_0) & \nabla_D^{[0]}\sigma_1 & \cdots & \nabla_D^{[0]}\sigma_k \\
\nabla_D^{[1]}(\sigma_D \cdot \sigma_0) & \nabla_D^{[1]}\sigma_1 & \cdots & \nabla_D^{[1]}\sigma_k \\
\vdots & \vdots & \ddots & \vdots \\
\nabla_D^{[k]}(\sigma_D \cdot \sigma_0) & \nabla_D^{[k]}\sigma_1 & \cdots & \nabla_D^{[k]}\sigma_k
\end{vmatrix} \\
&= \sigma_D \cdot W_{D, U}^{V\!(\!-\!\log \!D)}(\sigma_0; \sigma_1, \dots, \sigma_k)
.\qedhere\end{align*}
\end{proof}

For a holomorphic line bundle $L$ on $X$, take an open subset $U$, on which both $L \otimes \mathcal{O}_X(-D)$ and $\mathcal{O}_X(D)$ can be trivialized. 
Choose a trivialization of $L \otimes \mathcal{O}_X(-D)_{\restriction U}$ and a trivialization of $\mathcal{O}_X(D)_{\restriction U}$.
Observe that the trivialization of $L \otimes \mathcal{O}_X(-D)_{\restriction U}$ induces a map $H^0(X, L) \to H^0(U, \mathcal{O}_X(D))$,
which associates $\sigma_U^D = \left( \dfrac{\sigma}{\sigma_D\!} \right)_{U \setminus D} \!\! \sigma_D$ to a global section $\sigma \in H^0(X, L)$,
where $\left( \dfrac{\sigma}{\sigma_D\!} \right)_{U \setminus D} \!\!\in \mathcal{O}(U \setminus D)$ is a meromorphic function on $U$ corresponding to $\dfrac{\sigma}{\sigma_D\!} \in H^0(X \setminus D, L \otimes \mathcal{O}_X(-D))$ under the trivialization of $L \otimes \mathcal{O}_X(-D)_{\restriction U}$.
Given global sections $\sigma_0 \in H^0(X, L \otimes \mathcal{O}_X(-D))$ and $\sigma_1, \dots, \sigma_k \in H^0(X, L)$, we thus obtain a locally defined invariant jet differential
\[
W_{D, U}(\sigma_0; \sigma_1, \dots, \sigma_k) \bydef W_U(\sigma_{0, U}; \sigma_{1, U}^D, \dots, \sigma_{k, U}^D) \in \mathcal{E}_{k, k'}V^{\dual}\!(\log D)(U)
.\]
It follows from Lemma \ref{lem:sigmaW_D} that
\begin{equation}\label{eqn:sigmaW_D}
W_{D, U}(\sigma_0; \sigma_1, \dots, \sigma_k) = \sigma_D^{-1} \cdot W_U(\sigma_D \cdot \sigma_0, \sigma_1, \dots, \sigma_k)
,\end{equation}
where the Wronskian on the right-hand side is under the trivialization of $L_{\restriction U}$ given by the composition of the trivialization of $L \otimes \mathcal{O}_X(-D)_{\restriction U}$ and $\mathcal{O}_X(D)_{\restriction U}$.
In addition, these logarithmic Wronskians glue together into a global section
\[
W_D(\sigma_0; \sigma_1, \dots, \sigma_k) \in H^0(X, E_{k, k'}V^{\dual}\!(\log D) \otimes L^{k+1} \otimes \mathcal{O}_X(-D))
.\]

What's more, from \eqref{eqn:sigmaW_D}, we easily get the following proposition.
\begin{proposition}\label{prp:sigmaW_D}
Let $(X, D, V(-\log D))$ be a logarithmic directed manifold.
For any $\sigma_0 \in H^0(X, L \otimes \mathcal{O}_X(-D))$ and any $\sigma_1, \dots, \sigma_k \in H^0(X, L)$, one has
\[
W(\sigma_D \cdot \sigma_0, \sigma_1, \dots, \sigma_k) = \sigma_D \cdot W_D(\sigma_0; \sigma_1, \dots, \sigma_k)
.\]
\end{proposition}

Thanks to \eqref{eqn:EVlogD}, there is a global section
\[
\omega_D(\sigma_0; \sigma_1, \dots, \sigma_k) \in H^0(X_k(D), \mathcal{O}_{X_k(D)}(k') \otimes \bar{\pi}_{0, k}^{\ast} (L^{k+1} \otimes \mathcal{O}_X(-D)))
\]
corresponding to $W_D(\sigma_0; \sigma_1, \dots, \sigma_k)$.
That is to say,
\[
(\bar{\pi}_{0, k})_{\ast} \omega_D(\sigma_0; \sigma_1, \dots, \sigma_k) = W_D(\sigma_0; \sigma_1, \dots, \sigma_k)
.\]

\begin{remark}
In the case where $L = \mathcal{O}_X(D)$ and $\sigma_0 = 1 \in H^0(X, \mathcal{O}_X)$, our construction of the logarithmic Wronskian is identified with the construction in \cite{Brotbek2019}.
\end{remark}

\subsection{Logarithmic Wronskian ideal sheaf}
Let $X$ be a projective manifold, and take an effective divisor $D$ on $X$.
For a linear system $\mathfrak{S} \subset |L|$ on $X$, we define
\[
\mathfrak{S}_{\setminus D} \bydef \{D' \in |L \otimes \mathcal{O}_X(-D)| \mid D+D' \in \mathfrak{S}\}
.\]

Take a logarithmic directed projective manifold $(X, D, V(-\log D))$, and choose a linear system $\mathfrak{S} \subset |L|$ on $X$ such that $\mathfrak{S}_{\setminus D}$ is non-empty.
Set
\[
\mathbb{W}(X_k(D), \mathfrak{S}) \bydef \Span \{\omega_D(\sigma_0; \sigma_1, \dots, \sigma_k) \mid \sigma_0 \in \mathbb{S}_{\setminus D}, \sigma_1, \dots, \sigma_k \in \mathbb{S}\}
,\]
which is a linear subspace of $H^0(X_k(D), \mathcal{O}_{X_k(D)}(k') \otimes \bar{\pi}_{0, k}^{\ast} (L^{k+1} \otimes \mathcal{O}_X(-D)))$,
where $\mathbb{S}_{\setminus D} \subset H^0(X, L \otimes \mathcal{O}_X(-D))$ is the linear subspace corresponding to $\mathfrak{S}_{\setminus D}$.
We thus denote by $\mathfrak{W}(X_k(D), \mathfrak{S})$ the linear system corresponding to $\mathbb{W}(X_k(D), \mathfrak{S})$.
And we refer to its base ideal $\mathfrak{w}(X_k(D), \mathfrak{S})$ as the $k$-th logarithmic Wronskian ideal sheaf of $\mathfrak{S}$.

Thanks to Proposition \ref{prp:sigmaW_D}, given a linear system $\mathfrak{S} \subset |L|$ as above, there is a inclusion
\begin{equation*}
\hspace{-.5eM}
\begin{tikzcd}[column sep=scriptsize]
(\bar{\pi}_{0, k})_{\ast} \mathbb{W}(X_k(D), \mathfrak{S}) \arrow[r, "\sigma_D \cdot"] & H^0(X, E_{k, k'}V^{\dual} \!\otimes L^{k+1}) \arrow[r, "\cong"] & H^0(X_k, \mathcal{O}_{X_k}(k') \otimes \pi_{0, k}^{\ast} L^{k+1})
.\end{tikzcd}
\end{equation*}
Then we denote its image by
\[
\mathbb{W}(X_k, D, \mathfrak{S}) = \Span \{\omega(\sigma_0, \sigma_1, \dots, \sigma_k) \mid \sigma_0 \in \mathbb{S}_{\supset D}, \, \sigma_1, \dots, \sigma_k \in \mathbb{S}\}
,\]
where $\mathbb{S}_{\supset D} \bydef \sigma_D \cdot \mathbb{S}_{\setminus D} = \{\sigma \in \mathbb{S} \mid (\sigma) \geq D\} \cup \{0\}$.
Hence we denote the corresponding linear system by $\mathfrak{W}(X_k, D, \mathfrak{S})$, and denote the base ideal of $\mathfrak{W}(X_k, D, \mathfrak{S})$ by $\mathfrak{w}(X_k, D, \mathfrak{S})$.

More generally, for a closed subscheme $Z$ of $X$, take a linear system $\mathfrak{S} \subset |L|$ on $X$ such that $\mathfrak{S}_{\supset Z} \bydef \{D' \in \mathfrak{S} \mid Z \subset D'\}$ is non-empty.
Then we introduce the transparent notation for a linear subspace of $H^0(X_k, \mathcal{O}_{X_k}(k') \otimes \pi_{0, k}^{\ast} L^{k+1})$
\[
\mathbb{W}(X_k, Z, \mathfrak{S}) \bydef \Span \{\omega(\sigma_0, \sigma_1, \dots, \sigma_k) \mid \sigma_0 \in \mathbb{S}_{\supset Z}, \, \sigma_1, \dots, \sigma_k \in \mathbb{S}\}
,\]
where $\mathbb{S}_{\supset Z} \bydef \{\sigma \in \mathbb{S} \mid Z \subset (\sigma)\}$ is the linear subspace corresponding to $\mathfrak{S}_{\supset Z}$.
We denote the corresponding linear system by $\mathfrak{W}(X_k, Z, \mathfrak{S})$.
Therefore, we call its base ideal $\mathfrak{w}(X_k, Z, \mathfrak{S})$ the $k$-th Wronskian ideal sheaf of $\mathfrak{S}$ along $Z$.

Moreover, set
\begin{equation}\label{eqn:WXkDS}
\widebar{\mathbb W}(X_k, Z, \mathfrak{S}) \bydef \Span \{\widebar{\omega}(\sigma_0, \sigma_1, \dots, \sigma_k) \mid \sigma_0 \in \mathbb{S}_{\supset Z}, \, \sigma_1, \dots, \sigma_k \in \mathbb{S}\}
.\end{equation}
Likewise, we denote the corresponding linear system by $\widebar{\mathfrak W}(X_k, Z, \mathfrak{S})$, and denote the base ideal of $\widebar{\mathfrak W}(X_k, Z, \mathfrak{S})$ by $\widebar{\mathfrak w}(X_k, Z, \mathfrak{S})$.

\section{Preliminaries of Nevanlinna theory}\label{sec:NT}

\subsection{Weil functions}

First, we recall Weil functions $\lambda_Z$ for each closed subscheme $Z$ of a (complex) irreducible projective variety $X$, which is first introduced by Silverman \cite{Silverman1987} in Diophantine geometry.

Given a (complex) irreducible quasi-projective variety $X$, let $\mathcal{C}(X)$ be the set of all real-valued continuous functions defined on some Zariski open subsets of $X$,
that is to say,
\[
\mathcal{C}(X) \bydef \{\lambda \colon U \to \mathbb{R} \mid \lambda \text{ is continuous, }U \text{ is a Zariski open subset of }X\}
.\]

\begin{definition}
Let $X$ be as above, and let $\lambda_1, \lambda_2 \in \mathcal{C}(X)$.
We write $\lambda_1 \leq_{\langle X \rangle} \lambda_2$ if there is a continuous function $\alpha$ defined on $X$ and satisfies that $\lambda_1 \leq \lambda_2 + \alpha$ on some non-empty Zariski open subset $U$ where $\lambda_1$ and $\lambda_2$ are both defined.
Moreover, we write $\lambda_1 =_{\langle X \rangle} \lambda_2$, if $\lambda_1 \leq_{\langle X \rangle} \lambda_2$ and $\lambda_2 \leq_{\langle X \rangle} \lambda_1$.
\end{definition}
Notice that $=_{\langle X \rangle}$ defines the equivalent relation on the set $\mathcal{C}(X)$, and $\leq_{\langle X \rangle}$ defines the order relation on the set $\mathcal{C}(X)/=_{\langle X \rangle}$.
Furthermore, if $\kappa_1, \kappa_2, \lambda_1, \lambda_2 \in \mathcal{C}(X)$ such that $\kappa_1 =_{\langle X \rangle} \kappa_2$ and $\lambda_1 =_{\langle X \rangle} \lambda_2$, we check that
\begin{align*}
-\lambda_1 &=_{\langle X \rangle} -\lambda_2, \\
\kappa_1+\lambda_1 &=_{\langle X \rangle} \kappa_2+\lambda_2, \\
\min \{\kappa_1, \lambda_1\} &=_{\langle X \rangle} \min \{\kappa_2, \lambda_2\}, \\
\max \{\kappa_1, \lambda_1\} &=_{\langle X \rangle} \max \{\kappa_1, \lambda_1\}
.\end{align*}

\begin{definition}\label{def:lambda_Z}
Let $Z \subset X$ be a closed subscheme. A Weil function $\lambda_{Z}$ for $Z$ is a continuous function $\lambda_{Z}: X \setminus \Supp Z \to \mathbb{R}$ satisfying that there exist an affine Zariski open covering $\{U_{\epsilon}\}_{\epsilon \in \Upsilon}$ of $X$ and a system of generators $f_{1}^{\epsilon}, \ldots, f_{r_{\epsilon}}^{\epsilon} \in \mathcal{I}_{Z}(U_{\epsilon}) \subset \mathcal{O}_X(U_{\epsilon})$ such that
\begin{equation}\label{eqn:lambda_Z}
\lambda_{Z}(x) =_{\langle U_{\epsilon} \rangle} -\log \mathop{\max}_{1 \leq i \leq r_{\epsilon}} |f_i^{\epsilon}(x)|
.\end{equation}
\end{definition}
\begin{remark}\label{rmk:pnorm}
Observe the fact that $\infty$-norm is equivalent to $p$-norm, for each integer $p \geq 1$.
More precisely,
\[
\log \mathop{\max}_{1 \leq i \leq r} |\alpha_i| =_{\langle U \rangle} \frac{1}{p} \log \left( \sum_{i=1}^r |\alpha_i|^p \right)
,\]
for any Zariski open subset $U \subset X$.
Thus \eqref{eqn:lambda_Z} can be also written as
\[
\lambda_{Z}(x) =_{\langle U_{\epsilon} \rangle} - \frac{1}{p} \log \left( \sum_{i=1}^{r_{\epsilon}} |f_i^{\epsilon}(x)|^p \right)
.\]
\end{remark}

Let $Z(X)$ denote the set of closed subschemes of $X$.
Indeed, the definition above gives a map from $Z(X)$ to $\mathcal{C}(X)/=_{\langle X \rangle}$
\[
\lambda \colon Z(X) \longrightarrow \mathcal{C}(X)/=_{\langle X \rangle}
.\]
Notice that closed subschemes $Z \in Z(X)$ are in one-to-one correspondence with ideal sheaves $\mathcal{I}_Z \subset \mathcal{O}_X$ (see Chapter \uppercase\expandafter{\romannumeral2}, Section 5.9 of \cite{Hartshorne1977}).
We will often implicitly identify a closed subscheme $Z$ with its ideal sheaf $\mathcal{I}_Z$.

Next, we need to recall the definitions of a number of natural operations on $Z(X)$, before going to introduce some properties of Weil functions $\lambda_Z$.
\begin{definition}\label{def:I_Z}
Let $Z_1, Z_2$ be closed subschemes of $X$.
\begin{enumerate}[i)]
\item\label{itm:Icap} The intersection $Z_1 \cap Z_2$ of $Z_1$ and $Z_2$ is the closed subscheme of $X$ with the ideal sheaf $\mathcal{I}_{Z_1 \cap Z_2} = \mathcal{I}_{Z_1} + \mathcal{I}_{Z_2}$.
\item The sum $Z_1 + Z_2$ of $Z_1$ and $Z_2$ is the closed subscheme of $X$ with the ideal sheaf $\mathcal{I}_{Z_1 + Z_2} = \mathcal{I}_{Z_1} \cdot \mathcal{I}_{Z_2}$.
\item We say that $Z_1$ is contained in $Z_2$, i.e., $Z_1 \subset Z_2$, if and only if $\mathcal{I}_{Z_2} \subset \mathcal{I}_{Z_1}$.
\item Let $\varphi \colon X \to Y$ be a morphism of varieties, and let $Z$ be a closed subscheme of $Y$. The inverse image $\varphi^{\ast} Z$ is the closed subscheme of $X$ with the ideal sheaf $\mathcal{I}_{\varphi^{\ast} Z} =\varphi^{-1} \mathcal{I}_Z \cdot \mathcal{O}_X$, where $\varphi^{-1} \mathcal{I}_Z \cdot \mathcal{O}_X$ is the image in $\mathcal{O}_X$ of the sheaf $\varphi^{\ast} \mathcal{I}_Z \subset \varphi^{\ast} \mathcal{O}_Y$.
\end{enumerate}
\end{definition}

Hence the map $\lambda$ satisfies the following properties.
\begin{proposition}\cite{Yamanoi2004}\label{prp:lambda_Z}
Let $Z_1, Z_2$ be closed subschemes of $X$. Then
\begin{enumerate}[i)]
\item\label{itm:lambdacap} $\lambda_{Z_1 \cap Z_2} =_{\langle X \rangle} \min \{\lambda_{Z_1}, \lambda_{Z_2}\}$.
\item\label{itm:lambdasum} $\lambda_{Z_1+Z_2} =_{\langle X \rangle} \lambda_{Z_1}+\lambda_{Z_2}$.
\item\label{itm:lambdasubset} If $Z_1 \subset Z_2$, then $\lambda_{Z_1} \leq_{\langle X \rangle} \lambda_{Z_2}$.
\item If $\Supp Z_1 \subset \Supp Z_2$, then there exists a constant $c>0$ such that $\lambda_{Z_1} \leq_{\langle X \rangle} c \lambda_{Z_2}$.
\item\label{itm:lambdapullback} If $\varphi \colon X \to Y$ is a morphism of varieties, then we have $\lambda_{\varphi^{\ast} Z} =_{\langle X \rangle} \lambda_{Z} \circ \varphi$ for a closed subscheme $Z$ of $Y$.
\end{enumerate}
\end{proposition}

What's more, think of an effective Cartier divisor $D$ on $X$ as a locally principal closed subscheme of $X$.
Assume that $X$ as above is non-singular besides, and let $\sigma$ be a canonical section of $\mathcal{O}(D)$, namely, $(\sigma) = D$.
Choose a continuous Hermitian metric $h$ on the corresponding line bundle $\mathcal{O}(D)$ with the weight function $\varphi$.
Take an affine Zariski open subset $U \subset X$ such that $\mathcal{O}(D)_{\restriction U}$ can be trivialized.
Thanks to the continuity of $\varphi$ on $U$, if $\sigma_U$ is associated to $\sigma$ under a trivialization, then
\[
\log |\sigma_U| =_{\langle U \rangle} \log \|\sigma\|_h
.\]
Obviously, all such open subset $U$ give an affine Zariski open covering of $X$.
Therefore, a Weil function $\lambda_D$ for $D$ can be also defined as a real-valued continuous function on $X \setminus \Supp D$ satisfying
\begin{equation}\label{eqn:lambda_D}
\lambda_D =_{\langle X \rangle} - \log \|\sigma\|_h
.\end{equation}
\begin{remark}
More generally, if $D$ is a $\mathbb{Q}$-Cartier divisor and $\ell D$ is an effective Cartier divisor, then we can define $\lambda_D \bydef \dfrac{1}{\ell} \lambda_{\ell D}$.
This definition is well-defined on account of Proposition \ref{prp:lambda_Z}, \ref{itm:lambdasum}).
\end{remark}

\subsection{Notation and Terminology in Nevanlinna theory}\label{ssc:NTnotation}

From now on, we are going to prepare the notation and some well known results in Nevanlinna theory.

On a complex manifold $X$, the exterior derivative $\dif$ takes a complex function and yields a complex $1$-form.
It decomposes as $\dif = \partial+\bar{\partial}$ as complex $1$-forms decompose into $(1, 0)$-forms and $(0, 1)$-forms.
More concretely, take an open subset $U \subset X$ with a local coordinate system $z_1, \dots, z_n$, and we will write
\[
\partial f = \sum_{j=1}^n \frac{\partial f}{\partial z_j} \dif z_j, \qquad
\bar{\partial} f = \sum_{j=1}^n \frac{\partial f}{\partial \bar{z}_j} \dif \bar{z}_j
,\]
for a complex function $f$ on $U$.
These operators can be automatically extended to forms of higher degree (see Chapter 0, Section 2.3 of \cite{Griffiths1978} for details).
By convention, set $\dif^c = \dfrac{i}{4\pi}(\bar{\partial}-\partial)$.
We have $\ddif^c = \dfrac{i}{2\pi} \partial \bar{\partial}$.

Let us recall the singular Hermitian metric and the curvature current of a holomorphic line bundle.
\begin{definition}\cite{Demailly1992}
Let $L \to X$ be a holomorphic line bundle over a complex manifold $X$.
We say that $h$ is a singular metric on $L$ if for any local trivialization $L_{\restriction U} \cong U \times \mathbb{C}$ of $L$, the metric is given by
\begin{equation}\label{eqn:metrich}
\|\xi\|_h^2 \bydef |\xi|^2 e^{-\varphi}
\end{equation}
where $\varphi \in L^1_{\loc}(U)$ is a real valued function, called the weight function of the metric $h$ with respect to the such local trivialization.
We say that $h$ admits a closed subset $\Sing(h) \subset X$ as its degeneration set if $\varphi$ is locally bounded on $X \setminus \Sing(h)$ and is unbounded on a neighborhood of any point of $\Sing(h)$.
\end{definition}
If $\varphi'$ is the weight function of the metric $h$ with respect to another local trivialization, then we see that $\varphi' = \varphi + \log |g|$ where $g$ is the transition function between two local trivializations.
The curvature current of $L$ is then defined to be the closed $(1, 1)$-current $\ddif^c [\varphi]$, which is independent of choice of local trivializations.
And the assumption $\varphi \in L^1_{\loc}$ guarantees that $\ddif^c [\varphi]$ exists in the sense of distributions.

We will always assume here that the weight function $\varphi$ is quasi plurisubharmonic.
Recall that a function $\varphi$ is said to be quasi plurisubharmonic if $\varphi$ is locally the sum of a plurisubharmonic function and of a smooth function (so that in particular $\varphi \in L^1_{\loc}$).
Then the curvature current $\ddif^c [\varphi]$ is well-defined as a current and is locally bounded from below by a negative $(1, 1)$-form with constant coefficients.
In addition, $\ddif^c [\varphi]$ is in the first Chern class $c_1(L, h)$.

\begin{example}
Take a linear system $\mathfrak{S} \subset |L|$ on a projective complex manifold $X$.
One can define a natural (possibly singular) Hermitian metric $h_{\mathfrak S}$ on $L$, whose weight function $\varphi_{\mathfrak S}$ is given by
\begin{equation}\label{eqn:phi_S}
\varphi_{\mathfrak S} = \log \left(\sum_{i=0}^N |\varsigma_i|^2\right)
,\end{equation}
where $N = \dim \mathfrak{S}$, $\{\varsigma_i \mid i = 0, 1, \dots, N\}$ is a basis of $\mathbb{S}$, and $\mathbb{S} \subset H^0(X, L)$ is the linear subspace corresponding to $\mathfrak{S}$.
Obviously, $\Sing(h_{\mathfrak S}) = \Bs(\mathfrak{S})$.
Accordingly, the metric $h_{\mathfrak S}$ is smooth if and only if $\mathfrak{S}$ is base point free.
\end{example}

Observe that $|\sigma|^2 = e^{\varphi_{\mathfrak S}}$ in the particular case where $\mathfrak{S} = \{D\} \subset |D|$.
Taking logarithms on both sides of \eqref{eqn:metrich}, we immediately get
\[
- \log \|\sigma_D\|_h = \frac{1}{2}(\varphi - \varphi_{\mathfrak S})
,\]
for any (possibly singular) Hermitian metric $h$ on $\mathcal{O}(D)$ with the weighted function $\varphi$.
And it gives a Weil function for $D$ on account of \eqref{eqn:lambda_D}, if $h$ is a continuous Hermitian metric.

More generally, take a linear system $\mathfrak{S} \subset |L|$ on a projective complex manifold $X$.
Let us introduce the transparent notation $Z_{\mathfrak s}$ for the closed subscheme of $X$ with the ideal sheaf $\mathcal{I}_{Z_{\mathfrak s}} = \mathfrak{s}$, where $\mathfrak{s}$ is the base ideal of $\mathfrak{S}$.
We say that $Z_{\mathfrak s}$ is the scheme of base points of $\mathfrak{S}$.
Take a basis $\{\varsigma_i \mid i = 0, 1, \dots, N\}$ of $\mathbb{S}$, where $N = \dim \mathfrak{S}$, $\mathbb{S} \subset H^0(X, L)$ is the linear subspace corresponding to $\mathfrak{S}$.
It is evident that
\[
Z_{\mathfrak s} = \bigcap_{i=0}^N D_i
,\]
where $D_i = (\varsigma_i)$ for each $0 \leq i \leq N$.
According to Proposition \ref{prp:lambda_Z}, \ref{itm:lambdacap}) and Remark \ref{rmk:pnorm},
\begin{equation}\label{eqn:lambda_Zs}
\lambda_{Z_{\mathfrak s}} =_{\langle X \rangle} - \log \mathop{\max}_{0 \leq i \leq N} \|\varsigma_i\|_h =_{\langle X \rangle} \frac{1}{2}(\varphi - \varphi_{\mathfrak S})
,\end{equation}
where $\varphi$ is the weighted function of a continuous Hermitian metric $h$ on $L$.

\begin{definition}\label{def:currentcount}
Let $\Phi \in \mathcal{D}'^{(1, 1)}(\mathbb{C})$ be a current of bidegree $(1, 1)$ (equivalently, of bidimension $(0, 0)$) such that $\dif \Phi = 0$ and $\Phi$ is representable by integration. One defines
\[
\mathcal{N}(\Phi, r) \bydef \Phi\left(\log^{+}\!\frac{r}{|z|}\right) = \int_0^r \frac{\dif t}{t} \int_{\bigtriangleup(t)} \!\Phi
,\]
where $\log^{+}\!t \bydef \max \{\log t, 0\}$ for any $t \in \mathbb{R}_{+}$.
\end{definition}

Now, let $f \colon \mathbb{C} \to X$ be a holomorphic curve in a projective manifold $X$.
Let $Z \subset X$ be a closed subscheme such that $f(\mathbb{C}) \not\subset \Supp Z$.
Choose a Weil function $\lambda_Z$ for $Z$.
Thereupon one defines the proximity function of $f$ for $Z$ by
\[
m_f(r, Z) \bydef \int_{\partial \bigtriangleup(r)} \!\lambda_Z \circ f \dif^c |z|^2 = \int_0^{2\pi} \!\lambda_Z \circ f(re^{i\theta}) \frac{\dif \theta}{2\pi}
.\]
According to Proposition \ref{prp:lambda_Z}, \ref{itm:lambdasubset}), we have $0 \leq_{\langle X \rangle} \lambda_Z$.
By definition, there is a continuous function $\alpha$ defined on $X$ and satisfies that $0 \leq \lambda_Z + \alpha$ on some non-empty Zariski open subset.
Thanks to the compactness of $X$, $\alpha$ is indeed bounded on $X$.
Consequently, $0 \leq m_f(r, Z) + O(1)$.
Here $O(\bullet)$ is the Landau symbol ``big-O''.
Consider a non-negative real-valued function $\varphi \colon \mathbb{R}_{\geq 0} \to \mathbb{R}_{\geq 0}$.
$O(\varphi)$ represents any quantity $\psi \colon \mathbb{R}_{\geq 0} \to \mathbb{R}$ satisfying that
\[
|\psi(r)| \leq c \,\varphi(r)
\]
for some constant $c > 0$ and all $r \in \mathbb{R}_{\geq 0}$.
In particular, $O(1)$ denotes all bounded real-valued functions on $\mathbb{R}_{\geq 0}$.
If choosing another Weil function $\lambda'_Z$ for $Z$, we obtain a new function $m'_f(r, Z)$.
Likewise, we see that $m_f(r, Z) = m'_f(r, Z) + O(1)$.
Hence we prefer to consider the proximity functions as functions modulo $O(1)$.

Next, let $f$ and $Z$ be as above. One also defines the counting function of $f$ for $Z$ by
\[
N_f(r, Z) \bydef \mathcal{N}(f^{\ast}[Z], r)
.\]

Finally, one defines the characteristic function (or order function) of $f$ for $Z$ by
\[
T_f(r, Z) \bydef m_f(r, Z) + N_f(r, Z)
.\]

\subsection{First main theorem}\label{ssc:FMT}

Now, investigate the particular case where $Z$ is a (Cartier) divisor $D$.
From \eqref{eqn:lambda_D}, it allows us to choose the Weil function $\lambda_D = - \log \|\sigma\|_h$ for some continuous Hermitian metric $h$ on $\mathcal{O}(D)$.
We thus write
\[
m_f(r, D) = - \int_0^{2\pi} \!\log \|\sigma \circ f(re^{i\theta})\|_h \frac{\dif \theta}{2\pi}
,\]
where $\sigma$ be a canonical section of $\mathcal{O}(D)$.
Let us recall the Poincaré-Lelong formula (cf. \cite{Lelong1964}) and the Green-Jensen formula, which play important roles in Nevanlinna theory.
\begin{theorem}[Poincaré-Lelong formula]
Let $f \not\equiv 0$ be a meromorphic function on a complex manifold $X$ of dimension $n$.
Let $\eta$ be a $(2n-2)$-form of $C^2$-class on $X$ with compact support. Then
\[
\int_X \log |f|^2 \ddif^c \eta = \int_{D_f} \!\eta
,\]
which amounts to saying that, in the sense of currents,
\[
\ddif^c[\log |f|^2] = [D_f]
,\]
where $D_f$ is a divisor locally expressed by $(f_1) - (f_2)$, if $f$ is locally expressed by the ratio $f = \dfrac{f_1}{f_2}$ with $f_2 \not\equiv 0$.
\end{theorem}

\begin{proof}
See Theorem 5.1.13 in \cite{Noguchi1990} or Theorem 2.2.16 in \cite{Noguchi2014}.
\end{proof}

Moreover, we also have the following theorem as an application of the Poincaré-Lelong formula to global sections. 
\begin{theorem}\label{thm:PLsigma}
Let $X$ be a complex manifold, let $D$ be an effective divisor on $X$ with $D = (\sigma)$. Choose a continuous Hermitian metric $h$ on $\mathcal{O}(D)$ with the weight function $\varphi$. Then $\log \|\sigma\|_h^2 \in L^1_{\loc}$ and
\[
\ddif^c [\log \|\sigma\|_h^2] = [D] - [\ddif^c \varphi]
.\]
\end{theorem}

\begin{proof}
See Lemma 1.12 in \cite{Carlson1972}.
\end{proof}

\begin{theorem}[Green-Jensen formula]
Let $g \colon \mathbb{C} \to [-\infty, \infty]$ be a function on $\mathbb{C}$ such that $\ddif^c [g]$ is of order zero, that is to say, $g$ can be locally written as a difference of two subharmonic functions.
Suppose that $g(0)$ is finite. Then
\[
\mathcal{N}(\ddif^c [g], r) = \frac{1}{2} \left(\int_0^{2\pi} \! g(re^{i\theta}) \frac{\dif \theta}{2\pi} - g(0)\right)
.\]
\end{theorem}

\begin{proof}
Actually, this theorem is just the one variable case of Corollary 2.1.37 in \cite{Noguchi2014} with the supposition that the Lelong number of $\ddif^c [g]$ at the origin $0$ is zero.
\end{proof}

Suppose $f(0) \notin D$. It follows from the Green-Jensen formula and Theorem \ref{thm:PLsigma} that
\begin{align*}
m_f(r, D) &= - \frac{1}{2} \int_0^{2\pi} \!\log \|\sigma \circ f(re^{i\theta})\|_h^2 \frac{\dif \theta}{2\pi} \\
&= - \mathcal{N}(f^{\ast} \ddif^c [\log \|\sigma\|_h^2], r) - \log \|\sigma \circ f(0)\|_h \\
&= \mathcal{N}(f^{\ast} [\ddif^c \varphi], r) - \mathcal{N}(f^{\ast} [D], r) - \log \|\sigma \circ f(0)\|_h
.\end{align*}
Therefore,
\begin{equation}\label{eqn:FMT}
T_f(r, D) = \mathcal{N}(f^{\ast} [\ddif^c \varphi], r) - \log \|\sigma \circ f(0)\|_h
.\end{equation}
This implies that as a function modulo $O(1)$, $T_f(r, D)$ merely depends on the line bundle $\mathcal{O}(D)$ associated to the divisor $D$.
Accordingly, for any holomorphic line bundle $L$, one can define the characteristic function modulo $O(1)$ of $f$ for $L$ by
\[
T_f(r, L) \bydef \mathcal{N}(f^{\ast} [\ddif^c \varphi], r)
,\]
where $\varphi$ is the weight function of some continuous Hermitian metric $h$ on $L$.
If choosing another continuous Hermitian metric $h'$ on $L$ with the weight function $\varphi'$ we obtain a new function $T'_f(r, L)$.
Thanks to the compactness of $X$, $\varphi-\varphi'$ is a bounded function on $X$.
By the Green-Jensen formula,
\[
T_f(r, L) - T'_f(r, L) = \frac{1}{2} \left(\int_0^{2\pi} \!(\varphi(re^{i\theta})-\varphi'(re^{i\theta})) \frac{\dif \theta}{2\pi} - \varphi(0) - \varphi'(0)\right) = O(1)
.\]
Consequently, $T_f(r, L)$ is well-defined as a function modulo $O(1)$.

\begin{remark}\label{rmk:zero}
Actually, the discussion above is still valid for the case where $f(0) \in D$.
If $f(0) \in D$, then we can write $\|\sigma \circ f(z)\|_h^2 = |z|^{2m} g^2(z)$ with $g(0) \neq 0$, where $\sigma \circ f$ has a zero of order $m$ at the origin $0$.
Actually, $m$ is the Lelong number of the current $\ddif^c [\log \|\sigma \circ f\|_h^2]$ at the origin $0$.
Hence
\begin{align*}
m_f(r, D) &= - \frac{1}{2} \int_0^{2\pi} \!\log g^2(re^{i\theta}) \frac{\dif \theta}{2\pi} - \frac{1}{2} \int_0^{2\pi} \!\log |re^{i\theta}z|^{2m} \frac{\dif \theta}{2\pi} \\
&= - \mathcal{N}(\ddif^c [\log g^2], r) - \mathcal{N}(\ddif^c [\log |z|^{2m}], r) - \log |g(0)| \\
&= - \mathcal{N}(f^{\ast} \ddif^c [\log \|\sigma\|_h^2], r) - \log |g(0)| \\
&= \mathcal{N}(f^{\ast} [\ddif^c \varphi], r) - \mathcal{N}(f^{\ast} [D], r) - \log |g(0)|
.\end{align*}
\end{remark}

Furthermore, by \eqref{eqn:FMT}, one easily concludes the first main theorem (FMT for short) in Nevanlinna Theory.
\begin{theorem}[FMT]
Let $X$ be a projective manifold, and let $L$ be a (holomorphic) line bundle on $X$.
Take a Cartier divisor $D$ on $X$ such that $\mathcal{O}(D) = L$. Take two effective divisors $D_1$ and $D_2$ satisfying $D = D_1 - D_2$, if necessary. Then
\[
T_f(r, L) = T_f(r, D_1) - T_f(r, D_2) + O(1)
,\]
for any holomorphic curve $f \colon \mathbb{C} \to X$ such that $f(\mathbb{C}) \not\subset \Supp D$.
\end{theorem}

\subsection{Crofton's formula}

Let us consider the unitary group $\textnormal{U}_{n+1}$, which acts on the left on $\Gr(k, \mathbb{P}^n) \cong \Gr(k+1, n+1)$, and especially on $\mathbb{P}^n \cong \Gr(1, n+1)$.
Technically, one can replace $\textnormal{U}_{n+1}$ by the special unitary group $\textnormal{SU}_{n+1}$, which acts effectively on $\Gr(k, \mathbb{P}^n)$ and on $\mathbb{P}^n$.
Let $\mu$ denote the Haar measure on $\textnormal{U}_{n+1}$ normalized such that $\mu(\textnormal{U}_{n+1}) = 1$.

Shiffman \cite{Shiffman1974} established the following lemma as the central ingredient in his proof of Crofton's formula (on complex manifolds) and its generalizations.
We call this lemma Crofton's formula for currents on $\mathbb{P}^n$.
\begin{lemma}
Let $\Phi \in \mathcal{D}'^{(k, k)}(\mathbb{P}^n)$ be a current of bidegree $(k, k)$. For any $\eta \in \mathcal{D}^{(n-k, n-k)}(\mathbb{P}^n)$,
\[
\int_{\textnormal{U}_{n+1}} \hspace{-1eM} \Phi(g^{\ast} \eta) \dif \mu(g) = c \int_{\mathbb{P}^n} \!\omega^k \wedge \eta
,\]
where $c = \Phi(\omega^{n-k})$ and $\omega$ is the Kähler form associated to the Fubini-Study metric on $\mathbb{P}^n$.
If $\Phi$ is representable by integration, then the above formula holds for all $(n-k, n-k)$-form $\eta$ of $C^0$-class on $\mathbb{P}^n$ with compact support.
In other words,
\[
\int_{\textnormal{U}_{n+1}} \hspace{-1eM} g_{\ast} \Phi \dif \mu(g) = c [\omega^k]
,\]
and it is valid for $C^0$ forms, if $\Phi$ is representable by integration.
\end{lemma}

Now, apply this lemma with $\Phi = [H]$, where $H$ is a hyperplane of $\mathbb{P}^n$.
Notice that $[H](\omega^{n-1}) = 1$ and the Fubini-Study metric on $\mathbb{P}^n$ can be regarded as a Hermitian metric $h$ on hyperplane line bundle $\mathcal{O}_{\mathbb{P}^n}(1)$.
Thanks to Theorem \ref{thm:PLsigma}, we obtain
\[
\int_{\textnormal{U}_{n+1}} \hspace{-1eM} g_{\ast} \ddif^c [\log \|\sigma\|_h] \dif \mu(g) = 0
,\]
where $\sigma$ is a canonical section of $\mathcal{O}(H)$.

What's more, the following lemma gives the average of Weil functions for all the hyperplanes of $\mathbb{P}^n$ in the sense of $\mathcal{C}(\mathbb{P}^n) / =_{\langle \mathbb{P}^n \rangle}$.
\begin{lemma}\label{lem:lambda_Havg}
Let $H$ be a hyperplane of $\mathbb{P}^n$. Take a fixed canonical section $\sigma$ of $\mathcal{O}_{\mathbb{P}^n}(H)$. Then
\[
\int_{\textnormal{U}_{n+1}} \hspace{-1eM} g^{\ast} \log \|\sigma(x)\|_h \dif \mu(g) = c
,\]
for every $x \in \mathbb{P}^n$, where $c$ is a constant.
It implies that
\begin{equation}\label{eqn:lambda_Havg}
\int_{\textnormal{U}_{n+1}} \hspace{-1eM} g^{\ast} \lambda_H \dif \mu(g) =_{\langle \mathbb{P}^n \rangle} 0
.\end{equation}
\end{lemma}

\begin{proof}
Seeing that the unitary group $\textnormal{U}_{n+1}$ acts transitively on $\mathbb{P}^n$, the improper integral
\[
\int_{\textnormal{U}_{n+1}} \hspace{-1eM} g^{\ast} \log \|\sigma(x)\|_h \dif \mu(g) = \int_{\textnormal{U}_{n+1}} \hspace{-1eM} \log \|\sigma(g \cdot x)\|_h \dif \mu(g)
\]
equals to a constant $c$ independent of $x \in \mathbb{P}^n$ (cf. the reasoning in the proof of Proposition 3.5 in \cite{Cowen1976}).
\end{proof}

\begin{remark}\label{rmk:lambda_Havg}
Think of the complete linear system $\mathfrak{S} = |H|$ as the dual projective space $(\mathbb{P}^n)^{\dual}$.
Then $g^{\ast} H$ defines a right action $g \in \textnormal{U}_{N+1}$ on $H \in \mathfrak{S}$.
There is a unique measure $\rho$ on $\mathfrak{S} \cong (\mathbb{P}^n)^{\dual}$, which is invariant under the unitary group $\textnormal{U}_{n+1}$ and satisfies $\rho(\mathfrak{S}) = 1$.
Thus \eqref{eqn:lambda_Havg} can be reformulated as
\[
\int_{H \in \mathfrak{S}} \hspace{-1eM} \lambda_H \dif \rho(H) =_{\langle \mathbb{P}^n \rangle} 0
.\]
\end{remark}

Take an arbitrary linear system $\mathfrak{S} \subset |L|$ on a projective manifold $X$.
Observe that effective divisors $D \in \mathfrak{S}$ are in one-to-one correspondence with hyperplanes $H_D$ of $\mathfrak{S}^{\dual} \cong \mathbb{P}^n$.
Therefore, the unitary group $\textnormal{U}_{N+1}$ acts on the right on $\mathfrak{S}$, where $N = \dim \mathfrak{S}$.
More precisely, an action $g \in \textnormal{U}_{N+1}$ on $D \in \mathfrak{S}$ is described by $g^{\ast} D$, which corresponds to the hyperplane $g^{\ast} H_D$ of $\mathfrak{S}^{\dual}$.
Accordingly, we generalize Lemma \ref{lem:lambda_Havg} to divisors on projective manifolds.
\begin{theorem}\label{thm:lambda_Davg}
Let $X$ be a projective manifold, and take a linear system $\mathfrak{S} \subset |L|$ on $X$.
Given a divisor $D \in \mathfrak{S}$, we have
\begin{equation}\label{eqn:lambda_Davg}
\int_{\textnormal{U}_{N+1}} \hspace{-1eM} g^{\ast} \lambda_D \dif \mu(g) =_{\langle X \rangle} \lambda_{Z_{\mathfrak s}}
,\end{equation}
where we hereinafter write, by abuse of notation, $g^{\ast} \lambda_D \bydef \lambda_{g^{\ast} D}$, and $\mu$ is the Haar measure on $\textnormal{U}_{N+1}$ normalized such that $\mu(\textnormal{U}_{N+1}) = 1$.
\end{theorem}

\begin{proof}
Let $\mathbb{S} \subset H^0(X, L)$ be the linear subspace corresponding to $\mathfrak{S}$.
Take a basis $\{\varsigma_i \mid i = 0, 1, \dots, N\}$ of $\mathbb{S}$, where $N = \dim \mathfrak{S}$.
We naturally have a rational map
\begin{equation}\label{eqn:muS}
\begin{aligned}
\mu_{\mathfrak S} \colon X &\dashrightarrow \mathfrak{S}^{\dual} \cong \mathbb{P}^N \\
x &\longmapsto [\varsigma_0(x) : \dots : \varsigma_N(x)]
.\end{aligned}
\end{equation}
Then we are going to discuss the linear system $\mathfrak{S}$ into three different cases.

If $\mathfrak{S}$ is base point free, then $\mu_{\mathfrak S}$ is holomorphic.
We see that $D = \mu_{\mathfrak S}^{\ast} H_D$, for each $D \in \mathfrak{S}$.
In this case, the theorem immediately follows from Proposition \ref{prp:lambda_Z}, \ref{itm:lambdapullback}) and Lemma \ref{lem:lambda_Havg}.

Next, investigate the case when $\mathfrak{S}$ is not base point free and its base ideal $\mathfrak{s} = \mathcal{O}_X(-D_0)$ for an effective Cartier divisor $D_0$, i.e., $Z_{\mathfrak s} = D_0$.
Since $\mathfrak{S}_{\setminus D_0}$ is base point free, $\mu_{\mathfrak{S}_{\setminus D_0}}$ is a holomorphic map from $X$ to $\mathfrak{S}_{\setminus D_0}^{\dual} \cong \mathbb{P}^N$.
Regarding $\mathfrak{S}^{\dual}$ and $\mathfrak{S}_{\setminus D_0}^{\dual}$ as the same projective space, we easily check that $\mu_{\mathfrak{S}_{\setminus D_0} \restriction U} = \mu_{\mathfrak{S} \restriction U}$, where $U = X \setminus \Supp D_0$.
It implies that $\mu_{\mathfrak{S}_{\setminus D_0}}$ is actually the analytic extension of $\mu_{\mathfrak S}$.
Consequently,
\[
\int_{\textnormal{U}_{N+1}} \hspace{-1eM} g^{\ast} \lambda_D \dif \mu(g) =_{\langle X \rangle} \lambda_{D_0} + \int_{\textnormal{U}_{N+1}} \hspace{-1eM} g^{\ast} \lambda_{D-D_0} \dif \mu(g) =_{\langle X \rangle} \lambda_{D_0}
.\]

Finally, we just need to verify the case when $\mathfrak{S}$ is not base point free and $Z_{\mathfrak s}$ is not an effective Cartier divisor.
Take the blow-up of $X$ with respect to $\mathfrak{s}$
\[
\upsilon \colon \widehat{X} \longrightarrow X
,\]
and denote by $E_0 = \upsilon^{\ast} Z_{\mathfrak s}$ the exceptional divisor of $\upsilon$.
Hence there is a commutative diagram
\begin{equation*}
\begin{tikzcd}[column sep=scriptsize]
\widehat{X} \arrow[d, "\upsilon"] \arrow[rd, "\hat{\mu}"] & \\
X \arrow[r, dashed, "\mu"] & \mathfrak{S}^{\dual}
.\end{tikzcd}
\end{equation*}
Let us define $\upsilon^{\ast} \mathfrak{S} \bydef \{\upsilon^{\ast} D \mid D \in \mathfrak{S}\}$.
Notice that $\upsilon^{\ast} \mathfrak{S}$ satisfies the second case in our discussion.
Following the same reasoning as above and Proposition \ref{prp:lambda_Z}, \ref{itm:lambdapullback}), we obtain
\[
\int_{\textnormal{U}_{N+1}} \hspace{-1eM} g^{\ast} \lambda_D \circ \upsilon \dif \mu(g) =_{\langle \widehat{X} \rangle} \int_{\textnormal{U}_{N+1}} \hspace{-1eM} g^{\ast} \lambda_{\upsilon^{\ast} D} \dif \mu(g) =_{\langle \widehat{X} \rangle} \lambda_{E_0} =_{\langle \widehat{X} \rangle} \lambda_{Z_{\mathfrak s}} \circ \upsilon
,\]
which amounts to saying that there exists a constant $a > 0$ such that
\[
\lambda_{Z_{\mathfrak s}} \circ \upsilon(\hat{x}) - a \leq \int_{\textnormal{U}_{N+1}} \hspace{-1eM} g^{\ast} \lambda_D \circ \upsilon(\hat{x}) \dif \mu(g) \leq \lambda_{Z_{\mathfrak s}} \circ \upsilon(\hat{x}) + a
,\]
for every $\hat{x} \in \widehat{X} \setminus E_0$.
Seeing that $\upsilon$ restricted on $\widehat{X} \setminus E_0$ is an isomorphism, we have
\[
\lambda_{Z_{\mathfrak s}} (x) - a \leq \int_{\textnormal{U}_{N+1}} \hspace{-1eM} g^{\ast} \lambda_D (x) \dif \mu(g) \leq \lambda_{Z_{\mathfrak s}} (x) + a
,\]
for every $x \in X \setminus \Supp Z_{\mathfrak s}$.
By definition,
\[
\int_{\textnormal{U}_{N+1}} \hspace{-1eM} g^{\ast} \lambda_D \dif \mu(g) =_{\langle X \rangle} \lambda_{Z_{\mathfrak s}}
.\qedhere\]
\end{proof}

\begin{remark}
Analogous to Remark \ref{rmk:lambda_Havg}, \eqref{eqn:lambda_Davg} is reformulated as
\[
\int_{D \in \mathfrak{S}} \hspace{-1eM} \lambda_D \dif \rho(D) =_{\langle X \rangle} \lambda_{Z_{\mathfrak s}}
,\]
where $\rho$ is (slightly abusively) the unique measure on $\mathfrak{S} \cong (\mathbb{P}^N)^{\dual}$ invariant under the unitary group $\textnormal{U}_{N+1}$ and satisfying $\rho(\mathfrak{S}) = 1$.
\end{remark}

We refer to Theorem \ref{thm:lambda_Davg} as Crofton's formula for Weil functions for divisors on projective manifolds.

Take a linear system $\mathfrak{T} \subset |L|$ on $X$.
We say that a subset $\mathfrak{S} \subset \mathfrak{T}$ is a linear subsystem of $\mathfrak{T}$ on $X$, if it is a linear subspace for the projective space structure of $\mathfrak{T}$.
For each $0 \leq k \leq N-1$, consider the Grassmannian $\Gr(k, \mathfrak{T})$, which consist of all linear subsystems $\mathfrak{S}$ of $\mathfrak{T}$ on $X$ such that $\dim \mathfrak{S} = k$, where $N = \dim \mathfrak{T}$.
We thus conclude Crofton's formula for Weil functions for closed subschemes of projective manifolds.
\begin{theorem}\label{thm:lambda_Zsavg}
Let $X$ be a projective manifold, and take a linear system $\mathfrak{T} \subset |L|$ on $X$.
Choose a linear subsystem $\mathfrak{S}$ of $\mathfrak{T}$ on $X$ such that $\dim \mathfrak{S} = k$. Then
\begin{equation}\label{eqn:lambda_Zsavg}
\int_{\textnormal{U}_{N+1}} \hspace{-1eM} \lambda_{g^{\ast} Z_{\mathfrak s}} \dif \mu(g) =_{\langle X \rangle} \lambda_{Z_{\mathfrak t}}
.\end{equation}
\end{theorem}

\begin{proof}
For any $\mathfrak{S} \in \Gr(k, \mathfrak{T})$, taking a basis $\{\varsigma_i \mid i = 0, 1, \dots, k\}$ of $\mathbb{S}$,
we have $Z_{\mathfrak s} = \bigcap_{i=0}^k D_i$, where $D_i = (\varsigma_i)$ for each $0 \leq i \leq k$.
Think of $\Gr(k, \mathfrak{T})$ as the dual of $\Gr(k, \mathfrak{T}^{\dual}) \cong \Gr(k+1, N+1)$.
The unitary group $\textnormal{U}_{N+1}$ acts on the right on $\Gr(k, \mathfrak{T}) \cong \Gr(k+1, N+1)^{\dual}$.
In addition, it induces a right action on the set of the scheme of base points of all the $\mathfrak{S} \in \Gr(k, \mathfrak{T})$,
where an action $g \in \textnormal{U}_{N+1}$ on $Z_{\mathfrak s}$ can be described by
\[
g^{\ast} Z_{\mathfrak s} = \bigcap_{i=0}^k g^{\ast} D_i
,\]
regardless of the choice of the basis.
By Proposition \ref{prp:lambda_Z}, \ref{itm:lambdasubset}) and Theorem \ref{thm:lambda_Davg}, we get
\begin{align*}
\int_{\textnormal{U}_{N+1}} \hspace{-1eM} \lambda_{g^{\ast} Z_{\mathfrak s}} \dif \mu(g) &=_{\langle X \rangle} \int_{\textnormal{U}_{N+1}} \!\mathop{\min}_{0 \leq i \leq \ell} \lambda_{g^{\ast} D_i} \dif \mu(g) \\
&\leq_{\langle X \rangle} \mathop{\min}_{0 \leq i \leq \ell} \int_{\textnormal{U}_{N+1}} \hspace{-1eM} \lambda_{g^{\ast} D_i} \dif \mu(g) \\
&\leq_{\langle X \rangle} \lambda_{Z_{\mathfrak t}}
.\end{align*}
On the other hand, since $\mathfrak{s} \subset \mathfrak{t}$, namely, $Z_{\mathfrak t} \subset Z_{\mathfrak s}$ for all $\mathfrak{S} \in \Gr(k, \mathfrak{T})$, then Proposition \ref{prp:lambda_Z}, \ref{itm:lambdasubset}) leads to
\[
\lambda_{Z_{\mathfrak t}} \leq_{\langle X \rangle} \int_{\textnormal{U}_{N+1}} \hspace{-1eM} \lambda_{g^{\ast} Z_{\mathfrak s}} \dif \mu(g)
.\]
Then the theorem is verified.
\end{proof}

\begin{remark}\label{rmk:lambda_Zsavg}
Let $\rho_{k, N}$ be the unique measure on $\Gr(k, \mathfrak{T}) \cong \Gr(k+1, N+1)^{\dual}$ invariant under the unitary group $\textnormal{U}_{N+1}$ and satisfying $\rho_{k, N}(\Gr(k, \mathfrak{T})) = 1$.
Therefore, \eqref{eqn:lambda_Zsavg} is reformulated as
\[
\int_{\mathfrak{S} \in \Gr(k, \mathfrak{T})} \hspace{-1eM} \lambda_{Z_{\mathfrak s}} \dif \rho_{k, N}(\mathfrak{S}) =_{\langle X \rangle} \lambda_{Z_{\mathfrak t}}
.\]
In the particular case when $k=0$, we see that $\mathfrak{S} \in \Gr(0, \mathfrak{T})$ corresponds to $D \in \mathfrak{T}$, which implies that Theorem \ref{thm:lambda_Zsavg} for $k=0$ coincides with Theorem \ref{thm:lambda_Davg}.
\end{remark}

Choose linear subsystems $\mathfrak{S}_0, \dots, \mathfrak{S}_{\ell}$ of $\mathfrak{T} \subset |L|$ on $X$ such that $\mathbb{T} = \mathbb{S}_0 + \dots + \mathbb{S}_{\ell}$,
where $\mathbb{T}$, $\mathbb{S}_0, \dots, \mathbb{S}_{\ell}$ are the linear subspace of $H^0(X, L)$ corresponding to $\mathfrak{T}$, $\mathfrak{S}_0, \dots, \mathfrak{S}_{\ell}$, respectively.
Moreover, take another linear system $\mathfrak{S}$ (not necessary lying in $|L|$) on $X$ such that $\dim \mathfrak{S} \leq \dim \mathfrak{T}$, but $\dim \mathfrak{S} \geq \dim \mathfrak{S}_i$ for each $0 \leq i \leq \ell$.
Obviously, for each $0 \leq i \leq \ell$, there exists a surjective homomorphism $\beta_i \colon \mathbb{S} \twoheadrightarrow \mathbb{S}_i$,
where $\mathbb{S}$ are the linear space corresponding to $\mathfrak{S}$.
Take account of the following map
\begin{equation}\label{eqn:Zbeta}
\begin{aligned}
Z_{\beta} \colon \mathfrak{S} &\longrightarrow Z(X) \\
D &\longmapsto Z_{\beta}(D) = \bigcap_{i=0}^{\ell} \beta_i(D)
,\end{aligned}
\end{equation}
where $\beta_i(D) = (\beta_i(\sigma))$ for a canonical section $\sigma$ of $\mathcal{O}_X(D)$, and in particular, $\beta_i(D) = X$ if $\beta_i(\sigma) = 0$.
We make a convention that $\lambda_X \equiv \infty$ as a generalized function.
Then we obtain a generalization of Crofton's formula as follows.
\begin{theorem}\label{thm:lambda_ZbetaDavg}
Let $X$ be a projective manifold, and choose linear systems $\mathfrak{T}$, $\mathfrak{S}_0, \dots, \mathfrak{S}_{\ell}$, and $\mathfrak{S}$ on $X$ as above.
Let $\beta_i$ be as above, for each $0 \leq i \leq \ell$.
Given a divisor $D \in \mathfrak{S}$, we have
\begin{equation}\label{eqn:lambda_ZbetaDavg}
\int_{\textnormal{U}_{N+1}} \hspace{-1eM} \lambda_{Z_{\beta}(g^{\ast} D)} \dif \mu(g) =_{\langle X \rangle} \lambda_{Z_{\mathfrak t}}
,\end{equation}
where $N = \dim \mathfrak{S}$.
\end{theorem}

\begin{proof}
For each $0 \leq i \leq \ell$, the dual homomorphism $\beta_i^{\dual} \colon \mathbb{S}_i^{\dual} \hookrightarrow \mathbb{S}^{\dual}$ induces a closed embedding $\iota_{\beta_i^{\vee}} \colon \mathfrak{S}_i^{\dual} \cong \textnormal{P}(\mathbb{S}_i^{\dual}) \hookrightarrow \textnormal{P}(\mathbb{S}^{\dual}) \cong \mathfrak{S}^{\dual}$.
We easily check that $\beta_i(D)$ corresponds to the hyperplane $\iota_{\beta_i^{\vee}}^{\ast} H_D$ of $\mathfrak{S}_i^{\dual}$ if $\iota_{\beta_i^{\vee}}(\mathfrak{S}_i^{\dual}) \not\subset \Supp H_D$, and $\beta_i(D) = X$, which corresponds to $\mathfrak{S}_i^{\dual}$, if $\iota_{\beta_i^{\vee}}(\mathfrak{S}_i^{\dual}) \subset \Supp H_D$.
By abuse of notation, let $\beta_i(D)$ also denote the corresponding hyperplane of $\mathfrak{S}_i^{\dual}$ (or $\mathfrak{S}_i^{\dual}$ itself).
According to Proposition \ref{prp:lambda_Z}, \ref{itm:lambdapullback}), we get
\[
\int_{\textnormal{U}_{N+1}} \hspace{-1eM} \lambda_{\beta_i(g^{\ast} D)} \dif \mu(g) =_{\langle \mathfrak{S}_i^{\vee} \rangle} \int_{\textnormal{U}_{N+1}} \hspace{-1eM} g^{\ast} \lambda_{H_D} \circ \iota_{\beta_i^{\vee}} \dif \mu(g) =_{\langle \mathfrak{S}_i^{\vee} \rangle} 0
.\]
Following the same reasoning in the proof of Theorem \ref{thm:lambda_Davg}, the above formula yields
\begin{equation}\label{eqn:lambda_betaDavg}
\int_{\textnormal{U}_{N+1}} \hspace{-1eM} \lambda_{\beta_i(g^{\ast} D)} \dif \mu(g) =_{\langle X \rangle} \lambda_{Z_i}
,\end{equation}
where $Z_i = Z_{\mathfrak s_i}$ for each $0 \leq i \leq \ell$.
We assert that the theorem follows from the same reasoning in the proof of Theorem \ref{thm:lambda_Zsavg}.
Concretely speaking, according to Proposition \ref{prp:lambda_Z}, \ref{itm:lambdacap}) and \eqref{eqn:lambda_betaDavg},
\begin{align*}
\int_{\textnormal{U}_{N+1}} \hspace{-1eM} \lambda_{Z_{\beta}(g^{\ast} D)} \dif \mu(g) &=_{\langle X \rangle} \int_{\textnormal{U}_{N+1}} \!\mathop{\min}_{0 \leq i \leq \ell} \lambda_{\beta_i(g^{\ast} D)} \dif \mu(g) \\
&\leq_{\langle X \rangle} \mathop{\min}_{0 \leq i \leq \ell} \int_{\textnormal{U}_{N+1}} \hspace{-1eM} \lambda_{\beta_i(g^{\ast} D)} \dif \mu(g) \\
&\leq_{\langle X \rangle} \mathop{\min}_{0 \leq i \leq \ell} \lambda_{Z_i} \leq_{\langle X \rangle} \lambda_{Z_{\mathfrak t}}
.\end{align*}
On the other hand, seeing that $Z_{\mathfrak t} \subset Z_{\beta}(D)$ for all $D \in \mathfrak{S}$, by Proposition \ref{prp:lambda_Z}, \ref{itm:lambdasubset}), we obtain
\[
\lambda_{Z_{\mathfrak t}} \leq_{\langle X \rangle} \int_{\textnormal{U}_{N+1}} \hspace{-1eM} \lambda_{Z_{\beta}(g^{\ast} D)} \dif \mu(g)
.\qedhere\]
\end{proof}

\begin{remark}
Likewise, \eqref{eqn:lambda_ZbetaDavg} is reformulated as
\[
\int_{D \in \mathfrak{S}} \hspace{-.5eM} \lambda_{Z_{\beta}(D)} \dif \rho(D) =_{\langle X \rangle} \lambda_{Z_{\mathfrak t}}
,\]
where $\rho$ is the unique measure on $\mathfrak{S} \cong (\mathbb{P}^N)^{\dual}$ invariant under the unitary group $\textnormal{U}_{N+1}$ and satisfying $\rho(\mathfrak{S}) = 1$.
\end{remark}

Furthermore, following the same reasoning in the proofs of Theorem \ref{thm:lambda_Zsavg} and Theorem \ref{thm:lambda_ZbetaDavg}, we generalize Theorem \ref{thm:lambda_ZbetaDavg} further to the Grassmannian.
\begin{theorem}\label{thm:lambda_betaZavg}
Let $X$ be a projective manifold, and choose linear systems $\mathfrak{T}$, $\mathfrak{S}_0, \dots, \mathfrak{S}_{\ell}$, and $\mathfrak{S}$ on $X$ as in Theorem \ref{thm:lambda_ZbetaDavg}.
Let $\beta_i$ be as in Theorem \ref{thm:lambda_ZbetaDavg}, for each $0 \leq i \leq \ell$.
For any linear subsystem $\mathfrak{O} \in \Gr(k, \mathfrak{S})$, consider the map
\begin{align*}
Z_{k, \beta} \colon \Gr(k, \mathfrak{S}) &\longrightarrow Z(X) \\
\mathfrak{O} &\longmapsto Z_{k, \beta}(\mathfrak{O}) = \bigcap_{i=0}^{\ell} \beta_i(\mathfrak{O})
,\end{align*}
where $\beta_i(\mathfrak{O})$ is the subscheme of base points of $\mathfrak{O}_{\beta_i}$, which corresponds to $\beta_i(\mathbb{O}) \subset \mathbb{S}_i$.
Then, we have
\begin{equation}\label{eqn:lambda_betaZavg}
\int_{\textnormal{U}_{N+1}} \hspace{-1eM} \lambda_{Z_{k, \beta}(g^{\ast} \mathfrak{O})} \dif \mu(g) =_{\langle X \rangle} \lambda_{Z_{\mathfrak t}}
,\end{equation}
where we write, by abuse of notation, $g^{\ast} \mathfrak{O}$ to describe a right action $g \in \textnormal{U}_{N+1}$ on $\mathfrak{O} \in \Gr(k, \mathfrak{S})$.
\end{theorem}

\begin{remark}
Analogous to Remark \ref{rmk:lambda_Zsavg}, \eqref{eqn:lambda_betaZavg} is reformulated as
\[
\int_{\mathfrak{O} \in \Gr(k, \mathfrak{S})} \hspace{-1eM} \lambda_{{Z_{k, \beta}(\mathfrak{O})}} \dif \rho_{k, N}(\mathfrak{O}) =_{\langle X \rangle} \lambda_{Z_{\mathfrak t}}
,\]
where $\rho_{k, N}$ is the normalized measure on $\Gr(k, \mathfrak{S}) \cong \Gr(k+1, N+1)^{\dual}$ as in Remark \ref{rmk:lambda_Zsavg}.
What's more, Theorem \ref{thm:lambda_betaZavg} for $k=0$ coincides with Theorem \ref{thm:lambda_ZbetaDavg}.
\end{remark}

Take a linear system $\mathfrak{S} \subset |L|$ on $X$.
A holomorphic curve $f \colon \mathbb{C} \to X$ is said to be non-degenerated relative to $\mathfrak{S}$, if the image $f(\mathbb{C})$ does not lie in any $D \in \mathfrak{S}$.
By the definition of proximity functions, Crofton's formula for Weil functions in the sense of $\mathcal{C}(X) / =_{\langle X \rangle}$ implies Crofton's formula for proximity functions modulo $O(1)$.
\begin{theorem}\label{thm:m_Zbetaavg}
Let $X$ be a projective manifold, and choose linear systems $\mathfrak{T}$, $\mathfrak{S}_0, \dots, \mathfrak{S}_{\ell}$, and $\mathfrak{S}$ on $X$ as in Theorem \ref{thm:lambda_ZbetaDavg}.
Let $\beta_i$ be as in Theorem \ref{thm:lambda_ZbetaDavg}, for each $0 \leq i \leq \ell$.
Let $f \colon \mathbb{C} \to X$ be a holomorphic curve, which is non-degenerated relative to $\mathfrak{T}$.
If $\bigcap_{i=0}^{\ell} \textnormal{Ker} \,\beta_i = 0$, then
\[
\int_{D \in \mathfrak{S}} \hspace{-.5eM} m_f(r, Z_{\beta}(D)) \dif \rho(D) = m_f(r, Z_{\mathfrak t}) + O(1)
.\]
In addition, suppose that $\dim \left( \bigcap_{i=0}^{\ell} \textnormal{Ker} \,\beta_i \right) \leq k$, for each $0 \leq k \leq N-1$. We have
\[
\int_{\mathfrak{O} \in \Gr(k, \mathfrak{S})} \hspace{-1eM} m_f(r, Z_{k, \beta}(\mathfrak{O})) \dif \rho_{k, N}(\mathfrak{O}) = m_f(r, Z_{\mathfrak t}) + O(1)
.\]
\end{theorem}

\begin{remark}\label{rmk:m_Zbetaavg}
The condition that $f$ is non-degenerated relative to $\mathfrak{T}$ and the condition that $\bigcap_{i=0}^{\ell} \textnormal{Ker} \,\beta_i = 0$ (or the condition that $\dim \left( \bigcap_{i=0}^{\ell} \textnormal{Ker} \,\beta_i \right) \leq k$, respectively) are to guarantee that $f(\mathbb{C}) \not\subset Z_{\beta}(D)$ (or $f(\mathbb{C}) \not\subset Z_{k, \beta}(\mathfrak{O})$), and the proximity function $m_f(r, Z_{\beta}(D))$ (or $m_f(r, Z_{k, \beta}(\mathfrak{O}))$) is accordingly defined for any $D \in \mathfrak{S}$ (or $\mathfrak{O} \in \Gr(k, \mathfrak{S})$).
We make a convention that $m_f(r, Z) \equiv \infty$ if $f(\mathbb{C}) \subset Z$.
Even if such conditions is replaced by a weaker condition that $f$ is a non-constant holomorphic curve satisfying that $f(\mathbb{C}) \not\subset Z_{\mathfrak t}$, these two formulas in Theorem \ref{thm:m_Zbetaavg} still hold as the improper integrals.
\end{remark}

\subsection{Calculus Lemma}

\begin{lemma}[Borel's Calculus Lemma]
Let $h \colon \mathbb{R}_{\geq 0} \to \mathbb{R}_{\geq 0}$ be a monotonically increasing function.
For every $\delta > 0$, the inequality
\[
\frac{\dif}{\dif r} h(r) \leq h^{1+\delta}(r)
\]
holds for all $r \in \mathbb{R}_{\geq 0}$ outside a Borel subset $E$ of finite Lebesgue measure.
\end{lemma}

\begin{proof}
Thanks to the monotone of $h(r)$, the derivative $\dfrac{\dif}{\dif r} h(r)$ exists almost everywhere.
Observe that the inequality always holds in the particular case where $h(r)$ is constant.
Assume that $h(r) \not\equiv 0$, and take $r_0 \geq 0$ such that $h(r_0) > 0$.
Let $E \subset \mathbb{R}_{\geq 0}$ be the set of $r$ satisfying
\[
\frac{\dif}{\dif r} h(r) > h^{1+\delta}(r)
.\]
Seeing that $\int_{E \cap [0, r_0]} \dif r \leq r_0 < \infty$, it suffices to show that $\int_{E \cap (r_0, \infty)} \dif r < \infty$.
By the definition of $E$, we have
\[
\int_{E \cap (r_0, \infty)} \!\dif r \leq \int_{r_0}^{\infty} \frac{1}{h^{1+\delta}(r)} \frac{\dif}{\dif r} h(r) \dif r \leq \int_{h(r_0)}^{\infty} \frac{\dif h}{h^{1+\delta}} = \frac{1}{\delta h(r_0)^{\delta}} < \infty
.\qedhere\]
\end{proof}

For brevity, we will use the notation
\[
\varphi(r) \leq \left. \psi(r) \right\rVert_{E}
\]
to mean that the stated inequality holds outside a Borel subset $E$ of finite Lebesgue measure,
where $\varphi$ and $\psi$ are non-negative real-valued functions on $\mathbb{R}_{\geq 0}$.

The typical use of Calculus Lemma is as follows.
\begin{lemma}\label{lem:calculus}
Let $\gamma \colon \mathbb{C} \to \mathbb{R}_{\geq 0}$ be a non-negative function on $\mathbb{C}$.
For every $\delta > 0$, we have
\[
\int_0^{2\pi} \gamma(re^{i\theta}) \frac{\dif \theta}{2\pi} \leq \left. \left(\mathcal{N}([\gamma \ddif^c |z|^2], r)\right)^{(1+\delta)^2} \right\rVert_{E}
.\]
\end{lemma}

\begin{proof}
Notice that $\mathcal{N}([\gamma \ddif^c |z|^2], r) \equiv 0$ if and  only if $\gamma$ is almost everywhere zero on $\mathbb{C}$.
Obviously, the inequality holds in this case.
Assume that $\mathcal{N}([\gamma \ddif^c |z|^2], r) \not\equiv 0$.
There exists a constant $r_0 > 0$ such that $\mathcal{N}([\gamma \ddif^c |z|^2], r) \geq \dfrac{1}{r}$ for all $r \geq r_0$.
Using the polar coordinates, we see that $\ddif^c |z|^2 = 2r \dif r \wedge \dfrac{\dif \theta}{2\pi}$.
Hence
\[
r\frac{\dif}{\dif r} \mathcal{N}([\gamma \ddif^c |z|^2], r) = \int_{\bigtriangleup(r)} \gamma(z) \ddif^c |z|^2 = \int_0^{2\pi} \frac{\dif \theta}{2\pi} \int_0^r 2t \gamma(te^{i\theta}) \dif t
.\]
Applying the Borel's Calculus Lemma twice, we get
\begin{align*}
\int_0^{2\pi} \gamma(re^{i\theta}) \frac{\dif \theta}{2\pi} &= \frac{1}{2r} \frac{\dif}{\dif r} \left(r\frac{\dif}{\dif r} \mathcal{N}([\gamma \ddif^c |z|^2], r)\right) \\
&= \frac{1}{2} \left(\frac{1}{r} \frac{\dif}{\dif r} \mathcal{N}([\gamma \ddif^c |z|^2], r) + \frac{\dif^2}{\dif r^2} \mathcal{N}([\gamma \ddif^c |z|^2], r)\right) \\
&\leq \frac{1}{2} \left. \left(\frac{1}{r} + \left(\mathcal{N}([\gamma \ddif^c |z|^2], r)\right)^{\delta(1+\delta)}\right) \left(\mathcal{N}([\gamma \ddif^c |z|^2], r)\right)^{1+\delta} \right\rVert_{E_1} \\
&\leq \left. \left(\mathcal{N}([\gamma \ddif^c |z|^2], r)\right)^{(1+\delta)^2} \right\rVert_{E_1 \cup [0, r_0]}
.\qedhere\end{align*}
\end{proof}

\begin{remark}
Lemma \ref{lem:calculus} refines the typical application of Calculus Lemma (such as Lemma 3.4 in \cite{Lang1990}, or Lemma A\! 2.1.5 in \cite{Ru2001}), from the original Borel's Calculus Lemma rather than the improved version of Calculus Lemma notwithstanding.
\end{remark}

For any line bundle $L$ and any ample line bundle $A$ on $X$, one can find a positive integer $m$ such that $L^{-1} \otimes A^m$ is ample.
Consequently, we immediately obtain the following property of the characteristic functions.
\begin{proposition}\label{prp:TrA}
Let $L$ and $A$ be as above. Then there is a constant $c$ such that for any non-constant holomorphic curve $f \colon \mathbb{C} \to X$,
\[
T_f(r, L) \leq c \, T_f(r, A)
.\]
Moreover, if $L$ is ample as well, then there is a constant $c$ such that for any non-constant holomorphic curve $f \colon \mathbb{C} \to X$,
\begin{equation}\label{eqn:TrA}
\frac{1}{c} \, T_f(r, A) \leq T_f(r, L) \leq c \, T_f(r, A)
.\end{equation}
\end{proposition}

\begin{definition}\label{def:Sf}
Let $S_f(r)$ represent any quantity $\psi \colon \mathbb{R}_{\geq 0} \to \mathbb{R}_{\geq 0}$ satisfying that
\[
\psi(r) \leq \left. O(\log T_f(r, A)) \right\rVert_{E}
.\]
Especially, if $f$ is an algebraic curve,
\[
\psi(r) \leq O(1) \quad (r \to \infty)
.\]
That is to say, there exists a constant $c$ such that
\[
\psi(r) \leq c, \qquad \forall \, r \geq r_0
,\]
for some constant $r_0 > 0$.
\end{definition}
We see from \eqref{eqn:TrA} that the notation $S_f(r)$ is independent of the choice of the ample line bundle $A$.

\section{Ahlfors' lemma on logarithmic derivative}\label{sec:ALLD}

\subsection{Derived curves or associated curves}
First, recall the $k$-th derived curve (or associated curve) of a holomorphic curve $f$ (see Section 2.1 of \cite{Fujimoto1993} or Chapter 2, Section 4.1 of \cite{Griffiths1978} for more details).
Let $f$ be a holomorphic curve from $\mathbb{C}$ to $\mathbb{P}^{n}$, and take a reduced representation $[f_0 : \dots : f_n]$ of $f$.
Thereupon one gets a holomorphic curve
\begin{align*}
F \colon \mathbb{C} &\longrightarrow \mathbb{C}^{n+1} \setminus \{0\} \\
z &\longmapsto (f_0(z), \dots, f_n(z))
.\end{align*}
It is apparent that $f = \pr^n \circ F$, where $\pr^n \colon \mathbb{C}^{n+1} \setminus \{0\} \to \mathbb{P}^n$ is the canonical projection through the origin $0 \in \mathbb{C}^{n+1}$.
For each $0 \leq k \leq n$, consider the holomorphic curve
\begin{equation}\label{eqn:F^k}
\begin{aligned}
F^k \colon \mathbb{C} &\longrightarrow \bigwedge^{k+1} \mathbb{C}^{n+1} \\
z &\longmapsto F(z) \wedge F^{(1)}(z) \wedge \dots \wedge F^{(k)}(z)
,\end{aligned}
\end{equation}
where $F^{(\ell)} \bydef (f_0^{(\ell)}, \dots, f_n^{(\ell)})$ is a holomorphic curve in $\mathbb{C}^{n+1}$, for each $1 \leq \ell \leq k$.
In particular, $F^n$ is indeed a holomorphic function, known as the Wronskian $W(f)$ of $f$.

A holomorphic curve $f$ in $\mathbb{P}^n$ is said to be linear non-degenerate, if its image does not lie in any hyperplane of $\mathbb{P}^n$.
It is evident that $F^k \equiv 0$, if and only if the image $f(\mathbb{C})$ lies in a linear subspace $\mathbb{P}^{k-1} \subset \mathbb{P}^n$.

One abusively denotes by $\pr \colon \bigwedge^{k+1} \mathbb{C}^{n+1} \setminus \{0\} \to \textnormal{P}(\bigwedge^{k+1} \mathbb{C}^{n+1}) \cong \mathbb{P}^{n_k}$ the canonical projection through the origin $0 \in \mathbb{C}^{n_k+1} \cong \bigwedge^{k+1} \mathbb{C}^{n+1}$, where $n_k = \binom{n+1}{k+1}-1$.
Assume that the image $f(\mathbb{C})$ does not lie in any linear subspace $\mathbb{P}^{k-1} \subset \mathbb{P}^n$.
Then one defines a holomorphic curve
\[
f^k \colon \mathbb{C} \longrightarrow \textnormal{P}(\bigwedge^{k+1} \mathbb{C}^{n+1}) \cong \mathbb{P}^{n_k}
,\]
to be the analytic extension of $\pr \circ F_{\restriction U}^k$, where $U = \mathbb{C} \setminus (F^k)^{-1}(0)$.
Take the blow-up of $\mathbb{C}^{n_k+1}$ at the origin
\[
\upsilon_0 \colon \widetilde{\mathbb{C}}^{n_k+1} \bydef \Bl_{\mathcal{I}_0}(\mathbb{C}^{n_k+1}) \longrightarrow \mathbb{C}^{n_k+1}
,\]
where $\mathcal{I}_0 = \{s \in \mathcal{O}_{\mathbb{C}^{n_k+1}} \mid s(0) = 0\}$.
It is evident that the rational map $\mathbb{C}^{n_k+1} \dashrightarrow \mathbb{P}^{n_k}$ is resolved into $\widetilde{\pr} \colon \widetilde{\mathbb{C}}^{n_k+1} \to \mathbb{P}^{n_k}$ by the blow-up $\upsilon_0$.
If one denotes the lifting of $F^k$ to $\widetilde{\mathbb{C}}^{n_k+1}$ by $\widetilde{F}^k$, then $f^k = \widetilde{\pr} \circ \widetilde{F}^k$, which is called the $k$-th derived curve of $f$ by Fujimoto \cite{Fujimoto1993}.

One says that $\bm{v} \in \bigwedge^{k+1} \mathbb{C}^{n+1}$ is a decomposable $(k+1)$-vector if it can be written as $\bm{v} = v_0 \wedge \dots \wedge v_k$ with $k+1$ vectors $v_i \in \mathbb{C}^{n+1}$ for $0 \leq i \leq k$.
Observe that every non-zero decomposable $(k+1)$-vector $\bm{v} \in \bigwedge^{k+1} \mathbb{C}^{n+1}$ determines an element $S$ in $\Gr(k+1, n+1)$,
and $\pr(\bm{v}) = \Pluc(S)$, where one abusively writes the Plücker embedding $\Pluc \colon \Gr(k+1, n+1) \hookrightarrow \mathbb{P}^{n_k}$.
By definition, all the image of $F_k$ are decomposable, thus $f^k$ factors through the Plücker embedding into a holomorphic curve $f_k$ in the Grassmannian, namely,
\begin{equation}\label{eqn:Plucfk}
f^k = \Pluc \circ f_k
.\end{equation}
Thinking of $\Gr(k+1, n+1)$ as a submanifold of $\textnormal{P}(\bigwedge^{k+1} \mathbb{C}^{n+1})$, Griffiths and Harris \cite{Griffiths1978} regard $f^k$ and $f_k$ as one holomorphic curve, which is referred to as the $k$-th associated curve of $f$.

Let $\omega_k = \ddif^c \log \|x\|^2$ be the Kähler form associated to the Fubini-Study metric on $\mathbb{P}^{n_k}$, where $\|x\|^2 = \sum_{i=0}^{n_k} |x_i|^2$ for a reduced representation $[x_0 : \dots : x_{n_k}]$ of $x \in \mathbb{P}^{n_k}$.
By calculation, one obtains the infinitesimal Plücker formula (see Chapter 2, Section 4.3 of \cite{Griffiths1978} for details)
\begin{equation}\label{eqn:Plücker}
(f^k)^{\ast} \omega_k = \frac{|F^{k-1}|^2 |F^{k+1}|^2}{|F^k|^4} \ddif^c |z|^2
,\end{equation}
for $z \in \mathbb{C}$, where we make a convention that $F^{-1} = 1$.
One writes $h_k = \dfrac{|F^{k-1}|^2 |F^{k+1}|^2}{|F^k|^4}$ and $\Omega_k = (f^k)^{\ast} \omega_k$.
Hence it follows from the Poincaré-Lelong formula that
\[
\ddif^c [\log h_k] = [\Omega_{k-1}] + [\Omega_{k+1}] - 2[\Omega_k] + \frac{1}{2} [Z_{h_k}]
,\]
where $Z_{h_k}$ is the closed analytic subspace of $\mathbb{C}$ defined by the equation $h_k = 0$.
Given a point $z_0 \in \mathbb{C}$, by choosing a suitable reduced representation of $f$, one writes $f$ in normal form near $z_0$
\[
f(z) = [1+\dots : (z-z_0)^{\nu_1+1}+\dots : (z-z_0)^{\nu_1+\nu_2+2}+\dots : \dots : (z-z_0)^{\nu_1+\dots+\nu_n+n}+\cdots]
,\]
with $\nu_i \in \mathbb{N}$ for all $1 \leq i \leq n$.
One easily checks that $F^k$ has a zero of order $k\nu_1 + (k-1)\nu_2 + \dots + \nu_k$ at $z_0$.
Here $F^k$ is said to have a zero of order $0$ at $z_0$, if $F^k(z_0) \neq 0$.
Therefore, $h_k$ has a zero of order $2\nu_{k+1}$ at $z_0$.

\subsection{Ahlfors' lemma on logarithmic derivative in projective spaces}
In this subsection, we are going to recall Ahlfors' Lemma on logarithmic derivative (Ahlfors' LLD for short) for the hyperplanes of $\mathbb{P}^{n}$.

Let us briefly introduce the $k$-th contact function of $f$ (see Section 2.3 of \cite{Fujimoto1993} or Section A\! 3.5 of \cite{Ru2001} for more details).
To proceed, we need the interior product of exterior algebra.

\begin{definition}
Take two vectors $\alpha \in \bigwedge^{p+1}\mathbb{C}^{n+1}$ and $\beta \in \bigwedge^{q+1}(\mathbb{C}^{n+1})^{\dual}$, where $0 \leq q \leq p \leq n$.
The interior product $\alpha \,\llcorner\, \beta$ is defined to be the unique vector in $\bigwedge^{p-q}\mathbb{C}^{n+1}$ satisfying
\begin{equation}\label{eqn:interior}
\gamma(\alpha \,\llcorner\, \beta) = (\beta \wedge \gamma)(\alpha)
,\end{equation}
for all $\gamma \in \bigwedge^{p-q}(\mathbb{C}^{n+1})^{\dual}$.
In the particular case where $p=q$, one has $\alpha \,\llcorner\, \beta = \beta(\alpha)$.
\end{definition}

Observe that each global section $\sigma \in H^0(\mathbb{P}^n, \mathcal{O}_{\mathbb{P}^n}(1))$ can be regarded as a linear functional of $\mathbb{C}^{n+1}$, which amounts to the isomorphism $H^0(\mathbb{P}^n, \mathcal{O}_{\mathbb{P}^n}(1)) \cong (\mathbb{C}^{n+1})^{\dual}$.
More concretely, for a reduced representation $[x_0 : \dots : x_n]$ of $x \in \mathbb{P}^n$, one can write $\sigma(x) = a_0x_0 + \dots +a_nx_n$, which implies that $\sigma$ can be represented by $\bm{a}=(a_0, \dots, a_n) \in (\mathbb{C}^{n+1})^{\dual}$.
For a hyperplane $H = (\sigma)$, one sees that $\bm{a}$ is indeed a normal vector of $H$.
Let $H$ be the hyperplane with a unit normal vector $\bm{a}$.
Thereupon the $k$-th contact function of $f$ for $H$ is defined by
\[
\phi_k(H) \bydef \frac{|F^k \,\llcorner\, \bm{a}|^2}{|F^k|^2}
,\]
for each $0 \leq k \leq n$.

Accordingly, the following lemma is known as Ahlfors' LLD.
\begin{lemma}[Ahlfors' LLD]\cite{Ahlfors1941}\label{lem:ALLD_1^1}
Let $H$ be a hyperplane of $\mathbb{P}^{n}$, and let $f \colon \mathbb{C} \to \mathbb{P}^n$ be a non-constant holomorphic curve such that $f(\mathbb{C}) \not\subset H$.
Then for any $0 < \varepsilon < 1$, one has
\begin{equation}
\varepsilon \mathcal{N}\left( \frac{\phi_1(H)}{\phi_0(H)^{1-\varepsilon}} \Omega_0, r \right) \leq (1+\varepsilon) T_f(r, \mathcal{O}_{\mathbb{P}^{n}}(1)) + O(1)
.\end{equation}
\end{lemma}

As a reminder, a linear subspace of $\mathbb{P}^n$ is said to be a closed subspace of $\mathbb{P}^n$ defined by linear homogeneous equations.
For any linear subspace $H^{\ell}$ of $\mathbb{P}^n$ of codimension $\ell+1$, it is easy to find $\ell + 1$ hyperplanes $H_0, \dots, H_{\ell}$ in $\mathbb{P}^n$ such that $H^{\ell} = H_0 \cap \dots \cap H_{\ell}$.
Viewing $H^{\ell}$ as a hyperplane in $\textnormal{P}(\bigwedge^{\ell+1} \mathbb{C}^{n+1})$, one sees that $\bm{a}_0 \wedge \dots \wedge \bm{a}_{\ell} \in \bigwedge^{\ell+1}(\mathbb{C}^{n+1})^{\dual}$ is a normal vector of $H^{\ell}$, where $\bm{a}_i$ is a normal vector of $H_i$ for $0 \leq i \leq \ell$.
Actually, each linear subspace $H^{\ell}$ one-to-one corresponds to a unit normal vector $\alpha \in \bigwedge^{\ell+1}(\mathbb{C}^{n+1})^{\dual}$.
Let $H^{\ell}$ be the hyperplane in $\textnormal{P}(\bigwedge^{\ell+1} \mathbb{C}^{n+1})$ with a unit normal vector $\alpha \in \bigwedge^{\ell+1}(\mathbb{C}^{n+1})^{\dual}$.
Analogously, the $k$-th contact function of $f$ for $H^{\ell}$ is defined by
\[
\phi_k(H^{\ell}) \bydef \frac{|F^k \,\llcorner\, \alpha|^2}{|F^k|^2}
,\]
for each $0 \leq \ell \leq k \leq n$.

Take a linear subspace $H^{\ell-1}$ of $\mathbb{P}^n$ of codimension $\ell$ with a unit normal vector $\beta \in \bigwedge^{\ell}(\mathbb{C}^{n+1})^{\dual}$.
Let $f_{H^{\ell-1}} \colon \mathbb{C} \to \textnormal{P}((H^{\ell-1})^{\perp}) \cong \mathbb{P}^{n-\ell}$ be the analytic extension of $\pr \circ (F^{\ell} \,\llcorner\, \beta)_{\restriction U}$, where $U = \mathbb{C} \setminus (F^{\ell} \,\llcorner\, \beta)^{-1}(0)$.
Fujimoto showed that $f_{H^{\ell-1}}$ is linear non-degenerate if $f$ is so (see Proposition 2.6.2 in \cite{Fujimoto1993}).
And he called it the contracted curve of $f$ to the direction $H^{\ell-1}$ (see Section 2.6 of \cite{Fujimoto1993} for more details).

Suppose that $H^{\ell}$ is a linear subspace of $\mathbb{P}^n$ of codimension $\ell + 1$ such that $H^{\ell} = H^{\ell-1} \cap H_0$ for some hyperplane $H_0$ of $\mathbb{P}^n$.
Then one can obtains the following formula (see also (2.6.10) in \cite{Fujimoto1993}):
\[
\phi_k(H^{\ell}) = \bar{\phi}_{k-\ell}(H_0) \phi_k(H^{\ell-1})
,\]
for each $\ell \leq k \leq n$, where $\bar{\phi}_{k-\ell}(H_0)$ is the $(k-\ell)$-th contact function of $f_{H^{\ell-1}}$ for $H_0$.

By induction, one has Ahlfors' LLD for $(k+1)$-jets intersecting linear subspaces of codimension $\ell+1$.
\begin{lemma}\cite{Ahlfors1941}\label{lem:ALLD}
Let $H^{\ell}$ be a complete intersection of $\ell + 1$ hyperplanes in $\mathbb{P}^n$, for $0 \leq \ell \leq n-1$.
Let $f \colon \mathbb{C} \to \mathbb{P}^n$ be a non-constant holomorphic curve such that $f(\mathbb{C}) \not\subset H^{\ell}$.
For each $0 \leq k \leq n-1$, suppose that the image $f(\mathbb{C})$ does not lie in any linear subspace $\mathbb{P}^k \subset \mathbb{P}^n$ besides.
Then there is a constant $c_{k-\ell}$ such that for any $0 < \varepsilon < 1$, one has
\[
\varepsilon \mathcal{N}\left( \frac{\phi_{k+1}(H^{\ell})}{\phi_k(H^{\ell})^{1-\varepsilon}} \Omega_k, r \right) \leq c_{k-\ell} T_{f^k}(r, \mathcal{O}_{\mathbb{P}^{n_k}}(1)) + O(1)
.\]
\end{lemma}

\subsection{Maps to the Grassmannian}
In the following three sections, we will only consider the case of intersecting divisors for simplicity.

Let $(X, V)$ be a directed projective manifold, and choose a linear system $\mathfrak{S} \subset |L|$ on $X$.
Suppose that $\mathfrak{S}$ is base point free, i.e., $\Bs(\mathfrak{S}) = \emptyset$.
Accordingly, $\mu_{\mathfrak S}$ defined by \eqref{eqn:muS} is indeed a holomorphic map as follows.
\begin{align*}
\mu_{\mathfrak S} \colon X &\longrightarrow \mathfrak{S}^{\dual} \cong \mathbb{P}^N \\
x &\longmapsto [\varsigma_0(x) : \dots : \varsigma_N(x)]
,\end{align*}
where $N = \dim \mathfrak{S}$, $\{\varsigma_i \mid i = 0, 1, \dots, N\}$ is a basis of $\mathbb{S}$, and $\mathbb{S} \subset H^0(X, L)$ is the linear subspace corresponding to $\mathfrak{S}$.
Obviously, if $L$ is very ample and $\mathfrak{S}$ is the complete linear system $|L|$, then $\mu_{\mathfrak S}$ is a closed embedding.

Let $f \colon (\mathbb{C}, T_{\mathbb{C}}) \to (X, V)$ be a non-constant holomorphic curve.
It is evident that $[\varsigma_0 \circ f : \dots : \varsigma_N \circ f]$ is a reduced representation of $\tilde{f}$, where $\tilde{f} \bydef \mu_{\mathfrak S} \circ f$ is a non-constant holomorphic curve in $\mathbb{P}^N$.
In the particular case where $(X, V) = (\mathbb{P}^n, T_{\mathbb{P}^n})$ and $\mathfrak{S} = |H|$ for some hyperplane $H$ of $\mathbb{P}^n$,
we see that $\mu_{\mathfrak S}$ is indeed an isomorphism and $[\varsigma_0 \circ f : \dots : \varsigma_n \circ f]$ gives a reduced representation of $f$.

Analogous to \eqref{eqn:F^k}, we construct a holomorphic curve $\widetilde{F}^k$ in $\bigwedge^{k+1} \mathbb{C}^{N+1}$ associated to $\tilde{f}$ as follows.
Take the standard basis $\{\bm{e}^i \mid 0 \leq i \leq N\}$ of $\mathbb{C}^{N+1}$, namely, $\bm{e}^i = (\delta_1^{i+1}, \dots, \delta_{N+1}^{i+1}) \in \mathbb{C}^{N+1}$ for each $0 \leq i \leq N$.
Here $\delta_j^i$ is the Kronecker symbol, defined by
\[
\delta_j^i \bydef
\begin{cases}
1 & i=j, \\
0 & i \neq j
,\end{cases}
\]
where $i, j \in \mathbb{Z}_{+}$.
Consequently, we get
\[
\widetilde{F}(z) = \sum_{i=0}^N \varsigma_i \circ f(z) \bm{e}^i
.\]
Moreover, $\{\bm{e}^{i_0} \wedge \dots \wedge \bm{e}^{i_k} \mid 0 \leq i_0 < \dots < i_k \leq N\}$ is also a basis of $\bigwedge^{k+1} \mathbb{C}^{N+1}$.
By computation, we obtain
\begin{equation}\label{eqn:tildeF^k}
\widetilde{F}^k = \sum_{0 \leq i_0 < \dots < i_k \leq N} \hspace{-1eM} W(\varsigma_{i_0}, \dots, \varsigma_{i_k})(j_kf) \bm{e}^{i_0} \wedge \dots \wedge \bm{e}^{i_k}
.\end{equation}

Inspired by the $k$-th derived curve and \eqref{eqn:tildeF^k}, for each $1 \leq k \leq N$, we define a rational map
\begin{align*}
\varpi_{k, \mathfrak{S}} \colon X_k &\dashrightarrow \textnormal{P}(\bigwedge^{k+1} \mathbb{S}^{\dual}) \cong \mathbb{P}^{N_k} \\
x_k &\longmapsto [\widebar{\omega}(\varsigma_{i_0}, \dots, \varsigma_{i_k})(x_k)]_{0 \leq i_0 < \dots < i_k \leq N}
,\end{align*}
where $N_k = \binom{N+1}{k+1}-1$.
Apparently, the indeterminacy locus of $\varpi_{k, \mathfrak{S}}$ lies in the base locus $\Bs(\widebar{\mathfrak W}(X_k, \mathfrak{S})) \subset X_k$.
Assume that the lifting $f_{[k]}$ of $f$ to $(X_k, V_k)$ satisfies that $f_{[k]}(\mathbb{C}) \not\subset \Bs(\mathfrak{W}(X_k, \mathfrak{S}))$.
Consider the $k$-th derived curve $\tilde{f}^k$ of $\tilde{f}$.
It follows from \eqref{eqn:tildeF^k} that
\begin{equation}\label{eqn:f^kU}
\tilde{f}_{\restriction U}^k = \varpi_{k, \mathfrak{S}} \circ f_{[k] \restriction U}
,\end{equation}
where $U = \mathbb{C} \setminus f_{[k]}^{-1}(\Bs(\mathfrak{W}(X_k, \mathfrak{S})))$.

As mentioned above, the $k$-th derived curve factors through the Plücker embedding into a holomorphic curve in the Grassmannian, which is known as the $k$-th associated curve of $f$.
Likewise, we are going to show that $\varpi_{k, \mathfrak{S}}$ factors through the Plücker embedding as well.

Take a fixed trivialization of $L$ on an admissible open subset $U \subset X$.
By Corollary \ref{cor:regular}, there exists an open subset $U_k \subset \pi_{0, k}^{-1}(U)$, which allows us to take a germ of curves $f_{x_k}$ tangent to $V$ for every $x_k \in U_k$, depending holomorphically on $x_k$, such that $(f_{x_k})_{[k]}(0)=x_k$ and $(f_{x_k})'_{[k-1]}(0) \neq 0$.
For any $0 \leq \ell \leq k$, we denote
\[
\dif_{U, \mathfrak{S}}^{[\ell]}(x_k) \bydef \left( \dif^{[\ell]}\varsigma_{i, U}(j_kf_{x_k}) \right)_{0 \leq i \leq N} = \left( (\varsigma_{i, U} \circ f_{x_k})^{(\ell)}(0) \right)_{0 \leq i \leq N} \in \mathbb{S}^{\dual}
.\]
Consider the rational map
\begin{align*}
\mu_{U_k, \mathfrak{S}} \colon U_k &\dashrightarrow \Gr(k, \mathfrak{S}^{\dual}) \cong \Gr(k+1, N+1) \\
x_k &\longmapsto \Span\left(\dif_{U, \mathfrak{S}}^{[0]}(x_k), \dots, \dif_{U, \mathfrak{S}}^{[k]}(x_k)\right)
.\end{align*}
The dependence on the choice of $f_{x_k}$ and the trivialization of $L_{\restriction U}$ notwithstanding, we have
\[
\varpi_{k, \mathfrak{S} \restriction U_k} = \Pluc \circ \mu_{U_k, \mathfrak{S}}
,\]
where $\Pluc \colon \Gr(k+1, N+1) \hookrightarrow \mathbb{P}^{N_k}$ is, by abuse of notation, the Plücker embedding.
This verifies that $\varpi_{k, \mathfrak{S}}$ factors through the Plücker embedding into a rational map
\[
\mu_{k, \mathfrak{S}} \colon X_k \dashrightarrow \Gr(k, \mathfrak{S}^{\dual}) \cong \Gr(k+1, N+1)
.\]
From \eqref{eqn:Plucfk} and \eqref{eqn:f^kU}, we immediately get
\begin{equation}\label{eqn:f_kU}
\tilde{f}_{k \restriction U} = \mu_{k, \mathfrak{S}} \circ f_{[k] \restriction U}
,\end{equation}
where $\tilde{f}_k$ is the $k$-th associated curve of $\tilde{f}$, and $U$ is defined above.

Furthermore, observe that $\Bs(\widebar{\mathfrak W}(X_k, \mathfrak{S})) \neq \emptyset$ holds for each $k \geq 2$.
Take the blow-up of $X_k$ with respect to $\widebar{\mathfrak w}(X_k, \mathfrak{S})$
\[
\upsilon_{k, \mathfrak{S}} \colon \widehat{X}_{k, \mathfrak{S}} \bydef \Bl_{\widebar{\mathfrak w}(X_k, \mathfrak{S})}(X_k) \longrightarrow X_k
.\]
Let $E_{k, \mathfrak{S}} \bydef \upsilon_{k, \mathfrak{S}}^{\ast} Z_{\widebar{\mathfrak w}(X_k, \mathfrak{S})}$ be the exceptional divisor of $\upsilon_{k, \mathfrak{S}}$, where $Z_{\widebar{\mathfrak w}(X_k, \mathfrak{S})}$ is the closed subscheme of $X_k$ with the ideal sheaf $\widebar{\mathfrak w}(X_k, \mathfrak{S})$.
Definitely, $Z_{\widebar{\mathfrak w}(X_k, \mathfrak{S})}$ is the scheme of base points of $\widebar{\mathfrak W}(X_k, \mathfrak{S})$.
Hence $E_{k, \mathfrak{S}}$ is the scheme of base points of $\upsilon_{k, \mathfrak{S}}^{\ast} \widebar{\mathfrak W}(X_k, \mathfrak{S})$,
which implies that $\widehat{\mathfrak W}(X_k, \mathfrak{S}) \bydef (\upsilon_{k, \mathfrak{S}}^{\ast} \widebar{\mathfrak W}(X_k, \mathfrak{S}))_{\setminus E_{k, \mathfrak{S}}}$ is base point free.

Seeing that $\upsilon_{k, \mathfrak{S}}$ resolves $\varpi_{k, \mathfrak{S}}$, along with $\mu_{k, \mathfrak{S}}$, we obtain a commutative diagram as follows:
\begin{equation*}
\begin{tikzcd}[column sep=scriptsize]
\widehat{X}_{k, \mathfrak{S}} \arrow[r, "\hat{\mu}_{k, \mathfrak{S}}"] \arrow[d, "\upsilon_{k, \mathfrak{S}}"'] & \Gr(k, \mathfrak{S}^{\dual}) \arrow[d, shift right=.5ex, "\Pluc"]  \\
X_k \arrow[ru, dashed, "\mu_{k, \mathfrak{S}}"]  \arrow[r, dashed, "\varpi_{k, \mathfrak{S}}"'] & \textnormal{P}(\bigwedge^{k+1} \mathbb{S}^{\dual})
.\end{tikzcd}
\end{equation*}
Denote the lifting of $f_{[k]}$ to $\widehat{X}_{k, \mathfrak{S}}$ by $\hat{f}_{[k]}$.
Thus \eqref{eqn:f_kU} yields
\[
\tilde{f}_k = \hat{\mu}_{k, \mathfrak{S}} \circ \hat{f}_{[k]}
.\]
Let $\mathcal{O}_{\Gr}(1) \bydef \Pluc^{\ast} \mathcal{O}_{\mathbb{P}^{N_k}}(1)$ be the hyperplane line bundle on $\Gr(k, \mathfrak{S}^{\dual})$.
It is evident that
\begin{equation}\label{eqn:hatW}
\widehat{\mathfrak W}(X_k, \mathfrak{S}) = \hat{\mu}_{k, \mathfrak{S}}^{\ast} |\mathcal{O}_{\Gr}(1)|
,\end{equation}
and
\begin{equation}\label{eqn:Ek}
\upsilon_{k, \mathfrak{S}}^{\ast} (\mathcal{O}_{X_k}(\bm{a}^k) \otimes \pi_{0, k}^{\ast} L^{k+1}) \otimes \mathcal{O}_{\widehat{X}_{k, \mathfrak{S}}}(-E_{k, \mathfrak{S}}) \cong \hat{\mu}_{k, \mathfrak{S}}^{\ast} \mathcal{O}_{\Gr}(1)
.\end{equation}

\subsection{Generalization of the infinitesimal Plücker formula}

For each $1 \leq k \leq N$, set
\begin{equation}\label{eqn:Lk}
L_k = \upsilon_{k, \mathfrak{S}}^{\ast} (\mathcal{O}_{X_k}(\bm{a}^k) \otimes \pi_{0, k}^{\ast} L^{k+1}) \otimes \mathcal{O}_{\widehat{X}_{k, \mathfrak{S}}}(-E_{k, \mathfrak{S}})
\end{equation}
and $L_0 = L$.
It follows from \eqref{eqn:hatW} and \eqref{eqn:Ek} that $\widehat{\mathfrak W}(X_k, \mathfrak{S}) = |L_k|$, and for any global sections $\sigma_0, \dots, \sigma_k \in H^0(X, L)$, there exists a global section $\widehat{\omega}(\sigma_0, \dots, \sigma_k) \in H^0(\widehat{X}_k, L_k)$ such that $\widehat{\sigma}_{k, \mathfrak{S}} \cdot \widehat{\omega}(\sigma_0, \dots, \sigma_k) = \upsilon_{k, \mathfrak{S}}^{\ast} \widebar{\omega}(\sigma_0, \dots, \sigma_k)$ for a canonical section $\widehat{\sigma}_{k, \mathfrak{S}}$ of $\mathcal{O}_{\widehat{X}_k}(E_{k, \mathfrak{S}})$.
Consider the Fubini-Study metric on $\mathbb{P}^{N_k}$, which can be regarded as a Hermitian metric $\tilde{h}_{FS}$ on hyperplane line bundle $\mathcal{O}_{\mathbb{P}^{N_k}}(1)$.
We naturally get a canonical Hermitian metric $\Pluc^{\ast} \tilde{h}_{FS}$ on $\mathcal{O}_{\Gr}(1)$.
What's more, $\hat{\mu}_{k, \mathfrak{S}}^{\ast} \Pluc^{\ast} \tilde{h}_{FS}$ is a smooth Hermitian metric on $L_k$,
whose weight function $\varphi_{\widehat{\mathfrak W}(X_k, \mathfrak{S})}$ is defined as \eqref{eqn:phi_S}.
More concretely,
\[
\varphi_{\widehat{\mathfrak W}(X_k, \mathfrak{S})} = \log \left(\sum_{0 \leq i_0 < \dots < i_k \leq n} \hspace{-1eM} |\widehat{\omega}(\varsigma_{i_0}, \dots, \varsigma_{i_k})|^2\right)
.\]
This implies that
\[
(\tilde{f}^k)^{\ast} \widetilde{\omega}_k = \hat{f}_{[k]}^{\ast} \ddif^c \varphi_{\widehat{\mathfrak W}(X_k, \mathfrak{S})}
,\]
where $\widetilde{\omega}_k$ is the Kähler form associated to the Fubini-Study metric on $\mathbb{P}^{N_k}$.
In addition, $\varpi_{k, \mathfrak{S}}^{\ast} \tilde{h}_{FS}$\footnote{Here the ``pull-back'' by a meromorphic map is defined as (2.3) in \cite{Shiffman1976}.} gives a singular Hermitian metric on $\mathcal{O}_{X_k}(\bm{a}^k) \otimes \pi_{0, k}^{\ast} L^{k+1}$ with the weight function $\varphi_{\widebar{\mathfrak W}(X_k, \mathfrak{S})}$.

Regarding $\varphi_{\mathfrak{W}(X_k, \mathfrak{S})}$ as a function on $\mathbb{W}(X_k, \mathfrak{S})^{\dual} \hookrightarrow \bigwedge^{k+1} \mathbb{C}^{N+1}$, we see from \eqref{eqn:wf'} and \eqref{eqn:tildeF^k} that
\begin{equation}\label{eqn:normtildeF^k}
|\widetilde{F}^k|^2 = \exp(\varphi_{\mathfrak{W}(X_k, \mathfrak{S})} \circ f_{[k]}) |f_{[k-1]}'|^{2k'}
\end{equation}
holds on the complex plane except stationary points, where $f_{[k-1]}'$ is defined by \eqref{eqn:fk-1'}.
Take account of the formula
\[
f_{[k-2]}' = \dif \pi_{k-1}(f_{[k-1]}) \cdot f_{[k-1]}'
.\]
By the construction of the line bundle morphism \eqref{eqn:Gamma} and the definition of $\varGamma_k$, the above formula yields
\begin{equation}\label{eqn:Gammafk-1'}
f_{[k-2]}' = (\sigma_{\varGamma_k} \circ f_{[k]}) \cdot f_{[k-1]}'
,\end{equation}
where $\sigma_{\varGamma_k}$ is a canonical section of $\mathcal{O}_{X_k}(\varGamma_k)$.
Thanks to \eqref{eqn:bkGamma} and \eqref{eqn:Gammafk-1'}, \eqref{eqn:normtildeF^k} can be also formulated as
\begin{equation}\label{eqn:normtildeF^kbar}
|\widetilde{F}^k|^2 = \exp(\varphi_{\widebar{\mathfrak W}(X_k, \mathfrak{S})} \circ f_{[k]}) |f'|^{2k} |f_{[1]}'|^{2(k-1)} \cdots |f_{[k-1]}'|^2
.\end{equation}
Indeed, one can check that \eqref{eqn:normtildeF^kbar} holds even at stationary points.
Therefore, the infinitesimal Plücker formula
\[
(\tilde{f}^k)^{\ast} \widetilde{\omega}_k = \frac{|\widetilde{F}^{k-1}|^2 |\widetilde{F}^{k+1}|^2}{|\widetilde{F}^k|^4} \ddif^c |z|^2
\]
on $\mathcal{O}_{\mathbb{P}^{N_k}}(1)$ is reformulated as
\[
\hat{f}_{[k]}^{\ast} \ddif^c \varphi_{\widehat{\mathfrak W}(X_k, \mathfrak{S})} = \exp(\varphi_{\widebar{\mathfrak W}(X_{k-1}, \mathfrak{S})} \circ f_{[k-1]} + \varphi_{\widebar{\mathfrak W}(X_{k+1}, \mathfrak{S})} \circ f_{[k+1]} - 2 \varphi_{\widebar{\mathfrak W}(X_k, \mathfrak{S})} \circ f_{[k]}) |f_{[k]}'|^2 \ddif^c |z|^2
,\]
where we make the conventions that $\varphi_{\widebar{\mathfrak W}(X_{-1}, \mathfrak{S})} = 0$ and $\widebar{\mathfrak W}(X_0, \mathfrak{S}) = \mathfrak{S}$.
For shortness, we write
\begin{equation}\label{eqn:gammakS}
\gamma_{k, \mathfrak{S}} = \exp(\varphi_{\widebar{\mathfrak W}(X_{k-1}, \mathfrak{S})} \circ f_{[k-1]} + \varphi_{\widebar{\mathfrak W}(X_{k+1}, \mathfrak{S})} \circ f_{[k+1]} - 2 \varphi_{\widebar{\mathfrak W}(X_k, \mathfrak{S})} \circ f_{[k]}) |f_{[k]}'|^2
\end{equation}
Using the Poincaré-Lelong formula, \eqref{eqn:normtildeF^kbar} leads to
\begin{align*}
\ddif^c [\log |\widetilde{F}^k|^2] &= f_{[k]}^{\ast} \ddif^c [\varphi_{\widebar{\mathfrak W}(X_k, \mathfrak{S})}] + \sum_{\ell=0}^{k-1} (k-\ell) [Z_{f'_{[\ell]}}] \\
&= f_{[k]}^{\ast} [\ddif^c \varphi_{\widebar{\mathfrak W}(X_k, \mathfrak{S})}] + f_{[k]}^{\ast} [Z_{\widebar{\mathfrak w}(X_k, \mathfrak{S})}] + \sum_{\ell=0}^{k-1} (k-\ell) [Z_{f'_{[\ell]}}] \\
&= \hat{f}_{[k]}^{\ast} \ddif^c [\varphi_{\widehat{\mathfrak W}(X_k, \mathfrak{S})}] + f_{[k]}^{\ast} [Z_{\widebar{\mathfrak w}(X_k, \mathfrak{S})}] + \sum_{\ell=0}^{k-1} (k-\ell) [Z_{f'_{[\ell]}}]
.\end{align*}
Accordingly,
\[
\ddif^c [\log \gamma_{k, \mathfrak{S}}] = \hat{f}_{[k-1]}^{\ast} \ddif^c [\varphi_{\widehat{\mathfrak W}(X_{k-1}, \mathfrak{S})}] + \hat{f}_{[k+1]}^{\ast} \ddif^c [\varphi_{\widehat{\mathfrak W}(X_{k+1}, \mathfrak{S})}] - 2 \hat{f}_{[k]}^{\ast} \ddif^c [\varphi_{\widehat{\mathfrak W}(X_k, \mathfrak{S})}] + \frac{1}{2} [Z_{\gamma_{k, \mathfrak{S}}}]
,\]
where
\[
\frac{1}{2} [Z_{\gamma_{k, \mathfrak{S}}}] = f_{[k-1]}^{\ast} [Z_{\widebar{\mathfrak w}(X_{k-1}, \mathfrak{S})}] + f_{[k+1]}^{\ast} [Z_{\widebar{\mathfrak w}(X_{k+1}, \mathfrak{S})}] - f_{[k]}^{\ast} 2 [Z_{\widebar{\mathfrak w}(X_k, \mathfrak{S})}] + [Z_{f'_{[k]}}]
.\]

Now, consider the case where $k=0$. We hereinafter prefer to write $\gamma_{\mathfrak S}$ rather than $\gamma_{0, \mathfrak{S}}$.
Thanks to the assumption $\Bs(\mathfrak{S}) = \emptyset$, $\mu_{\mathfrak S}^{\ast} \tilde{h}_{FS}$ gives a smooth Hermitian metric on $L$ with the weight function $\varphi_{\mathfrak S}$.
In this case, the above formula is simplified as
\[
\ddif^c [\log \gamma_{\mathfrak S}] = \hat{f}_{[1]}^{\ast} \ddif^c [\varphi_{\widehat{\mathfrak W}(X_1, \mathfrak{S})}] - 2 f^{\ast} \ddif^c [\varphi_{\mathfrak S}] + \frac{1}{2} [Z_{\gamma_{\mathfrak S}}]
,\]
where $\dfrac{1}{2} [Z_{\gamma_{\mathfrak S}}] = f_{[1]}^{\ast} [Z_{\mathfrak{w}(X_1, \mathfrak{S})}] + [Z_{f'}]$.
Consequently, we obtain
\[
\mathcal{N}(\ddif^c [\log \gamma_{\mathfrak S}], r) = T_{\hat{f}_{[1]}}(r, \upsilon_{1, \mathfrak{S}}^{\ast} \mathcal{O}_{X_1}(1) \otimes \mathcal{O}(-E_{1, \mathfrak{S}})) + N_{f_{[1]}}(r, Z_{\mathfrak{w}(X_1, \mathfrak{S})}) + \mathcal{N}([Z_{f'}], r)
.\]
In particular, if $\mathfrak{S}$ separates 1-jets (e.g., $\mathfrak{S} = |L|$ for some very ample line bundle $L$), then both $\varpi_{1, \mathfrak{S}}$ and $\mu_{1, \mathfrak{S}}$ are holomorphic.
Moreover, $Z_{\mathfrak{w}(X_1, \mathfrak{S})} = \emptyset$, $E_{1, \mathfrak{S}} = \emptyset$ and $\upsilon_{1, \mathfrak{S}}$ is indeed the identity map.
In this case, we get
\begin{equation}\label{eqn:Nddclogg}
\mathcal{N}(\ddif^c [\log \gamma_{\mathfrak S}], r) = T_{f_{[1]}}(r, \mathcal{O}_{X_1}(1)) + \mathcal{N}([Z_{f'}], r)
.\end{equation}
On the other hand, without loss of generality, suppose that $\gamma_{\mathfrak S}(0) \neq 0$, otherwise we can deal with it in the similar way to Remark \ref{rmk:zero}.
By the Green-Jensen formula, Jensen inequality and Lemma \ref{lem:calculus},
\begin{align*}
\mathcal{N}(\ddif^c [\log \gamma_{\mathfrak S}], r) &= \frac{1}{2} \left( \int_0^{2\pi} \log \gamma_{\mathfrak S}(re^{i\theta}) \frac{\dif \theta}{2\pi} - \log \gamma_{\mathfrak S}(0) \right) \\
&\leq \frac{1}{2} \log \int_0^{2\pi} \gamma_{\mathfrak S}(re^{i\theta}) \frac{\dif \theta}{2\pi} - \frac{1}{2} \log \gamma_{\mathfrak S}(0) \\
&\leq \frac{(1+\delta)^2}{2} \log \mathcal{N}([\gamma_{\mathfrak S} \ddif^c |z|^2], r) - \left. \frac{1}{2} \log \gamma_{\mathfrak S}(0) \right\rVert_{E} \\
&\leq \frac{(1+\delta)^2}{2} \log T_f(r, L) - \left. \frac{1}{2} \log \gamma_{\mathfrak S}(0) \right\rVert_{E}
.\end{align*}
Especially, if $f$ is an algebraic curve, we see that $\gamma_{\mathfrak S}(re^{i\theta}) \leq O(1)$ as $r \to \infty$.
Thus
\[
\mathcal{N}(\ddif^c [\log \gamma_{\mathfrak S}], r) = \frac{1}{2} \left( \int_0^{2\pi} \log \gamma_{\mathfrak S}(re^{i\theta}) \frac{\dif \theta}{2\pi} - \log \gamma_{\mathfrak S}(0) \right) \leq O(1) \quad (r \to \infty)
.\]

Combined with \eqref{eqn:Nddclogg}, these inequalities yield
\begin{equation}\label{eqn:O_X1(1)}
T_{f_{[1]}}(r, \mathcal{O}_{X_1}(1)) + \mathcal{N}([Z_{f'}], r) \leq S_f(r)
.\end{equation}
Observe that $\mathcal{O}_{X_1}(1) \otimes \pi_1^{\ast} A$ is ample for some sufficiently ample line bundle $A$ on $X$.
According to \eqref{eqn:O_X1(1)} and Proposition \ref{prp:TrA}, $S_{f_{[1]}}(r) \leq S_f(r)$.
By induction, for each $k \in \mathbb{Z}_{+}$, hence we have
\begin{equation}\label{eqn:Sfk}
S_{f_{[k]}}(r) \leq S_f(r)
.\end{equation}
For each $k \in \mathbb{Z}_{+}$, think of $(X_k, V_k)$ as the Demailly-Semple $1$-jet tower of $(X_{k-1}, V_{k-1})$.
Applying \eqref{eqn:O_X1(1)} to $f_{[k-1]}$, we immediately get the following theorem.
\begin{theorem}\label{thm:O_Xk(1)}
Let $(X, V)$ be a directed projective manifold. Let $f \colon (\mathbb{C}, T_{\mathbb{C}}) \to (X, V)$ be a non-constant holomorphic curve. For each $k \in \mathbb{Z}_{+}$, we have
\[
T_{f_{[k]}}(r, \mathcal{O}_{X_k}(1)) + \mathcal{N}([Z_{f_{[k-1]}'}], r) \leq S_f(r)
.\]
\end{theorem}

\subsection{Ahlfors' lemma on logarithmic derivative over directed manifolds}

For any divisor $D \in \mathfrak{S}$, there exists a hyperplane $\widetilde{H}$ of $\mathfrak{S}^{\dual} \cong \mathbb{P}^N$ such that $D = \mu_{\mathfrak S}^{\ast} \widetilde{H}$.
Let $\sigma$ be the ``unit'' canonical section of $\mathcal{O}(D)$ with respect to the basis $\{\varsigma_i \mid i = 0, 1, \dots, N\}$ and let $\bm{\tilde{a}} \in (\mathbb{C}^{N+1})^{\dual}$ be a unit normal vector of $\widetilde{H}$.
By computation, we have
\[
\widetilde{F}^k \,\llcorner\, \bm{\tilde{a}} = \sum_{0 \leq i_1 < \dots < i_k \leq n} \hspace{-1eM} W(\sigma, \varsigma_{i_1}, \dots, \varsigma_{i_k})(j_kf) \bm{e}^{i_1} \wedge \dots \wedge \bm{e}^{i_k}
.\]
Analogous to \eqref{eqn:normtildeF^kbar}, we have
\begin{equation}\label{eqn:normtildeF^kabar}
|\widetilde{F}^k \,\llcorner\, \bm{\tilde{a}}|^2 = \exp(\widebar{\varphi}_{\widebar{\mathfrak W}(X_k, D, \mathfrak{S})} \circ f_{[k]}) |f'|^{2k} |f_{[1]}'|^{2(k-1)} \cdots |f_{[k-1]}'|^2
,\end{equation}
where $\widebar{\mathfrak W}(X_k, D, \mathfrak{S})$ is the linear system corresponding to $\widebar{\mathbb W}(X_k, D, \mathfrak{S})$ defined by \eqref{eqn:WXkDS}.
By \eqref{eqn:normtildeF^kbar} and \eqref{eqn:normtildeF^kabar}, the $k$-th contact function of $\tilde{f}$ for $\widetilde{H}$ is given by
\begin{equation}\label{eqn:phikH}
\widetilde{\phi}_k(\widetilde{H}) = \frac{|\widetilde{F}^k \,\llcorner\, \bm{\tilde{a}}|^2}{|\widetilde{F}^k|^2} = \exp\left((\varphi_{\widebar{\mathfrak W}(X_k, D, \mathfrak{S})} - \varphi_{\widebar{\mathfrak W}(X_k, \mathfrak{S})}) \circ f_{[k]}\right)
.\end{equation}
Notice that $\widebar{\mathfrak w}(X_k, D, \mathfrak{S}) \subset \widebar{\mathfrak w}(X_k, \mathfrak{S})$, for any $D \in \mathfrak{S}$.
Therefore,
\[
\upsilon_{k, \mathfrak{S}}^{-1} \widebar{\mathfrak w}(X_k, D, \mathfrak{S}) \cdot \mathcal{O}_{\widehat{X}_{k, \mathfrak{S}}}(E_{k, \mathfrak{S}}) \subset \mathcal{O}_{\widehat{X}_{k, \mathfrak{S}}}
,\]
which is the ideal sheaf of a certain closed subscheme of $\widehat{X}_{k, \mathfrak{S}}$, denoted by $D_{\mathfrak S}^{\langle k \rangle}$.
Let $D_{\mathfrak S}^{[k]}$ be the closed subscheme of $X_k$ with the ideal sheaf $\widebar{\mathfrak w}(X_k, D, \mathfrak{S})$.
By definition,
\begin{equation}\label{eqn:D+E}
D_{\mathfrak S}^{\langle k \rangle} + E_{k, \mathfrak{S}} = \upsilon_{k, \mathfrak{S}}^{\ast} D_{\mathfrak S}^{[k]}
.\end{equation}
Take an arbitrary continuous Hermitian metric on $\mathcal{O}_{X_k}(\bm{a}^k) \otimes \pi_{0, k}^{\ast} L^{k+1}$ with the weight function $\varphi$.
It follows from \eqref{eqn:lambda_Zs} that $\dfrac{1}{2} (\varphi - \varphi_{\widebar{\mathfrak W}(X_k, \mathfrak{S})})$ and $\dfrac{1}{2} (\varphi - \varphi_{\widebar{\mathfrak W}(X_k, D, \mathfrak{S})})$ are Weil functions for the closed subschemes $Z_{\widebar{\mathfrak w}(X_k, \mathfrak{S})}$ and $D_{\mathfrak S}^{[k]}$ of $X_k$, respectively.
Accordingly, $\dfrac{1}{2} (\varphi_{\widebar{\mathfrak W}(X_k, \mathfrak{S})} - \varphi_{\widebar{\mathfrak W}(X_k, D, \mathfrak{S})}) \circ \upsilon_{k, \mathfrak{S}}$ is a Weil function for the closed subscheme $D_{\mathfrak S}^{\langle k \rangle}$ of $\widehat{X}_{k, \mathfrak{S}}$.
Thanks to \eqref{eqn:phikH} , for any Weil function $\lambda_{D_{\mathfrak S}^{\langle k \rangle}}$ for the closed subscheme $D_{\mathfrak S}^{\langle k \rangle}$, there exists some constant $a \geq 1$ such that
\[
\frac{1}{a} \widetilde{\phi}_k(\widetilde{H}) \leq \!\exp\left(-2\lambda_{D_{\mathfrak S}^{\langle k \rangle}} \circ \hat{f}_{[k]}\right) \leq a \widetilde{\phi}_k(\widetilde{H})
.\]

Consequently, applying Lemma \ref{lem:ALLD} to holomorphic curves $\widetilde{f} = \mu_{\mathfrak S} \circ f$ in $\mathfrak{S}^{\dual} \cong \mathbb{P}^N$, we conclude Ahlfors’ lemma on logarithmic derivative over directed manifolds, which is formulated in the language of Weil functions for closed subschemes of Demailly-Semple jet towers.
\begin{lemma}\label{lem:ALLDdirected}
Let $(X, V)$ be a directed projective manifold, and take an effective divisor $D$ on $X$ such that $|D|$ is base point free.
Choose a linear system $\mathfrak{S} \subset |D|$ satisfying that $D \in \mathfrak{S}$ and $\mathfrak{S}$ is also base point free.
For each $0 \leq k \leq N-1$, let $f \colon (\mathbb{C}, T_{\mathbb{C}}) \to (X, V)$ be a non-constant holomorphic curve such that $f_{[k]}(\mathbb{C}) \not\subset \Bs(\mathfrak{W}(X_k, \mathfrak{S}))$.
Suppose that $f(\mathbb{C}) \not\subset \Supp D$.
Then there is a constant $c_k$ such that for any $0 < \varepsilon < 1$, one has
\[
\varepsilon \mathcal{N}\!\left(\exp\!\left(2(1-\varepsilon)\lambda_{D_{\mathfrak S}^{\langle k \rangle}} \!\circ\! \hat{f}_{[k]} - 2\lambda_{D_{\mathfrak S}^{\langle k+1 \rangle}} \!\circ\! \hat{f}_{[k+1]}\right) \hat{f}_{[k]}^{\ast} \ddif^c \varphi_{\widehat{\mathfrak W}(X_k, \mathfrak{S})}, r\right) \leq c_k T_{\hat{f}_{[k]}}(r, L_k) + O(1)
,\]
where $N = \dim \mathfrak{S}$, and $L_k$ is defined by \eqref{eqn:Lk} with $L = \mathcal{O}(D)$.
Here $c_k$ merely depends on the choice of Weil functions.
\end{lemma}

\begin{remark}\label{rmk:separate1}
In particular, if $\mathfrak{S}$ separates $1$-jets at every point of $X$, then both $\varpi_{1, \mathfrak{S}}$ and $\mu_{1, \mathfrak{S}}$ are holomorphic.
Thus $\upsilon_{1, \mathfrak{S}}$ is indeed the identity map and $E_{1, \mathfrak{S}} = \emptyset$.
According to \eqref{eqn:D+E}, $D_{\mathfrak S}^{\langle 1 \rangle} = \upsilon_{1, \mathfrak{S}}^{\ast} D_{\mathfrak S}^{[1]}$.
Furthermore, the discussion in the next section will show that $D_{\mathfrak S}^{[1]}$ is independent of the choice of $\mathfrak{S}$, provided that $\mathfrak{S}$ separates $1$-jets at every point of $X$.
\end{remark}

\subsection{Jets of closed subschemes over Demailly-Semple jet towers}
Let $(X, V)$ be a directed manifold.
For a closed subscheme $Z$ of $X$, we define the closed subscheme $Z^{(1)}$ of $X_1$ as follows.
In the case when $X$ is affine, choose some functions $\zeta_1, \dots, \zeta_s$ generating $\mathcal{I}_Z(X) \subset \mathcal{O}_X(X)$, where $\mathcal{I}_Z$ is the ideal sheaf of $Z$ of $X$.
Then we define the closed subscheme $Z^{(1)}$ of $X_1 = \textnormal{P}(V)$ determined by the homogeneous ideal
\[
I(Z^{(1)}) = (\zeta_1, \dots, \zeta_s, \dif \zeta_1, \dots, \dif \zeta_s)
,\]
where we abusively write $\zeta_i \bydef \dif_{U, V}^{[0]} \zeta_i = p_{1, V}^{\ast} \zeta_i$ and $\dif \zeta_i \bydef \dif_{U, V}^{[1]} \zeta_i$ for each $1 \leq i \leq s$.
We see from the Leibniz rule that $Z^{(1)}$ is well-defined, regardless of the choice of the generators of $\mathcal{I}_Z(X)$.
More generally, even when $X$ is not affine, we can take an affine open covering $\{U_{\epsilon}\}_{\epsilon \in \Upsilon}$ of $X$ and construct the closed subschemes $(Z_{\restriction U_{\epsilon}})^{(1)}$ of $\pi_1^{-1}(U_{\epsilon}) = \textnormal{P}(V_{\restriction U_{\epsilon}})$.
By definition, these closed subschemes can be glued together into the closed subscheme $Z^{(1)}$ of $X_1=\textnormal{P}(V)$.

For each $k \in \mathbb{Z}_{+}$, we inductively define
\[
Z^{(k)} \bydef (Z^{(k-1)})^{(1)}
\]
as a closed subscheme of $X_k=\textnormal{P}(V_{k-1})$.

If $Z$ is reduced and non-singular, we easily check that $Z^{(1)}$ is identified with the Demailly-Semple $1$-jet tower of $(Z, T_Z \cap V)$.
In other words,
\[
Z^{(1)} = \textnormal{P}(T_Z \cap V)
.\]

\begin{remark}
Let $Z = Z_1 + Z_2$ with $Z_1 \cap Z_2$ is not empty. One can see that $Z_1^{(1)} + Z_2^{(1)} \subsetneqq Z^{(1)}$ by basic computations.
In particular, through straightforward calculations, one also has
\begin{equation}\label{eqn:lZ^(1)}
(\ell Z)^{(1)} = (\ell-1) \pi_1^{\ast} Z + Z^{(1)}
,\end{equation}
for each $\ell \in \mathbb{Z}_{+}$.
\end{remark}

What's more, one obtains the following proposition by definition.
\begin{proposition}\label{prp:cap^(1)=^(1)cap}
Let $(X, V)$ be a directed manifold, and let $Z_1, \dots, Z_q$ be closed subschemes of $X$. For each $k \in \mathbb{Z}_{+}$, we have
\[
Z^{(k)} = \bigcap_{i=1}^q Z_i^{(k)}
,\]
where $Z = \bigcap_{i=1}^q Z_i$.
\end{proposition}

\begin{proof}
Take an affine open subset $U \subset X$.
Suppose that $\zeta_{i, 1}, \dots, \zeta_{i, s_i}$ is a system of generators of $\mathcal{I}_{Z_i}(U) \subset \mathcal{O}_X(U)$, for each $1 \leq i \leq q$.
By Definition \ref{def:I_Z}, \ref{itm:Icap}),
\[
\mathcal{I}_Z(U) = (\zeta_{1, 1}, \dots, \zeta_{1, s_1}, \dots, \zeta_{q, 1}, \dots, \zeta_{q, s_q})
,\]
and by the definition of $Z^{(1)}$,
\[
I((Z_{\restriction U})^{(1)}) = (\zeta_{1, 1}, \dots, \zeta_{q, s_q}, \dif \zeta_{1, 1}, \dots, \dif \zeta_{q, s_q})
.\]
On the other hand,
\[
I((Z_{i \restriction U})^{(1)}) = (\zeta_{i, 1}, \dots, \zeta_{i, s_i}, \dif \zeta_{i, 1}, \dots, \dif \zeta_{i, s_i})
,\]
for each $1 \leq i \leq q$.
Again, by Definition \ref{def:I_Z}, \ref{itm:Icap}),
\[
I(\bigcap_{i=1}^q (Z_{i \restriction U})^{(1)}) = (\zeta_{1, 1}, \dots, \zeta_{q, s_q}, \dif \zeta_{1, 1}, \dots, \dif \zeta_{q, s_q})
.\]
Hence $(Z_{\restriction U})^{(1)} = \bigcap_{i=1}^q (Z_{i \restriction U})^{(1)}$.
Notice that all such open subsets $U$ give an affine open covering of $X$. Thus
\[
Z^{(1)} = \bigcap_{i=1}^q Z_i^{(1)}
.\]
Then the proposition is verified by induction.
\end{proof}

Let $Z$ be a closed subscheme of $X$.
Choose a linear system $\mathfrak{S} \subset |L|$ on $X$ satisfying that there exist some divisors $D_1, \dots, D_s \in \mathfrak{S}$ such that $Z = \bigcap_{i=1}^s D_i$.
It amounts to saying that $Z$ is the closed subscheme of $X$ with the ideal sheaf $\mathcal{I}_Z = \mathfrak{s}_{\supset Z}$, which is the base ideal of $\mathfrak{S}_{\supset Z}$.
Then we denote by $Z_{\mathfrak{S}}^{[k]}$ the closed subscheme of $X_k$ with the ideal sheaf $\widebar{\mathfrak w}(X_k, Z, \mathfrak{S})$.
Analogous to Proposition \ref{prp:cap^(1)=^(1)cap}, one also has
\begin{proposition}\label{prp:cap_S^k=_S^kcap}
Let $(X, V)$ be a directed manifold, and let $Z_1, \dots, Z_q$ be closed subschemes of $X$.
Choose a linear system $\mathfrak{S} \subset |L|$ on $X$ such that $\mathcal{I}_{Z_i} = \mathfrak{s}_{\supset Z_i}$ for each $1 \leq i \leq q$.
For each $k \in \mathbb{Z}_{+}$, we have
\[
Z_{\mathfrak{S}}^{[k]} = \bigcap_{i=1}^q (Z_i)_{\mathfrak{S}}^{[k]}
,\]
where $Z = \bigcap_{i=1}^q Z_i$.
\end{proposition}

In addition, the following proposition indicates a relation between $Z^{(1)}$ and $Z_{\mathfrak S}^{[1]}$.
\begin{proposition}\label{prp:Z^(1)<Z_S^1}
Let $(X, V)$ be a directed manifold, and let $Z$ be a closed subscheme of $X$.
Choose a linear system $\mathfrak{S} \subset |L|$ on $X$ such that $\mathcal{I}_Z = \mathfrak{s}_{\supset Z}$. Then
\[
\mathcal{I}_{Z^{(1)}} \cdot \mathfrak{w}(X_1, \mathfrak{S}) \subset \mathfrak{w}(X_1, Z, \mathfrak{S}) \subset \mathcal{I}_{Z^{(1)}}
.\]
That is to say,
\[
Z^{(1)} \subset Z_{\mathfrak{S}}^{[1]} \subset Z^{(1)} + Z_{\mathfrak{w}(X_1, \mathfrak{S})}
.\]
\end{proposition}

\begin{proof}
Choose an affine open subset $U \subset X$, on which $L_{\restriction U}$ can be trivialized.
It suffices to show that
\[
(Z_{\restriction U})^{(1)} \subset (Z_{\mathfrak{S}}^{[1]})_{\restriction \pi_1^{\!-\!1\!}(U)} \subset (Z_{\restriction U})^{(1)} + (Z_{\mathfrak{w}(X_1, \mathfrak{S})})_{\restriction \pi_1^{\!-\!1\!}(U)}
.\]
Take a trivialization of $L_{\restriction U}$, which associates $\sigma_U \in \mathcal{O}(U)$ to a global section $\sigma \in H^0(X, L)$.
Choose a convenient basis $\{\varsigma_i \mid i = 0, 1, \dots, N\}$ of $\mathbb{S}$ such that $\{\varsigma_i \mid i = 0, 1, \dots, N'\}$ is a basis of $\mathbb{S}_{\supset Z}$,
where $N = \dim \mathfrak{S}$ and $N' = \dim \mathfrak{S}_{\supset Z}$.
Therefore, the closed subscheme $(Z_{\mathfrak{S}}^{[1]})_{\restriction \pi_1^{\!-\!1\!}(U)}$ of $\pi_1^{-1}(U) = \textnormal{P}(V_{\restriction U})$ is determined by the homogeneous ideal
\[
I((Z_{\mathfrak{S}}^{[1]})_{\restriction \pi_1^{\!-\!1\!}(U)}) = (W_U(\varsigma_{i'}, \varsigma_i))_{0 \leq i' \leq N', \, 0 \leq i' < i \leq N}
.\]
Analogously, the closed subscheme $(Z_{\mathfrak{w}(X_1, \mathfrak{S})})_{\restriction \pi_1^{\!-\!1\!}(U)}$ of $\pi_1^{-1}(U) = \textnormal{P}(V_{\restriction U})$ is determined by the homogeneous ideal
\[
I((Z_{\mathfrak{w}(X_1, \mathfrak{S})})_{\restriction \pi_1^{\!-\!1\!}(U)}) = (W_U(\varsigma_i, \varsigma_j))_{0 \leq i < j \leq N}
.\]
Seeing that $\varsigma_{0, U}, \dots, \varsigma_{N, U}$ generates $\mathfrak{s}_{\supset Z}(U) \subset \mathcal{O}_X(U)$,
the closed subscheme $(Z_{\restriction U})^{(1)}$ of $\pi_1^{-1}(U) = \textnormal{P}(V_{\restriction U})$ is determined by the homogeneous ideal
\[
I((Z_{\restriction U})^{(1)}) = (\varsigma_{0, U}, \dots, \varsigma_{N', U}, \dif \varsigma_{0, U}, \dots, \dif \varsigma_{N', U})
.\]
Apparently, for any $0 \leq i' \leq N'$ and $0 \leq i' < i \leq N$,
\[
W_U(\varsigma_{i'}, \varsigma_i) \bydef \varsigma_{i', U} \dif \varsigma_{i, U} - \varsigma_{i, U} \dif \varsigma_{i', U} \in I((Z_{\restriction U})^{(1)})
,\]
which implies that
\[
I((Z_{\mathfrak{S}}^{[1]})_{\restriction \pi_1^{\!-\!1\!}(U)}) \subset I((Z_{\restriction U})^{(1)})
.\]
In other words,
\[
(Z_{\restriction U})^{(1)} \subset (Z_{\mathfrak{S}}^{[1]})_{\restriction \pi_1^{\!-\!1\!}(U)}
.\]
On the other hand, we assert that
\[
(Z_{\mathfrak{S}}^{[1]})_{\restriction \pi_1^{\!-\!1\!}(U)} \subset (Z_{\restriction U})^{(1)} + (Z_{\mathfrak{w}(X_1, \mathfrak{S})})_{\restriction \pi_1^{\!-\!1\!}(U)}
,\]
which amounts to saying that
\[
I((Z_{\restriction U})^{(1)}) \cdot I((Z_{\mathfrak{w}(X_1, \mathfrak{S})})_{\restriction \pi_1^{\!-\!1\!}(U)}) \subset I((Z_{\mathfrak{S}}^{[1]})_{\restriction \pi_1^{\!-\!1\!}(U)})
.\]
It suffices to verify that
\[
\varsigma_{i', U} W_U(\varsigma_i, \varsigma_j), \dif \varsigma_{i', U} W_U(\varsigma_i, \varsigma_j) \in I((Z_{\mathfrak{S}}^{[1]})_{\restriction \pi_1^{\!-\!1\!}(U)})
,\]
for any $0 \leq i' \leq N'$ and $0 \leq i < j \leq N$.
It is obvious if $i \leq i'$.
If $0 \leq i' < i < j \leq N$, then
\begin{align*}
\varsigma_{i', U} W_U(\varsigma_i, \varsigma_j) &= \varsigma_{i, U} W_U(\varsigma_{i'}, \varsigma_j) - \varsigma_{j, U} W_U(\varsigma_{i'}, \varsigma_i) \in I((Z_{\mathfrak{S}}^{[1]})_{\restriction \pi_1^{\!-\!1\!}(U)}), \\
\dif \varsigma_{i', U} W_U(\varsigma_i, \varsigma_j) &= \dif \varsigma_{i, U} W_U(\varsigma_{i'}, \varsigma_j) - \dif \varsigma_{j, U} W_U(\varsigma_{i'}, \varsigma_i) \in I((Z_{\mathfrak{S}}^{[1]})_{\restriction \pi_1^{\!-\!1\!}(U)})
.\qedhere\end{align*}
\end{proof}

In the particular case where $\mathfrak{S}$ separates $1$-jets at every point of $X$, it follows from Lemma \ref{lem:wseparate} that $\mathfrak{w}(X_1, \mathfrak{S}) = \mathcal{O}_X$.
Accordingly, we immediately get the following corollary.
\begin{corollary}\label{cor:wZseparate1}
Let $(X, V)$ and $Z$ be as above.
Choose a linear system $\mathfrak{S}$ on $X$ such that $\mathfrak{S}$ separates $1$-jets at every point of $X$ and $\mathcal{I}_Z = \mathfrak{s}_{\supset Z}$. Then
\[
\mathfrak{w}(X_1, Z, \mathfrak{S}) = \mathcal{I}_{Z^{(1)}}
,\]
namely,
\[
Z_{\mathfrak{S}}^{[1]} = Z^{(1)}
.\]
\end{corollary}

Next, we are going to show that one can always find a linear system $\mathfrak{S}$ satisfying the condition of Corollary \ref{cor:wZseparate1}.
\begin{lemma}\cite{Silverman1987}\label{lem:Zcap}
Let $Z$ be a closed subscheme of $X$, then there exist effective Cartier divisors $D_1, \dots, D_s$ such that
\[
Z = \bigcap_{i=1}^s D_i
.\]
\end{lemma}

\begin{lemma}\label{lem:Sexist}
Let $(X, V)$ be a directed projective manifold. Given an arbitrary closed subscheme $Z$ of $X$, we can find a linear system $\mathfrak{S}$ on $X$ such that $\mathfrak{S}$ separates $1$-jets at every point of $X$ and $\mathcal{I}_Z = \mathfrak{s}_{\supset Z}$.
\end{lemma}

\begin{proof}
According to Lemma \ref{lem:Zcap}, there exist effective divisors $D_1, \dots, D_s$ such that $Z = \bigcap_{i=1}^s D_i$.
Take a very ample line bundle $A$.
For each integer $1 \leq i \leq s$, there is a positive integer $m_i$ such that $A^m \otimes \mathcal{O}_X(-D_i)$ is global generated, for every integer $m \geq m_i$.
Let $L=A^{m_0}$ and $\mathfrak{S} = |L|$, where $\displaystyle m_0 = \mathop{\max}_{1 \leq i \leq s} m_i$.
Obviously, $\mathfrak{S}$ separates $1$-jets at every point of $X$, and $\mathcal{I}_{D_i} = \mathfrak{s}_{\supset D_i}$ for each $1 \leq i \leq s$.
By Definition \ref{def:I_Z}, \ref{itm:Icap}), we see that
\[
\mathcal{I}_Z = \mathfrak{s}_{\supset D_1} + \dots + \mathfrak{s}_{\supset D_s} \subset \mathfrak{s}_{\supset Z}
.\]
On the other hand, it is evident that $\mathfrak{s}_{\supset Z} \subset \mathcal{I}_Z$.
Consequently, $\mathcal{I}_Z = \mathfrak{s}_{\supset Z}$.
\end{proof}

\subsection{General form of Ahlfors' LLD over directed manifolds}

Let us apply Lemma \ref{lem:ALLDdirected} with $k = 0$.
Given an effective divisor $D$ on $X$, observe that $\mathcal{I}_D = \mathfrak{s}_{\supset D}$, provided that $D \in \mathfrak{S}$.
Remark \ref{rmk:separate1} and Corollary \ref{cor:wZseparate1} lead to a corollary of Lemma \ref{lem:ALLDdirected} as follows.
\begin{corollary}\label{cor:ALLDD^(1)}
Let $(X, V)$ be a directed projective manifold, and take a very ample divisor $D$ on $X$.
Choose a linear system $\mathfrak{S} \subset |D|$ satisfying that $D \in \mathfrak{S}$ and $\mathfrak{S}$ separates $1$-jets at every point of $X$.
Let $f \colon (\mathbb{C}, T_{\mathbb{C}}) \to (X, V)$ be a non-constant holomorphic curve such that $f(\mathbb{C}) \not\subset \Supp D$.
Then there is a constant $c$ such that the inequality
\[
\varepsilon \mathcal{N}\left(\exp\left(2(1-\varepsilon)\lambda_{D} \circ f - 2\lambda_{D^{(1)}} \circ f_{[1]}\right) f^{\ast} \ddif^c \varphi_{\mathfrak S}, r\right) \leq c T_{f}(r, L) + O(1)
\]
holds for any $0 < \varepsilon < 1$, where $L = \mathcal{O}(D)$.
\end{corollary}

Moreover, we can generalize Corollary \ref{cor:ALLDD^(1)} to closed subschemes of $X$.

\begin{lemma}\label{lem:ALLDZ^(1)}
Let $(X, V)$ be a directed projective manifold, and let $Z$ be a closed subscheme of $X$.
Let $f \colon (\mathbb{C}, T_{\mathbb{C}}) \to (X, V)$ be a non-constant holomorphic curve such that $f(\mathbb{C}) \not\subset \Supp Z$.
Choose a linear system $\mathfrak{S} \subset |L|$ on $X$ satisfying that $\mathfrak{S}$ separates $1$-jets at every point of $X$ and $\mathcal{I}_Z = \mathfrak{s}_{\supset Z}$.
Then there exists a constant $c > 0$ such that the inequality
\begin{equation}\label{eqn:ALLDZ^(1)}
\varepsilon \mathcal{N}\left(\exp\left(2(1-\varepsilon)\lambda_Z \circ f - 2\lambda_{Z^{(1)}} \circ f_{[1]}\right) f^{\ast} \ddif^c \varphi_{\mathfrak S}, r\right) \leq c T_f(r, L) + O(1)
\end{equation}
holds for any $0 < \varepsilon < 1$.
\end{lemma}

\begin{proof}
Since $f(\mathbb{C}) \not\subset \Supp Z$, we can find a basis $\{\varsigma_i \mid i = 0, 1, \dots, N'\}$ of $\mathbb{S}_{\supset Z}$ such that $f(\mathbb{C}) \not\subset \Supp D_i$ for each $0 \leq i \leq N'$, where $D_i = (\varsigma_i)$.
According to Corollary \ref{cor:ALLDD^(1)}, for any $0 < \varepsilon < 1$, we have
\[
\varepsilon \mathcal{N}\left(\exp\left(2(1-\varepsilon)\lambda_{D_i} \circ f - 2\lambda_{D_i^{(1)}} \circ f_{[1]}\right) f^{\ast} \ddif^c \varphi_{\mathfrak S}, r\right) \leq c' T_{f}(r, L) + O(1)
.\]
We see from the assumption $\mathcal{I}_Z = \mathfrak{s}_{\supset Z}$ that
\[
Z = \bigcap_{0 \leq i \leq N'} D_i
.\]
Furthermore, by Proposition \ref{prp:cap^(1)=^(1)cap}, we also obtain
\[
Z^{(1)} = \bigcap_{0 \leq i \leq N'} D_i^{(1)}
.\]
It follows from Proposition \ref{prp:lambda_Z}, \ref{itm:lambdacap}) that
\begin{align*}
(1-\varepsilon) \lambda_Z \circ \pi_1 - \lambda_{Z^{(1)}} &=_{\langle X_1 \rangle} (1-\varepsilon) \mathop{\min}_{0 \leq i \leq N'} \lambda_{D_i} \circ \pi_1 - \mathop{\min}_{0 \leq i \leq N'} \lambda_{D_i^{(1)}} \\
&\leq_{\langle X_1 \rangle} \mathop{\max}_{0 \leq i \leq N'} \left( (1-\varepsilon) \lambda_{D_i} \circ \pi_1 - \lambda_{D_i^{(1)}} \right)
,\end{align*}
which amounts to saying that
\[
(1-\varepsilon) \lambda_Z \circ f - \lambda_{Z^{(1)}} \circ f_{[1]} \leq \mathop{\max}_{0 \leq i \leq N'} \left( (1-\varepsilon) \lambda_{D_i} \circ f - \lambda_{D_i^{(1)}} \circ f_{[1]} \right)  + a
,\]
for some constant $a \geq 0$.

Hence we get
\begin{align*}
&\varepsilon \mathcal{N}\left(\exp\left(2(1-\varepsilon)\lambda_Z \circ f - 2\lambda_{Z^{(1)}} \circ f_{[1]}\right) f^{\ast} \ddif^c \varphi_{\mathfrak S}, r\right) \\
\leq \,& \varepsilon \mathcal{N}\left(\exp\left( \mathop{\max}_{0 \leq i \leq N'}\left( 2(1-\varepsilon)\lambda_{D_i} \circ f - 2\lambda_{D_i^{(1)}} \circ f_{[1]} \right) +2a \right) f^{\ast} \ddif^c \varphi_{\mathfrak S}, r\right) \\
\leq \,& \varepsilon \mathcal{N}\left( e^{2a} \mathop{\max}_{0 \leq i \leq N'} \exp\left( 2(1-\varepsilon) \lambda_{D_i} \circ f - 2 \lambda_{D_i^{(1)}} \circ f_{[1]} \right) f^{\ast} \ddif^c \varphi_{\mathfrak S}, r\right) \\
\leq \,& \varepsilon \mathcal{N}\left( e^{2a} \sum_{i=0}^{N'} \exp\left( 2(1-\varepsilon) \lambda_{D_i} \circ f - 2 \lambda_{D_i^{(1)}} \circ f_{[1]} \right) f^{\ast} \ddif^c \varphi_{\mathfrak S}, r\right) \\
\leq \,& \varepsilon e^{2a} \sum_{i=0}^{N'} \mathcal{N}\left( \exp\left( 2(1-\varepsilon) \lambda_{D_i} \circ f - 2 \lambda_{D_i^{(1)}} \circ f_{[1]} \right) f^{\ast} \ddif^c \varphi_{\mathfrak S}, r\right) \\
\leq \,& c' (N'+1) e^{2a} T_{f}(r, L) + O(1)
.\end{align*}
The lemma follows by taking $c = c' (N'+1) e^{2a}$.
\end{proof}

What's more, for arbitrary closed subschemes $Z_1, \dots, Z_q$ of $X$, we see that
\[
(1-\varepsilon) \mathop{\max}_{1 \leq i \leq q} \lambda_{Z_i} \circ \pi_1 - \mathop{\max}_{1 \leq i \leq q} \lambda_{Z_i^{(1)}} \leq \mathop{\max}_{1 \leq i \leq q} \left( (1-\varepsilon) \lambda_{Z_i} \circ \pi_1 - \lambda_{Z_i^{(1)}} \right)
.\]
Following the same reasoning as in the proof of Lemma \ref{lem:ALLDZ^(1)}, we conclude the following lemma from the above inequality.
\begin{lemma}\label{lem:GALLDD}
Let $(X, V)$ be a directed projective manifold, and let $Z_1, \dots, Z_q$ be closed subschemes of $X$. Let $f \colon (\mathbb{C}, T_{\mathbb{C}}) \to (X, V)$ be a non-constant holomorphic curve such that $f(\mathbb{C}) \not\subset \Supp (Z_1 + \dots + Z_q)$. For any linear system $\mathfrak{S} \subset |L|$ on $X$ satisfying that $\mathfrak{S}$ separates $1$-jets at every point of $X$ and $\mathcal{I}_Z = \mathfrak{s}_{\supset Z}$, there exists a constant $c > 0$ such that the inequality
\[
\varepsilon \mathcal{N}\left(\exp\left(2(1-\varepsilon)\mathop{\max}_{1 \leq i \leq q} \lambda_{Z_i} \circ f - 2\mathop{\max}_{1 \leq i \leq q} \lambda_{Z_i^{(1)}} \circ f_{[1]}\right) f^{\ast} \ddif^c \varphi_{\mathfrak S}, r\right) \leq c T_f(r, L) + O(1)
\]
holds for any $0 < \varepsilon < 1$.
\end{lemma}
We refer to this lemma as General form of Ahlfors' Lemma on Logarithmic Derivative over Directed Manifolds (GALLDD for short).

\section[Algebro-geometric version of Ahlfors' LLD]{Algebro-geometric version of Ahlfors' Lemma on Logarithmic Derivative}\label{sec:AALD}
\sectionmark{Algebro-geometric version of Ahlfors' LLD}

\subsection{Proof of AALD}

First, we are going to prove AALD for a closed subscheme $Z$ of $X$ for $1$-jets.
\begin{theorem}\label{thm:Z^(1)}
Let $(X, V)$ be a directed projective manifold, and let $Z$ be a closed subscheme of $X$.
Let $f \colon (\mathbb{C}, T_{\mathbb{C}}) \to (X, V)$ be a non-constant holomorphic curve such that $f(\mathbb{C}) \not\subset \Supp Z$.
Then we have
\[
m_f(r, Z) + T_{f_{[1]}}(r, \mathcal{O}_{X_1}(1)) + \mathcal{N}([Z_{f'}], r) \leq m_{f_{[1]}}(r, Z^{(1)}) + S_f(r)
.\]
\end{theorem}

\begin{proof}
Let $\gamma \ddif^c |z|^2$ denote the integrand of left hand side of \eqref{eqn:ALLDZ^(1)}, namely,
\[
\gamma = \exp\left(2((1-\varepsilon)\lambda_Z) \circ f - 2\lambda_{Z^{(1)}} \circ f_{[1]}\right) \gamma_{\mathfrak S}
,\]
where $\mathfrak{S} \subset |L|$ is a linear system on $X$ such that $\mathfrak{S}$ separates $1$-jets at every point of $X$, and $\mathcal{I}_Z = \mathfrak{s}_{\supset Z}$.

Without loss of generality, suppose that $\gamma_{\mathfrak S}(0) \neq 0$, otherwise we can deal with it in the similar way to Remark \ref{rmk:zero}.
Using the Green-Jensen formula and \eqref{eqn:Nddclogg} , we have
\begin{align*}
&\frac{1}{2} \int_0^{2\pi} \log \gamma(re^{i\theta}) \frac{\dif \theta}{2\pi} \\
=& \int_0^{2\pi} \!(1-\varepsilon)\lambda_Z \circ f(re^{i\theta}) \frac{\dif \theta}{2\pi} - \int_0^{2\pi} \!\lambda_{Z^{(1)}} \circ f_{[1]}(re^{i\theta}) \frac{\dif \theta}{2\pi} + \frac{1}{2} \int_0^{2\pi} \!\log \gamma_{\mathfrak S}(re^{i\theta}) \frac{\dif \theta}{2\pi} \\
=& (1-\varepsilon) m_f(r, Z) - m_{f_{[1]}}(r, Z^{(1)}) + \mathcal{N}(\ddif^c [\log \gamma_{\mathfrak S}], r) + \frac{1}{2} \log \gamma_{\mathfrak S}(0) \\
=& (1-\varepsilon) m_f(r, Z) - m_{f_{[1]}}(r, Z^{(1)}) + T_{f_{[1]}}(r, \mathcal{O}_{X_1}(1)) + \mathcal{N}([Z_{f'}], r) + \frac{1}{2} \log \gamma_{\mathfrak S}(0)
.\end{align*}
On the other hand, it follows from the Jensen inequality, Lemma \ref{lem:calculus} and Lemma \ref{lem:ALLDZ^(1)} that
\begin{align*}
\int_0^{2\pi} \log \gamma(re^{i\theta}) \frac{\dif \theta}{2\pi} &\leq \log \int_0^{2\pi} \gamma(re^{i\theta}) \frac{\dif \theta}{2\pi} \\
&\leq \left. (1+\delta)^2 \log \mathcal{N}([\gamma \ddif^c |z|^2], r) \right\rVert_E \\
&\leq \left. (1+\delta)^2 \left(\log T_f(r, L) - \log \varepsilon + O(1)\right) \right\rVert_E
.\end{align*}
Especially, if $f$ is an algebraic curve, we claim that $\gamma(re^{i\theta}) \leq O(1)$ as $r \to \infty$.
Seeing that $\gamma_{\mathfrak S}(re^{i\theta}) \leq O(1) \, (r \to \infty)$ in the case of algebraic curves, it suffices to verify $\exp\left(2((1-\varepsilon)\lambda_Z) \circ f(re^{i\theta}) - 2\lambda_{Z^{(1)}} \circ f_{[1]}(re^{i\theta})\right) \leq O(1) \, (r \to \infty)$.
Notice that the preimage of $Z$ under an algebraic curve $f$ is a finite set of points.
One easily find a constant $r_0 > 0$ such that the poles of $\exp\left(\lambda_Z \circ f\right)$ are all contained in the disk $\bigtriangleup(r_0)$.
Our claim is evidently true, if $f$ does not approach $Z$ as $z \to \infty$.
By computation, one can check that $\exp\left(\lambda_{Z^{(1)}} \circ f\right)$ also takes poles at $\infty$ with at least the same order as $\exp\left(\lambda_Z \circ f\right)$, if $f$ approaches $Z$ as $z \to \infty$.
This implies that $\exp\left(2((1-\varepsilon)\lambda_Z) \circ f(re^{i\theta}) - 2\lambda_{Z^{(1)}} \circ f_{[1]}(re^{i\theta})\right) \leq O(1) \, (r \to \infty)$.
Therefore, our claim holds, and we have
\[
\int_0^{2\pi} \log \gamma(re^{i\theta}) \frac{\dif \theta}{2\pi} \leq O(1) \quad (r \to \infty)
.\]
Combining the above inequalities, we obtain
\[
m_f(r, Z) + T_{f_{[1]}}(r, \mathcal{O}_{X_1}(1)) + \mathcal{N}([Z_{f'}], r) \leq \varepsilon m_f(r, Z) + m_{f_{[1]}}(r, Z^{(1)}) + O\left(\log \frac{1}{\varepsilon}\right) + S_f(r)
.\]
Since $\mathcal{I}_Z = \mathfrak{s}_{\supset Z}$ and $f(\mathbb{C}) \not\subset \Supp Z$, there always exists a divisor $D \in \mathfrak{S}_{\supset Z} \subset |L|$ such that $Z \subset D$ and $f(\mathbb{C}) \not\subset \Supp D$.
Accordingly,
\[
m_f(r, Z) \leq m_f(r, D) + O(1) \leq T_f(r, L) + O(1)
.\]
Take $\varepsilon(r) = \dfrac{1}{T_f(r, L)}$.
We thus see that
\[
\varepsilon m_f(r, Z) + O\left(\log \frac{1}{\varepsilon}\right) \leq S_f(r)
,\]
which verifies the theorem.
\end{proof}

What's more, for each $k \in \mathbb{Z}_{+}$, think of $(X_k, V_k)$ as the Demailly-Semple $1$-jet tower of $(X_{k-1}, V_{k-1})$.
Applying Theorem \ref{thm:Z^(1)} to $f_{[k-1]}$, we conclude Algebro-Geometric Version of Ahlfors’ Lemma on Logarithmic Derivative (AALD).
\begin{theorem}[AALD]\label{thm:Z^(k)}
Let $(X, V)$ be a directed projective manifold, and let $Z$ be a closed subscheme of $X$.
Let $f \colon (\mathbb{C}, T_{\mathbb{C}}) \to (X, V)$ be a non-constant holomorphic curve such that $f(\mathbb{C}) \not\subset \Supp Z$.
For each $k \in \mathbb{Z}_{+}$, we have
\[
m_{f_{[k-1]}}(r, Z^{(k-1)}) + T_{f_{[k]}}(r, \mathcal{O}_{X_{k}}(1)) + \mathcal{N}([Z_{f_{[k-1]}'}], r) \leq m_{f_{[k]}}(r, Z^{(k)}) + S_f(r)
.\]
\end{theorem}

\begin{proof}
Observe that $f_{[k-1]}(\mathbb{C}) \not\subset \Supp \pi_{0, k-1}^{\ast} Z$.
Hence $f_{[k-1]}(\mathbb{C}) \not\subset \Supp Z^{(k-1)}$, since $\Supp Z^{(k-1)} \subset \Supp \pi_{0, k-1}^{\ast} Z$.
Applying Theorem \ref{thm:Z^(1)} to $f_{[k-1]}$, we immediately get
\[
m_{f_{[k-1]}}(r, Z^{(k-1)}) + T_{f_{[k]}}(r, \mathcal{O}_{X_{k}}(1)) + \mathcal{N}([Z_{f_{[k-1]}'}], r) \leq m_{f_{[k]}}(r, Z^{(k)}) + S_{f_{[k-1]}}(r)
.\]
Consequently, the theorem follows from \eqref{eqn:Sfk}.
\end{proof}

\begin{remark}
In the particular case when $Z = \emptyset$, Theorem \ref{thm:Z^(k)} reduces to Theorem \ref{thm:O_Xk(1)}.
\end{remark}

\subsection{AALD with respect to a linear system}

In this subsection, we will investigate what happens if the linear system might not separate the $1$-jets at every point of $X$.

\begin{theorem}\label{thm:Z_S^1}
Let $(X, V)$ be a directed projective manifold, and let $Z$ be a closed subscheme of $X$.
Choose a linear system $\mathfrak{S} \subset |L|$ on $X$ satisfying that $\mathcal{I}_Z = \mathfrak{s}_{\supset Z}$.
Let $f \colon (\mathbb{C}, T_{\mathbb{C}}) \to (X, V)$ be a non-constant holomorphic curve such that $f_{[1]}(\mathbb{C}) \not\subset \Bs(\mathfrak{W}(X_1, \mathfrak{S}))$.
Suppose that $f(\mathbb{C}) \not\subset \Supp Z$.
Then,
\[
m_f(r, Z) + T_{f_{[1]}}(r, \mathcal{O}_{X_1}(1)) + \mathcal{N}([Z_{f'}], r) \leq m_{f_{[1]}}(r, Z_{\mathfrak S}^{[1]}) - m_f(r, Z_{\mathfrak s}) + S_f(r)
,\]
where $Z_{\mathfrak{S}}^{[k]}$ is the closed subscheme of $X_k$ with the ideal sheaf $\widebar{\mathfrak w}(X_k, Z, \mathfrak{S})$.
\end{theorem}

\begin{proof}
At first, if $\mathfrak{S}$ is base point free, i.e., $\Bs(\mathfrak{S}) = \emptyset$, the theorem follows from Proposition \ref{prp:Z^(1)<Z_S^1} and Theorem \ref{thm:Z^(1)}.

Next, consider the case when $\Bs(\mathfrak{S}) \neq \emptyset$ and $Z_{\mathfrak s} = D_0$ for a certain effective Cartier divisor $D_0$.
Notice that $\mathfrak{s}_{\supset Z} \subset \mathfrak{s}$, namely, $D_0 \subset Z$.
Hence there is a closed subscheme $Z_1$ of $X$ such that $Z = D_0 + Z_1$.
Let $\sigma_0$ be a canonical section of $\mathcal{O}_X(D_0)$.
By computation, we obtain $W(\sigma_0 \sigma_1, \sigma_0 \sigma_2) = \sigma_0^2 W(\sigma_1, \sigma_2)$ for any $\sigma_1 \in (\mathbb{S}_{\setminus D_0})_{\supset Z_1}$ and $\sigma_2 \in \mathbb{S}_{\setminus D_0}$.
For any effective divisor $D$ and any linear system $\mathfrak{S}$ on $X$, let us use the notation $D \cdot \mathfrak{S} \bydef \{ D + D' \mid D' \in \mathfrak{S}\}$.
Thus
\[
\mathfrak{W}(X_1, Z, \mathfrak{S}) = 2 \pi_1^{\ast} D_0 \cdot \mathfrak{W}(X_1, Z_1, \mathfrak{S}_{\setminus D_0})
.\]
It implies that
\[
\hat{\pi}_1^{-1} \mathfrak{s}^2 \cdot \mathfrak{w}(X_1, Z_1, \mathfrak{S}_{\setminus D_0}) = \mathfrak{w}(X_1, Z, \mathfrak{S})
.\]
In other words,
\begin{equation}\label{eqn:Z_S^1}
Z_{\mathfrak{S}}^{[1]} = 2 \pi_1^{\ast} D_0 + (Z_1)_{\mathfrak{S}_{\setminus D_0}}^{[1]}
.\end{equation}
Since $\mathfrak{S}_{\setminus D_0}$ is base point free, we have
\[
m_f(r, Z_1) + T_{f_{[1]}}(r, \mathcal{O}_{X_1}(1)) + \mathcal{N}([Z_{f'}], r) \leq m_{f_{[1]}}(r, (Z_1)_{\mathfrak{S}_{\setminus D_0}}^{[1]}) + S_f(r)
.\]
Combined with \eqref{eqn:Z_S^1}, the above inequality yields
\begin{equation}\label{eqn:Z_S^1-D}
m_f(r, Z) + T_{f_{[1]}}(r, \mathcal{O}_{X_1}(1)) + \mathcal{N}([Z_{f'}], r) \leq m_{f_{[1]}}(r, Z_{\mathfrak S}^{[1]}) - m_f(r, D_0) + S_f(r)
.\end{equation}

Finally, let us investigate the case when $\Bs(\mathfrak{S}) \neq \emptyset$ and $Z_{\mathfrak s}$ is not an effective Cartier divisor.
Take the blow-up of $(X, V)$ with respect to $\mathfrak{s}$
\[
\upsilon \colon (\widehat{X}, \widehat{V}) \longrightarrow (X, V)
,\]
and denote by $E_0 = \upsilon^{\ast} Z_{\mathfrak s}$ the exceptional divisor of $\upsilon$.
Therefore, there is a commutative diagram
\begin{equation*}
\begin{tikzcd}[column sep=scriptsize]
(\widehat{X}_1, \widehat{V}_1) \arrow[r, "\hat{\pi}_1"] \arrow[d, dashed, "\upsilon_1"] & (\widehat{X}, \widehat{V}) \arrow[d, "\upsilon"] \\
(X_1, V_1) \arrow[r, "\pi_1"] & (X, V)
,\end{tikzcd}
\end{equation*}
where $\upsilon_1$ is induced by the differential $\dif \upsilon$ restricted on $\widehat{V}$.
Obviously, $\upsilon_1$ is holomorphic if and only if $\dif \upsilon_{\restriction \widehat{V}} \colon \widehat{V} \to \upsilon^{\ast} V$ is injective.
Accordingly, $\Ind(\upsilon_1) = \textnormal{P}(T_{\widehat{X}/X} \cap \widehat{V})$.

Moreover, there is a resolution $\varrho \colon (\widetilde{X}_1, \widetilde{V}_1) \to (\widehat{X}_1, \widehat{V}_1)$ of the rational map $\upsilon_1 \colon \allowbreak (\widehat{X}_1, \widehat{V}_1) \dashrightarrow (X_1, V_1)$,
i.e., $\varrho$ is the composition of a finite number of blow-ups of directed manifolds with smooth centers such that the composition $\upsilon_1 \circ \varrho$ gives a holomorphic morphism $\widetilde{\upsilon}_1 \colon (\widetilde{X}_1, \widetilde{V}_1) \to (X_1, V_1)$.
This amounts to the commutativity of the following diagram.
\begin{equation*}
\begin{tikzcd}[column sep=scriptsize]
(\widetilde{X}_1, \widetilde{V}_1) \arrow[r, "\varrho"] \arrow[rd, "\widetilde{\upsilon}_1"] & (\widehat{X}_1, \widehat{V}_1) \arrow[d, dashed, "\upsilon_1"] \arrow[r, "\hat{\pi}_1"] \arrow[d, dashed, "\upsilon_1"] & (\widehat{X}, \widehat{V}) \arrow[d, "\upsilon"] \\
 & (X_1, V_1) \arrow[r, "\pi_1"] & (X, V)
.\end{tikzcd}
\end{equation*}

\begin{remark}
Seeing that the composition of projective morphisms is also a projective morphism,
$\varrho \colon \widetilde{X}_1 \to \widehat{X}_1$ is indeed a birational projective morphism of projective varieties.
According to Theorem 7.17 in Chapter \uppercase\expandafter{\romannumeral2} of \cite{Hartshorne1977}, there exists a coherent ideal sheaf $\mathcal{I} \subset \mathcal{O}_{\widehat{X}_1}$ such that $\widetilde{X}_1$ is isomorphic to the blow-up $\Bl_{\mathcal{I}}(\widehat{X}_1)$ of $\widehat{X}_1$.
Let $\widehat{U}_1 \subset \widehat{X}_1$ be the largest open subset such that $\varrho$ restricted on $\varrho^{-1}(\widehat{U}_1)$ is an isomorphism.
Furthermore, $\mathcal{I}$ can be chosen such that $\Supp(\mathcal{O}_{\widehat{X}_1} / \mathcal{I}) = \widehat{X}_1 \setminus \widehat{U}_1$, since $\widehat{X}_1$ is non-singular
(cf. Exercises 7.11, (c) in Chapter \uppercase\expandafter{\romannumeral2} of \cite{Hartshorne1977}).
\end{remark}

Analogous to the bundle morphism \eqref{eqn:Gamma}, there is a line bundle morphism over $\widetilde{X}_1$
\begin{equation*}
\begin{tikzcd}[column sep=scriptsize]
\varrho^{\ast} \mathcal{O}_{\widehat{X}_1}(-1) \arrow[r, hook] & \varrho^{\ast} \hat{\pi}_1^{\ast} \widehat{V} \arrow[r, "\widetilde{\upsilon}_1^{\ast} \upsilon_{\ast}"] &[.5em] \widetilde{\upsilon}_1^{\ast} \mathcal{O}_{X_1}(-1)
,\end{tikzcd}
\end{equation*}
where $\upsilon_{\ast} \colon \widehat{V} \to \mathcal{O}_{X_1}(-1)$.
And it admits $\widetilde{\varGamma}_{\upsilon} = \varrho^{\ast} \textnormal{P}(T_{\widehat{X}/X} \cap \widehat{V}) \subset \widetilde{X}_1$ as its zero divisor.
Consequently, we immediately get
\begin{equation}\label{eqn:O(Gammau)}
\varrho^{\ast} \mathcal{O}_{\widehat{X}_1}(1) = \widetilde{\upsilon}_1^{\ast} \mathcal{O}_{X_1}(1) \otimes \mathcal{O}_{\widetilde{X}_1}(\widetilde{\varGamma}_{\upsilon})
.\end{equation}
It is evident that $W(\sigma_0 \circ \upsilon, \sigma_1 \circ \upsilon) = W(\sigma_0, \sigma_1) \circ \dif \upsilon_{\restriction \widehat{V}}$ for any $\sigma_0 \in \mathbb{S}_{\supset Z}, \sigma_1 \in \mathbb{S}$.
Hence $\omega(\sigma_0 \circ \upsilon, \sigma_1 \circ \upsilon) \circ \varrho = (\omega(\sigma_0, \sigma_1) \circ \widetilde{\upsilon}_1) \widetilde{\gamma}$,
where $\widetilde{\gamma}$ is a canonical section of $\mathcal{O}(\widetilde{\varGamma}_{\upsilon})$.
That is to say,
\[
\varrho^{\ast} \mathfrak{W}(\widehat{X}_1, \widehat{Z}, \upsilon^{\ast} \mathfrak{S}) = \widetilde{\varGamma}_{\upsilon} \cdot \widetilde{\upsilon}_1^{\ast} \mathfrak{W}(X_1, Z, \mathfrak{S})
,\]
where $\widehat{Z} = \upsilon^{\ast} Z$.
Thus
\[
\varrho^{-1} \mathfrak{w}(\widehat{X}_1, \widehat{Z}, \upsilon^{\ast} \mathfrak{S}) \cdot \mathcal{O}_{\widetilde{X}_1} = \widetilde{\upsilon}_1^{-1} \mathfrak{w}(X_1, Z, \mathfrak{S}) \cdot \mathcal{I}_{\widetilde{\varGamma}_{\upsilon}}
,\]
which amounts to saying that
\begin{equation}\label{eqn:Z_S^1+Gammau}
\varrho^{\ast} \widehat{Z}_{\upsilon^{\ast} \mathfrak{S}}^{[1]} = \widetilde{\upsilon}_1^{\ast} Z_{\mathfrak S}^{[1]} + \widetilde{\varGamma}_{\upsilon}
.\end{equation}

Let $f \colon (\mathbb{C}, T_{\mathbb{C}}) \to (X, V)$ be a holomorphic curve such that $f(\mathbb{C}) \not\subset \Supp Z$, and take the lifting $\hat{f}$ of $f$ to $(\widehat{X}, \widehat{V})$.
Therefore, $\hat{f} \colon (\mathbb{C}, T_{\mathbb{C}}) \to (\widehat{X}, \widehat{V})$ is a holomorphic curve such that $f = \upsilon \circ \hat{f}$ and $\hat{f}(\mathbb{C}) \not\subset \Supp \widehat{Z}$.
Observe that the scheme of base points of $\upsilon^{\ast} \mathfrak{S}$ is the exceptional divisor $E_0$.
Applying \eqref{eqn:Z_S^1-D} to $\hat{f}$, we obtain
\begin{equation}\label{eqn:Z_S^1-E}
m_{\hat{f}}(r, \widehat{Z}) + T_{\hat{f}_{[1]}}(r, \mathcal{O}_{\widehat{X}_1}(1)) + \mathcal{N}([\hat{f}'], r) \leq m_{\hat{f}_{[1]}}(r, \widehat{Z}_{\upsilon^{\ast} \mathfrak{S}}^{[1]}) - m_{\hat{f}}(r, E_0) + S_{\hat{f}}(r)
,\end{equation}
where $\hat{f}_{[1]}$ is the canonical lifting of $\hat{f}$ to $(\widehat{X}_1, \widehat{V}_1)$.
Take the lifting $\tilde{f}_{[1]}$ of $\hat{f}_{[1]}$ to $(\widetilde{X}, \widetilde{V})$.
Then $\hat{f}_{[1]} = \varrho \circ \tilde{f}_{[1]}$ and $f_{[1]} = \widetilde{\upsilon}_1 \circ \tilde{f}_{[1]}$.
Accordingly, \eqref{eqn:O(Gammau)} yields
\begin{equation}\label{eqn:TO(Gammau)}
T_{\hat{f}_{[1]}}(r, \mathcal{O}_{\widehat{X}_1}(1)) = T_{f_{[1]}}(r, \mathcal{O}_{X_1}(1)) + T_{\tilde{f}_{[1]}}(r, \widetilde{\varGamma}_{\upsilon})
.\end{equation}
It follows from \eqref{eqn:Z_S^1+Gammau} that
\begin{equation}
m_{\hat{f}_{[1]}}(r, \widehat{Z}_{\upsilon^{\ast} \mathfrak{S}}^{[1]}) = m_{f_{[1]}}(r, Z_{\mathfrak S}^{[1]}) + m_{\tilde{f}_{[1]}}(r, \widetilde{\varGamma}_{\upsilon})
.\end{equation}
Actually, by the definition of $\widetilde{\varGamma}_{\upsilon}$, we have $N_{\tilde{f}_{[1]}}(r, \widetilde{\varGamma}_{\upsilon}) = \mathcal{N}([Z_{\dif \upsilon(\hat{f})}], r)$.
Consequently, $f' = \dif \upsilon(\hat{f}) \cdot \hat{f}'$ leads to
\begin{equation}
\mathcal{N}([Z_{f'}], r) = N_{\tilde{f}_{[1]}}(r, \widetilde{\varGamma}_{\upsilon}) + \mathcal{N}([Z_{\hat{f}'}], r)
.\end{equation}
Seeing that $\mathcal{O}_{\widehat{X}}(-E_0)$ is relatively ample over $X$, $\upsilon^{\ast} A \otimes \mathcal{O}_{\widehat{X}}(-E_0)$ is ample for a sufficiently ample line bundle $A$ on $X$.
We denote this ample line bundle on $\widehat{X}$ by $\widehat{A}$.
Thanks to the fact that $\hat{f}(\mathbb{C}) \not\subset \Supp E_0$, we see that $T_{\hat{f}}(r, \widehat{A}) \leq T_f(r, A) + O(1)$, which implies that
\begin{equation}\label{eqn:Sfhat}
S_{\hat{f}}(r) \leq S_f(r)
.\end{equation}
Substituting \eqref{eqn:TO(Gammau)}--\eqref{eqn:Sfhat} into \eqref{eqn:Z_S^1-E} , we conclude that
\[
m_f(r, Z) + T_{f_{[1]}}(r, \mathcal{O}_{X_1}(1)) + \mathcal{N}([Z_{f'}], r) \leq m_{f_{[1]}}(r, Z_{\mathfrak S}^{[1]}) - m_f(r, Z_{\mathfrak s}) + S_f(r)
.\qedhere\]
\end{proof}

\begin{remark}
In the particular case when $\mathfrak{S}$ separate the $1$-jets at every point of $X$, we immediately have $Z_{\mathfrak S}^{[1]} = Z^{(1)}$ and $Z_{\mathfrak s} = \emptyset$.
It implies that Theorem \ref{thm:Z_S^1} in this case coincides with Theorem \ref{thm:Z^(1)}.
\end{remark}

\subsection{General form of AALD}\label{ssc:GAALD}

For arbitrary closed subschemes $Z_1, \dots, Z_q$ of $X$, following \cite{Ru2020}, we will use the notation
\[
\lambda_{\bigvee_{i=1}^q Z_i} \bydef \mathop{\max}_{1 \leq i \leq q} \lambda_{Z_i}
,\]
for the sake of brevity. Hence we obtain
\begin{equation}\label{eqn:proximityvee}
m_f(r, \bigvee_{i=1}^q Z_i) = \int_0^{2\pi} \!\mathop{\max}_{1 \leq i \leq q} \lambda_{Z_i} \circ f(re^{i\theta}) \frac{\dif \theta}{2\pi}
.\end{equation}
Here $\bigvee_{i=1}^q Z_i$ is indeed a birational divisor, or b-divisor on $X$ introduced by Shokurov \cite{Shokurov1996}.
For the precise definition and more details, see \cite{Shokurov1996} or Section 1.7 of \cite{Corti2007}.
We easily check that $\lambda_{\bigvee_{i=1}^q Z_i} \in \mathcal{C}(X)/=_{\langle X \rangle}$.
Actually, it is a generalized Weil function, or b-Weil function on $X$ introduced by Vojta \cite{Vojta1996}.

Following the same reasoning as in the proof of Theorem \ref{thm:Z^(1)}, we conclude the following theorem from Lemma \ref{lem:GALLDD}.
\begin{theorem}\label{thm:Z^(1)vee}
Let $(X, V)$ be a directed projective manifold, and let $Z_1, \dots, Z_q$ be closed subschemes of $X$.
Let $f \colon (\mathbb{C}, T_{\mathbb{C}}) \to (X, V)$ be a non-constant holomorphic curve such that $f(\mathbb{C}) \not\subset \Supp (Z_1 + \dots + Z_q)$. Then
\[
m_f(r, \bigvee_{i=1}^q Z_i) + T_{f_{[1]}}(r, \mathcal{O}_{X_1}(1)) + \mathcal{N}([Z_{f'}], r) \leq m_{f_{[1]}}(r, \bigvee_{i=1}^q Z_i^{(1)}) + S_f(r)
.\]
\end{theorem}

Following the same reasoning as in the proof of Theorem \ref{thm:Z_S^1}, Theorem \ref{thm:Z^(1)vee} leads to the following theorem.
\begin{theorem}\label{thm:Z_S^1vee}
Let $(X, V)$ be a directed projective manifold, and let $Z_1, \dots, Z_q$ be closed subschemes of $X$.
Choose a linear system $\mathfrak{S}$ on $X$ satisfying that $\mathcal{I}_{Z_i} = \mathfrak{s}_{\supset Z_i}$ for each $1 \leq i \leq q$.
Let $f \colon (\mathbb{C}, T_{\mathbb{C}}) \to (X, V)$ be a non-constant holomorphic curve such that $f_{[1]}(\mathbb{C}) \not\subset \Bs(\mathfrak{W}(X_1, \mathfrak{S}))$.
Suppose that $f(\mathbb{C}) \not\subset \Supp (Z_1 + \dots + Z_q)$.
We thus have
\[
m_f(r, \bigvee_{i=1}^q Z_i) + T_{f_{[1]}}(r, \mathcal{O}_{X_1}(1)) + \mathcal{N}([Z_{f'}], r) \leq m_{f_{[1]}}(r, \bigvee_{i=1}^q (Z_i)_{\mathfrak S}^{[1]}) - m_f(r, Z_{\mathfrak s}) + S_f(r)
.\]
\end{theorem}

\begin{lemma}\label{lem:maxsum}
Let $Z_1, \dots, Z_q$ be arbitrary closed subschemes of $X$. Then
\[
\mathop{\max}_{\substack{I \subset J \\ |I| = \ell}} \sum_{i \in I} \lambda_{Z_i} =_{\langle X \rangle} \sum_{j=1}^{\ell} \mathop{\max}_{\substack{I \subset J \\ |I| = j}} \lambda_{Z_I}
\]
holds for each $1 \leq \ell \leq q$, where $|I|$ denotes the cardinality of $I$, and we hereinafter write $J = \{1, 2, \dots, q\}$, and $Z_I = \bigcap_{i \in I} Z_i$ for $I \subset J$.

In addition, let $\ell_0$ be the maximum of $|I|$ for all the subsets $I \subset J$ satisfying that $Z_I \neq \emptyset$. Then
\[
\sum_{i=1}^q \lambda_{Z_i} =_{\langle X \rangle} \mathop{\max}_{\substack{I \subset J \\ |I| = \ell_0}} \sum_{i \in I} \lambda_{Z_i}
.\]
\end{lemma}

\begin{proof}
According to Proposition \ref{prp:lambda_Z}, \ref{itm:lambdacap}), we get
\[
\mathop{\max}_{\substack{I \subset J \\ |I| = \ell}} \sum_{i \in I} \lambda_{Z_i} =_{\langle X \rangle} \sum_{j=1}^{\ell} \mathop{\max}_{\substack{I \subset J \\ |I| = j}} \mathop{\min}_{i \in I} \lambda_{Z_i} =_{\langle X \rangle} \sum_{j=1}^{\ell} \mathop{\max}_{\substack{I \subset J \\ |I| = j}} \lambda_{Z_I}
.\]
Notice that $Z_I = \emptyset$ for all the subsets $I \subset J$ with $|I| > \ell_0$.
It implies that $\mathop{\min}_{i \in I} \lambda_{Z_i}$ is a continuous function on $X$, for all the subsets $I \subset J$ with $|I| > \ell_0$.
\[
\sum_{i=1}^q \lambda_{Z_i} =_{\langle X \rangle} \sum_{j=1}^q \mathop{\max}_{\substack{I \subset J \\ |I| = j}} \mathop{\min}_{i \in I} \lambda_{Z_i} =_{\langle X \rangle} \sum_{j=1}^{\ell_0} \mathop{\max}_{\substack{I \subset J \\ |I| = j}} \mathop{\min}_{i \in I} \lambda_{Z_i} =_{\langle X \rangle} \mathop{\max}_{\substack{I \subset J \\ |I| = \ell_0}} \sum_{i \in I} \lambda_{Z_i}
.\qedhere\]
\end{proof}

Thanks to Theorem \ref{thm:Z^(1)vee} and Lemma \ref{lem:maxsum}, we conclude General form of Algebro-geometric Version of Ahlfors' Lemma on Logarithmic Derivative (GAALD for short) for $1$-jets.
\begin{theorem}[GAALD for $1$-jets]\label{thm:Z^(1)veesum}
Let $(X, V)$ be a directed projective manifold, and let $Z_1, \dots, Z_q$ be closed subschemes of $X$.
Take an integer $1 \leq \ell \leq q$.
Let $f \colon (\mathbb{C}, T_{\mathbb{C}}) \to (X, V)$ be a non-constant holomorphic curve such that $f(\mathbb{C}) \not\subset \Supp (Z_1 + \dots + Z_q)$. Then,
\[
m_f(r, \bigvee_{\substack{I \subset J \\ |I| = \ell}} \sum_{i \in I} Z_i) + T_{f_{[1]}}(r, \mathcal{O}_{X_1}(\ell)) + \ell \mathcal{N}([Z_{f'}], r)
\leq m_{f_{[1]}}(r, \bigvee_{\substack{I \subset J \\ |I| = \ell}} \sum_{i \in I} Z_i^{(1)}) + S_f(r)
.\]
\end{theorem}

\begin{proof}
It follows from Lemma \ref{lem:maxsum} and Proposition \ref{prp:cap^(1)=^(1)cap} that
\begin{align}
m_f(r, \bigvee_{\substack{I \subset J \\ |I| = \ell}} \sum_{i \in I} Z_i) &= \sum_{j=1}^{\ell} m_f(r, \bigvee_{\substack{I \subset J \\ |I| = j}} Z_I) \label{eqn:veesum}\\
m_f(r, \bigvee_{\substack{I \subset J \\ |I| = \ell}} \sum_{i \in I} Z_i^{(1)}) &= \sum_{j=1}^{\ell} m_f(r, \bigvee_{\substack{I \subset J \\ |I| = j}} \bigcap_{i \in I} Z_i^{(1)})
= \sum_{j=1}^{\ell} m_f(r, \bigvee_{\substack{I \subset J \\ |I| = j}} Z_I^{(1)}) \label{eqn:veesum^(1)}
.\end{align}

By Theorem \ref{thm:Z^(1)vee}, for each $1 \leq j \leq q$, we have
\[
m_f(r, \bigvee_{\substack{I \subset J \\ |I| = j}} Z_I) + T_{f_{[1]}}(r, \mathcal{O}_{X_1}(1)) + \mathcal{N}([Z_{f'}], r) \leq m_f(r, \bigvee_{\substack{I \subset J \\ |I| = j}} Z_I^{(1)}) + S_f(r)
.\]
Combining all these inequalities for $1 \leq j \leq \ell$, we get
\[
\sum_{j=1}^{\ell} m_f(r, \bigvee_{\substack{I \subset J \\ |I| = j}} Z_I) + T_{f_{[1]}}(r, \mathcal{O}_{X_1}(\ell)) + \ell \mathcal{N}([Z_{f'}], r) \leq \sum_{j=1}^{\ell} m_f(r, \bigvee_{\substack{I \subset J \\ |I| = j}} Z_I^{(1)}) + S_f(r)
.\]
Substituting \eqref{eqn:veesum} and \eqref{eqn:veesum^(1)} into the above inequality, the theorem is verified.
\end{proof}

Likewise, we obtain the following theorem from Theorem \ref{thm:Z_S^1vee} and Lemma \ref{lem:maxsum}, and call it GAALD for $1$-jets with respect to linear system $\mathfrak{S}$.
\begin{theorem}\label{thm:Z_S^1veesum}
Let $(X, V)$ be a directed projective manifold, and let $Z_1, \dots, Z_q$ be closed subschemes of $X$.
Choose a linear system $\mathfrak{S}$ on $X$ satisfying that $\mathcal{I}_{Z_i} = \mathfrak{s}_{\supset Z_i}$ for each $1 \leq i \leq q$.
Take an integer $1 \leq \ell \leq q$.
Let $f \colon (\mathbb{C}, T_{\mathbb{C}}) \to (X, V)$ be a non-constant holomorphic curve such that $f_{[1]}(\mathbb{C}) \not\subset \Bs(\mathfrak{W}(X_1, \mathfrak{S}))$.
Suppose that $f(\mathbb{C}) \not\subset \Supp (Z_1 + \dots + Z_q)$.
Then we have
\begin{align*}
&m_f(r, \!\bigvee_{\substack{I \subset J \\ |I| = \ell}} \!\sum_{i \in I} Z_i) + T_{f_{[1]}}(r, \mathcal{O}_{X_1}(\ell)) + \ell \mathcal{N}([Z_{f'}], r) \\
\leq \, &m_{f_{[1]}}\big(r, \!\bigvee_{\substack{I \subset J \\ |I| = \ell}} \!\sum_{i \in I} (Z_i)_{\mathfrak S}^{[1]}\big) - \ell m_f(r, Z_{\mathfrak s}) + S_f(r)
.\end{align*}
\end{theorem}

\section{Applications}\label{sec:Applications}

In this section, we are going to show the applications of GAALD and its transform, and get the second main theorem type results for holomorphic curves.

\subsection{Applications of GAALD}

For arbitrary closed subschemes $Z_1, \dots, Z_q$ of $X$, set
\[
\ell_k \bydef \max \{|I| \mid I \subset J, \, \bigcap_{i \in  I} Z_i^{(k)} \neq \emptyset\} 
.\]
We easily see that $\ell_k$ is a non-increasing sequence for $k$.
Therefore, we conclude the following theorem from GAALD.
\begin{theorem}\label{thm:Z^(k)sum}
Let $(X, V)$ be a directed projective manifold, and let $Z_1, \dots, Z_q$ be closed subschemes of $X$.
For each $k \in \mathbb{Z}_{+}$, let $f \colon (\mathbb{C}, T_{\mathbb{C}}) \to (X, V)$ be a non-constant holomorphic curve such that $f(\mathbb{C}) \not\subset \Supp (Z_1 + \dots + Z_q)$.
Then,
\[
\sum_{i =1}^q m_f(r, Z_i) + T_{f_{[k]}}(r, \mathcal{O}_{X_k}(\bm{\ell})) + \sum_{j=0}^{k-1} \ell_j \mathcal{N}([Z_{f_{[j]}'}], r)
\leq m_{f_{[k]}}(r, \bigvee_{\substack{I \subset J \\ |I| = \ell_k}} \sum_{i \in I} Z_i^{(k)}) + S_f(r)
,\]
where $\bm{\ell} = (\ell_0, \dots, \ell_{k-1}) \in \mathbb{Z}_{+}^k$ and $\ell_j$ is as above for each $0 \leq j \leq k$.
\end{theorem}

\begin{proof}
Observe that $f_{[j-1]}(\mathbb{C}) \not\subset \Supp \pi_{0, j-1}^{\ast} (Z_1 + \dots + Z_q)$, for every $1 \leq j \leq k$.
Hence $f_{[j-1]}(\mathbb{C}) \not\subset \Supp (Z_1^{(j-1)} + \dots + Z_q^{(j-1)})$, since $\Supp (Z_1^{(j-1)} + \dots + Z_q^{(j-1)}) \subset \Supp \pi_{0, j-1}^{\ast} (Z_1 + \dots + Z_q)$.
Think of $(X_j, V_j)$ as the Demailly-Semple $1$-jet tower of $(X_{j-1}, V_{j-1})$ and apply Theorem \ref{thm:Z^(1)veesum} to $f_{[j-1]}$ with $\ell = \ell_{j-1}$.
According to Lemma \ref{lem:maxsum} and \eqref{eqn:Sfk},
\begin{align*}
&m_f(r, \hspace{-.5eM} \bigvee_{\substack{I \subset J \\ |I| = \ell_{j-1}}} \hspace{-.5eM} \sum_{i \in I} Z_i^{(j-1)}) + T_{f_{[j]}}(r, \mathcal{O}_{X_1}(\ell_{j-1})) + \ell_{j-1} \mathcal{N}([Z_{f_{[j-1]}'}], r) \\
\leq \, & m_{f_{[j]}}(r, \hspace{-.5eM} \bigvee_{\substack{I \subset J \\ |I| = \ell_{j-1}}} \hspace{-.5eM} \sum_{i \in I} Z_i^{(j)}) + S_{f_{[j-1]}}(r)
\leq m_{f_{[j]}}(r, \!\bigvee_{\substack{I \subset J \\ |I| = \ell_j}} \!\sum_{i \in I} Z_i^{(j)}) + S_f(r)
\end{align*}
Combining all the inequalities for $1 \leq j \leq k$, we thus have
\[
m_f(r, \bigvee_{\substack{I \subset J \\ |I| = \ell_0}} \sum_{i \in I} Z_i) + T_{f_{[k]}}(r, \mathcal{O}_{X_k}(\bm{\ell})) + \sum_{j=0}^{k-1} \ell_j \mathcal{N}([Z_{f_{[j]}'}], r)
\leq m_{f_{[k]}}(r, \bigvee_{\substack{I \subset J \\ |I| = \ell_k}} \sum_{i \in I} Z_i^{(k)}) + S_f(r)
.\]
Therefore, the theorem follows from $\displaystyle\sum_{i =1}^q m_f(r, Z_i) = m_f(r, \bigvee_{\substack{I \subset J \\ |I| = \ell_0}} \sum_{i \in I} Z_i)$.
\end{proof}

Next, recall the simple normal crossings divisors on a complex manifold $X$.
One says a divisor $D = \sum_{i=1}^q D_i$ is simple normal crossing, if $D$ is normal crossing and each irreducible component $D_i$ is non-singular.
In particular, the hyperplanes $H_1, \dots, H_q$ of $\mathbb{P}^n$ are said to be in general position, if
\[
\dim \bigcap_{i \in  I} H_i \leq n - |I|
,\]
for any subset $I \subset J$, where $\dim \emptyset = - \infty$ by convention.
Obviously, the hyperplanes $H_1, \dots, H_q$ of $\mathbb{P}^n$ are in general position, if and only if $D = \sum_{i=0}^q H_i$ is a simple normal crossings divisor on $\mathbb{P}^n$.

For the simple normal crossings divisor $D$ in general, we easily check that $\ell_0 = n$, and $\ell_j = n-1$ for each $j \geq 1$.
So is it, even for the hyperplanes of $\mathbb{P}^n$ in general position.
On the other hand, the Wronskian in $\mathbb{P}^n$ is indeed associated to the unique global section $\overline{\omega} \in H^0(X_n, \mathcal{O}_{X_n}(\bm{a}^n) \otimes \pi_{0, n}^{\ast} \mathcal{O}_{\mathbb{P}^n}(n+1))$.
In the next section, we will show that, compared with Theorem \ref{thm:Z^(k)sum}, Theorem \ref{thm:Z_S^1veesum} is a more convenient tool to deal with SMT type questions.

\subsection{Applications of GAALD with respect to a linear system}

Let $Z_1$ and $Z_2$ be two closed subschemes of $X$.
Take a linear system $\mathfrak{S} \subset |L|$ on $X$ such that both $\mathfrak{S}_{\supset Z_1}$ and $\mathfrak{S}_{\supset Z_2}$ are non-empty.
For any positive integer $k$, we introduce the notation $\mathfrak{W}(X_k, Z_1, Z_2, \mathfrak{S})$ and $\widebar{\mathfrak W}(X_k, Z_1, Z_2, \mathfrak{S})$ for the linear systems on $X_k$ corresponding to the linear subspace
\[
\mathbb{W}(X_k, Z_1, Z_2, \mathfrak{S}) \bydef \Span \{\omega(\sigma_0, \sigma_1, \dots, \sigma_k) \mid \sigma_0 \in \mathbb{S}_{\supset Z_1}, \, \sigma_1 \in \mathbb{S}_{\supset Z_2}, \, \sigma_2, \dots, \sigma_k \in \mathbb{S}\}
\]
of $H^0(X_k, \mathcal{O}_{X_k}(k') \otimes \pi_{0, k}^{\ast} L^{k+1})$, and the linear subspace
\[
\widebar{\mathbb W}(X_k, Z_1, Z_2, \mathfrak{S}) \bydef \Span \{\widebar{\omega}(\sigma_0, \sigma_1, \dots, \sigma_k) \mid \sigma_0 \in \mathbb{S}_{\supset Z_1}, \, \sigma_1 \in \mathbb{S}_{\supset Z_2}, \, \sigma_2, \dots, \sigma_k \in \mathbb{S}\}
\]
of $H^0(X_k, \mathcal{O}_{X_k}(\bm{a}^k) \otimes \pi_{0, k}^{\ast} L^{k+1})$, respectively.
Moreover, we denote the base ideal sheaf of $\mathfrak{W}(X_k, Z_1, Z_2, \mathfrak{S})$ and $\widebar{\mathfrak W}(X_k, Z_1, Z_2, \mathfrak{S})$ by $\mathfrak{w}(X_k, Z_1, Z_2, \mathfrak{S})$ and $\widebar{\mathfrak w}(X_k, Z_1, Z_2, \mathfrak{S})$, respectively.

\begin{remark}
We see that $\mathbb{W}(X_k, D, D, \mathfrak{S}) = 0$ and $\widebar{\mathbb W}(X_k, D, D, \mathfrak{S}) = 0$, for any divisor $D \in \mathfrak{S}$.
\end{remark}

Now, we are going to investigate the closed subscheme of $X_k$ with the ideal sheaf $\widebar{\mathfrak w}(X_k, Z_1, Z_2, \mathfrak{S})$, denoted by $(Z_1, Z_2)_{\mathfrak{S}}^{[k]}$.
We get the following lemmas.
\begin{lemma}\label{lem:(D,Z)_S^1_WD^k}
Let $Z$ be a closed subscheme of $X$.
Let $\mathfrak{S} \subset |L|$ be linear systems on $X$ satisfying that $\mathfrak{S}_{\supset Z} \neq \emptyset$.
For any $k \in \mathbb{Z}_{+}$ and for each $D \in \mathfrak{S}$ except $D = Z$, we have
\begin{gather*}
\widebar{\mathfrak w}((X_1)_k, \mathfrak{W}(X_1, D, \mathfrak{S})) = \pi_{0, k+1}^{-1} \mathcal{I}_D^k \cdot \widebar{\mathfrak w}(X_{k+1}, D, \mathfrak{S}), \\
\widebar{\mathfrak w}((X_1)_k, (D, Z)_{\mathfrak{S}}^{[1]}, \mathfrak{W}(X_1, D, \mathfrak{S})) = \pi_{0, k+1}^{-1} \mathcal{I}_D^k \cdot \widebar{\mathfrak w}(X_{k+1}, D, Z, \mathfrak{S})
.\end{gather*}
In other words,
\begin{align*}
Z_{\widebar{\mathfrak w}((X_1)_k, \mathfrak{W}(X_1, D, \mathfrak{S}))} &= D_{\mathfrak{S}}^{[k+1]} + k \pi_{0, k+1}^{\ast} D, \\
((D, Z)_{\mathfrak{S}}^{[1]})_{\mathfrak{W}(X_1, D, \mathfrak{S})}^{[k]} &= (D, Z)_{\mathfrak{S}}^{[k+1]} + k \pi_{0, k+1}^{\ast} D
.\end{align*}
\end{lemma}

\begin{proof}
Take an open subset $U \subset X$ on which $L$ is trivialized.
Let $\sigma$ be a canonical section of $\mathcal{O}(D)$.
Applying the generalized Leibniz rule, for each $0 \leq j \leq k$ and for any $\varsigma \in \mathbb{S} \subset H^0(X, L)$,
\begin{align*}
\dif^{[\ell]} W_U(\sigma, \varsigma) &= \dif^{[\ell]} \left(\sigma_U\dif^{[1]}\varsigma - \varsigma_U\dif^{[1]}\sigma\right) \\
&= \sum_{j=0}^\ell \binom{\ell}{j} \left(\dif^{[\ell-j]}\sigma\dif^{[j+1]}\varsigma - \dif^{[\ell-j]}\varsigma\dif^{[j+1]}\sigma\right) \\
&= \sigma_U\dif^{[\ell+1]}\varsigma - \varsigma_U\dif^{[\ell+1]}\sigma + \sum_{j=1}^\ell \left(\binom{\ell}{j-1} - \binom{\ell}{j}\right) \dif^{[\ell-j+1]}\sigma\dif^{[j]}\varsigma
,\end{align*}
where we abusively write $\varsigma_U \bydef \dif_{U, V}^{[0]}\varsigma_U = p_{k, V}^{\ast} \varsigma_U$ and $\dif^{[j]}\varsigma \bydef \dif_{U, V}^{[j]}\varsigma_U$ for each $0 \leq j \leq \ell$.
Furthermore, take an affine open subset $U_1 \subset \pi_1^{-1}(U) \subset X_1$, and choose a local holomorphic section $\tau$ of $\mathcal{O}_{X_1}(1)(\pi_1^{-1}(U))$ such that $\tau$ is nowhere zero on $U_1$, namely, $\tau_{\restriction U_1}^{\scriptscriptstyle \vee} \in \mathcal{O}_{X_1}(1)(U_1)$.
We then obtain a local trivialization of $\mathcal{O}_{X_1}(1)(U_1)$ given by
\begin{align*}
\tau_{\restriction U_1}^{\scriptscriptstyle \vee} \otimes \colon \mathcal{O}_{X_1}(1)(U_1) &\longrightarrow \mathcal{O}(U_1) \\
\omega &\longmapsto \tau_{\restriction U_1}^{\scriptscriptstyle \vee} \otimes \omega
.\end{align*}
Let $t$ be a meromorphic function on the total space of $V_{\restriction U}$ satisfying that $\dfrac{1}{t} = (\pi_1)_{\ast} \tau^{\scriptscriptstyle \vee}$.
Apparently, $t$ is holomorphic and nowhere zero on $\widetilde{U}_1$, where $\widetilde{U}_1$ is an open subset of the total space of $V$ such that $\textnormal{P}(\widetilde{U}_1) = U_1$.
Thus the local trivializations of $L$ and $\mathcal{O}_{X_1}(1)$ induce a trivialization of $\mathcal{O}_{X_1}(1) \otimes \pi_1^{\ast} L^2$ on $U_1$,
which associates $\omega_{U_1}(\sigma_0, \sigma_1)$ to $\omega(\sigma_0, \sigma_1) \in H^0(X_1, \mathcal{O}_{X_1}(1) \otimes \pi_1^{\ast} L^2)$.
Consider the natural map
\begin{align*}
\pr_{k, \ell} \colon J_{k+\ell}V \setminus (X \times \{0\}) &\longrightarrow J_kV_{\ell} \\
j_{k+\ell}f &\longmapsto j_kf_{[\ell]}
.\end{align*}
We see that $\omega_{U_1}(\sigma_0, \sigma_1) \circ \pr_{0, 1} = t W_U(\sigma_0, \sigma_1)$ on $\widetilde{U}_1 \setminus (U \times \{0\})$.
Therefore, for any $\varsigma_0, \dots, \varsigma_k \in \mathbb{S}$, we get
\begin{align*}
&W_{U_1}(\omega(\sigma, \varsigma_0), \dots, \omega(\sigma, \varsigma_k)) \circ \pr_{k, 1} \\
=& t^{k+1} W_{\widetilde{U}_1 \setminus (U \times \{0\})}(W_U(\sigma, \varsigma_0), \dots, W_U(\sigma, \varsigma_k)) \\
=& t^{k+1}
\begin{vmatrix}
\dif^{[0]}W_U(\sigma, \varsigma_0) & \cdots & \dif^{[0]}W_U(\sigma, \varsigma_k) \\
\vdots & \ddots & \vdots \\
\dif^{[k]}W_U(\sigma, \varsigma_0) & \cdots & \dif^{[k]}W_U(\sigma, \varsigma_k)
\end{vmatrix} \\
=& t^{k+1} \sigma_U^{-1}
\begin{vmatrix}
\dif^{[0]}\sigma & \dif^{[0]}\varsigma_0 & \cdots & \dif^{[0]}\varsigma_k \\
0 & \dif^{[0]}W_U(\sigma, \varsigma_0) & \cdots & \dif^{[0]}W_U(\sigma, \varsigma_k) \\
\vdots & \vdots & \ddots & \vdots \\
0 & \dif^{[k]}W_U(\sigma, \varsigma_0) & \cdots & \dif^{[k]}W_U(\sigma, \varsigma_k)
\end{vmatrix} \\
=& t^{k+1} \sigma_U^{-1}
\begin{vmatrix}
\dif^{[0]}\sigma & \dif^{[0]}\varsigma_0 & \cdots & \dif^{[0]}\varsigma_k \\
\dif^{[0]}\sigma\dif^{[1]}\sigma & \dif^{[0]}\sigma\dif^{[1]}\varsigma_0 & \cdots & \dif^{[0]}\sigma\dif^{[1]}\varsigma_k \\
\vdots & \vdots & \ddots & \vdots \\
\dif^{[0]}\sigma\dif^{[k+1]}\sigma & \dif^{[0]}\sigma\dif^{[k+1]}\varsigma_0 & \cdots & \dif^{[0]}\sigma\dif^{[k+1]}\varsigma_k
\end{vmatrix} \\
=& t^{k+1} \sigma_U^k W_U(\sigma, \varsigma_0, \dots, \varsigma_k)
.\end{align*}
Notice that all such open subsets $U$ give an open covering of $X$, and all such open subsets $U_1$ give an affine open covering of $X_1$.
Accordingly,
\[
\widebar{\omega}(\omega(\sigma, \varsigma_0), \dots, \omega(\sigma, \varsigma_k)) = \pi_{0, k+1}^{\ast} \sigma^k \widebar{\omega}(\sigma, \varsigma_0, \dots, \varsigma_k)
,\]
which amounts to saying that
\[
\widebar{\mathfrak W}((X_1)_k, \mathfrak{W}(X_1, D, \mathfrak{S})) = k \pi_{0, k+1}^{\ast} D \cdot \widebar{\mathfrak W}(X_{k+1}, D, \mathfrak{S})
.\]
Consequently,
\[
\widebar{\mathfrak w}((X_1)_k, \mathfrak{W}(X_1, D, \mathfrak{S})) = \pi_{0, k+1}^{-1} \mathcal{I}_D^k \cdot \widebar{\mathfrak w}(X_{k+1}, D, \mathfrak{S})
.\]
Observe that $\omega(\sigma, \varsigma_0) \in \mathbb{W}(X_1, D, Z, \mathfrak{S})$ if and only if $\varsigma_0 \in \mathbb{S}_{\supset Z}$.
Hence
\[
\widebar{\mathfrak w}((X_1)_k, (D, Z)_{\mathfrak{S}}^{[1]}, \mathfrak{W}(X_1, D, \mathfrak{S})) = \pi_{0, k+1}^{-1} \mathcal{I}_D^k \cdot \widebar{\mathfrak w}(X_{k+1}, D, Z, \mathfrak{S})
.\qedhere\]
\end{proof}

We make a convention that $(D, Z)_{\mathfrak{S}}^{[k]} = X_k$ if $D = Z$.
Using Crofton's formula, we calculate the average of the proximity functions as follows.
\begin{lemma}\label{lem:(D,Z)_S^kavg}
Let $(X, V)$ be a directed projective manifold.
For each $0 \leq k \leq N-1$, let $f \colon (\mathbb{C}, T_{\mathbb{C}}) \to (X, V)$ be a non-constant holomorphic curve such that $f_{[k]}(\mathbb{C}) \not\subset \Bs(\mathfrak{W}(X_k, \mathfrak{S}))$.
For any linear system $\mathfrak{S} \subset |L|$ on $X$, we have
\[
\int_{D \in \mathfrak{S}} \!m_f(r, D_{\mathfrak{S}}^{[k]}) \dif \rho(D) = m_f(r, Z_{\widebar{\mathfrak w}(X_k, \mathfrak{S})}) + O(1)
,\]
where $N = \dim \mathfrak{S}$ and $\dif \rho(D)$ is the unique measure on $\mathfrak{S} \cong (\mathbb{P}^n)^{\dual}$ invariant under the unitary group and satisfying $\rho(\mathfrak{S}) = 1$.\\
Take a closed subscheme $Z$ of $X$.
Choose a linear system $\mathfrak{S} \subset |L|$ on $X$ such that $\mathcal{I}_{Z} = \mathfrak{s}_{\supset Z}$.
Then
\[
\int_{D \in \mathfrak{S}} \!m_f(r, (D, Z)_{\mathfrak{S}}^{[k]}) \dif \rho(D) = m_f(r, Z_{\mathfrak{S}}^{[k]}) + O(1)
\]
holds for any non-constant holomorphic curve satisfying that $f(\mathbb{C}) \not\subset \Supp Z$ besides.
What's more, let $Z_1, \dots, Z_q$ be closed subschemes of $X$.
Choose a linear system $\mathfrak{S} \subset |L|$ on $X$ such that $\mathcal{I}_{Z_i} = \mathfrak{s}_{\supset Z_i}$ for each $1 \leq i \leq q$.
\[
\int_{D \in \mathfrak{S}} \!m_f(r, \!\bigvee_{i=1}^q (D, Z_i)_{\mathfrak{S}}^{[k]}) \dif \rho(D) = m_f(r, \!\bigvee_{i=1}^q (Z_i)_{\mathfrak{S}}^{[k]}) + O(1)
\]
holds for any non-constant holomorphic curve satisfying that $f(\mathbb{C}) \not\subset \Supp (Z_1 + \dots + Z_q)$ besides.
\end{lemma}

\begin{proof}
Let $\mathbb{S} \subset H^0(X, L)$ be the linear subspace corresponding to $\mathfrak{S}$.
Take a basis $\{\varsigma_i \mid i = 0, 1, \dots, N\}$ of $\mathbb{S}$.
For each array of indices $0 \leq i_1 < \dots < i_k \leq N$, set
\[
\mathbb{S}_{i_1, \dots, i_k} \bydef \{\widebar{\omega}(\sigma, \varsigma_{i_1}, \dots, \varsigma_{i_k}) \mid \sigma \in \mathbb{S}\}
.\]
Naturally, there is a surjective homomorphism 
\begin{align*}
\beta_{i_1, \dots, i_k} \colon \mathbb{S} &\longrightarrow \mathbb{S}_{i_1, \dots, i_k} \\
\sigma &\longmapsto \widebar{\omega}(\sigma, \varsigma_{i_1}, \dots, \varsigma_{i_k})
,\end{align*}
for each array of indices $0 \leq i_1 < \dots < i_k \leq N$.
It is evident that
\[
\widebar{\mathbb W}(X_k, \mathfrak{S}) = \hspace{-1eM} \sum_{0 \leq i_1 < \dots < i_k \leq N} \hspace{-1eM} \mathbb{S}_{i_1, \dots, i_k}
,\]
and
\[
Z_{\beta}(D) = \hspace{-1eM} \bigcap_{0 \leq i_1 < \dots < i_k \leq N} \hspace{-1eM} \beta_i(D) = D_{\mathfrak{S}}^{[k]}
,\]
for any $D \in \mathfrak{S}$, where $Z_{\beta}$ is defined by \eqref{eqn:Zbeta}.
Applying Theorem \ref{thm:m_Zbetaavg} and Remark \ref{rmk:m_Zbetaavg} with $\mathfrak{T} = \widebar{\mathfrak W}(X_k, \mathfrak{S})$, and $\mathfrak{S}_{i_1, \dots, i_k}$, $\mathfrak{S}$ as above, we have
\[
\int_{D \in \mathfrak{S}} \!m_f(r, D_{\mathfrak{S}}^{[k]}) \dif \rho(D) = m_f(r, Z_{\widebar{\mathfrak w}(X_k, \mathfrak{S})}) + O(1)
.\]

In addition, choose a convenient basis $\{\varsigma_i \mid i = 0, 1, \dots, N\}$ of $\mathbb{S}$ such that $\{\varsigma_i \mid i = 0, 1, \dots, N'\}$ is a basis of $\mathbb{S}_{\supset Z}$,
where $N' = \dim \mathfrak{S}_{\supset Z}$.
We easily check that
\[
\widebar{\mathbb W}(X_k, Z, \mathfrak{S}) = \hspace{-1eM} \sum_{\substack{0 \leq i_1 \leq N' \\ 0 \leq i_1 < \dots < i_k \leq N}} \hspace{-1eM} \mathbb{S}_{i_1, \dots, i_k}
,\]
and
\[
Z_{\beta}(D) = \hspace{-1eM} \bigcap_{\substack{0 \leq i_1 \leq N' \\ 0 \leq i_1 < \dots < i_k \leq N}} \hspace{-1eM} \beta_i(D) = (D, Z)_{\mathfrak{S}}^{[k]}
,\]
for any $D \in \mathfrak{S}$.
It follows from Theorem \ref{thm:m_Zbetaavg} and Remark \ref{rmk:m_Zbetaavg} that
\[
\int_{D \in \mathfrak{S}} \!m_f(r, (D, Z)_{\mathfrak{S}}^{[k]}) \dif \rho(D) = m_f(r, Z_{\mathfrak{S}}^{[k]}) + O(1)
.\]

Moreover, observe that
\vspace{-1ex}
\[
\lambda_{\bigvee_{i=1}^q Z_i} \bydef \mathop{\max}_{1 \leq i \leq q} \lambda_{Z_i} = \sum_{j=1}^q (-1)^{j-1} \sum_{\substack{I \subset J \\ |I| = j}} \lambda_{Z_I}
.\]
By the definition of proximity functions and Proposition \ref{prp:cap_S^k=_S^kcap}, these formulas above lead to
\begin{align*}
\int_{D \in \mathfrak{S}} \!m_f(r, \!\bigvee_{i=1}^q (D, Z_i)_{\mathfrak{S}}^{[k]}) \dif \rho(D) &= \sum_{j=1}^q (-1)^{j-1} \sum_{\substack{I \subset J \\ |I| = j}} \int_{D \in \mathfrak{S}} \!m_f(r, (D, Z_I)_{\mathfrak{S}}^{[k]}) \dif \rho(D) \\
&= \sum_{j=1}^q (-1)^{j-1} \sum_{\substack{I \subset J \\ |I| = j}} m_f(r, (Z_I)_{\mathfrak{S}}^{[k]}) + O(1) \\
&= m_f(r, \!\bigvee_{i=1}^q (Z_i)_{\mathfrak{S}}^{[k]}) + O(1)
.\qedhere\end{align*}
\end{proof}

\begin{remark}
If $f \colon (\mathbb{C}, T_{\mathbb{C}}) \to (X, V)$ is non-degenerated relative to $\mathfrak{S}$, then $f$ satisfies all the non-degenerated condition in Lemma \ref{lem:(D,Z)_S^kavg}.
\end{remark}

By induction, we conclude the following theorem from Theorem \ref{thm:Z_S^1veesum}.
\begin{theorem}\label{thm:Z_S^kveesum}
Let $(X, V)$ be a directed projective manifold, and let $Z_1, \dots, Z_q$ be closed subschemes of $X$.
Choose a linear system $\mathfrak{S}$ on $X$ satisfying that $\mathcal{I}_{Z_i} = \mathfrak{s}_{\supset Z_i}$ for each $1 \leq i \leq q$.
Take an integer $1 \leq \ell \leq q$.
For each $1 \leq k \leq N$, let $f \colon (\mathbb{C}, T_{\mathbb{C}}) \to (X, V)$ be a non-constant holomorphic curve such that $f_{[k]}(\mathbb{C}) \not\subset \Bs(\mathfrak{W}(X_k, \mathfrak{S}))$.
Suppose that $f(\mathbb{C}) \not\subset \Supp (Z_1 + \dots + Z_q)$. We have
\begin{align*}
&m_{f_{[k-1]}}(r, \!\bigvee_{\substack{I \subset J \\ |I| = \ell}} \!\sum_{i \in I} (Z_i)_{\mathfrak S}^{[k-1]}) + T_{f_{[k]}}(r, \mathcal{O}_{X_k}(\ell)) + \ell \mathcal{N}([Z_{f_{[k-1]}'}], r) \\
\leq \, &m_{f_{[k]}}(r, \!\bigvee_{\substack{I \subset J \\ |I| = \ell}} \!\sum_{i \in I} (Z_i)_{\mathfrak S}^{[k]}) - \ell m_{f_{[k-1]}}(r, Z_{\widebar{\mathfrak w}(X_{k-1}, \mathfrak{S})}) + \ell m_{f_{[k-2]}}(r, Z_{\widebar{\mathfrak w}(X_{k-2}, \mathfrak{S})}) + S_f(r)
,\end{align*}
where $N = \dim \mathfrak{S}$ and we make the conventions that $(Z_i)_{\mathfrak S}^{[0]} = Z_{\mathfrak{s}_{\supset Z_i}}$, $\widebar{\mathfrak w}(X_0, \mathfrak{S}) = \mathfrak{s}$, and $m_{f_{[-1]}}(r, Z_{\widebar{\mathfrak w}(X_{-1}, \mathfrak{S})}) = 0$.
\end{theorem}

\begin{proof}
If $k=1$, the theorem definitely follows from Theorem \ref{thm:Z_S^1veesum}.
In the case when $k=2$, think of $(X_2, V_2)$ as the Demailly-Semple $1$-jet tower of $(X_1, V_1)$.
Take an arbitrary effective divisor $D \in \mathfrak{S}$.
For brevity, we write $(\widetilde{X}, \widetilde{V}) = (X_1, V_1)$, $\widetilde{\mathfrak S} = \mathfrak{w}(X_1, D, \mathfrak{S})$, and $\widetilde{Z}_i = (D, Z_i)_{\mathfrak S}^{[1]}$ for each $1 \leq i \leq q$.
Notice that $\mathcal{I}_{\widetilde{Z}_i} = \widetilde{\mathfrak S}_{\supset \widetilde{Z}_i}$.
Applying Theorem \ref{thm:Z_S^1veesum} to $\tilde{f} = f_{[1]}$ with respect to $\widetilde{\mathfrak S}$, we get
\begin{equation}\label{eqn:tildeZ_S^1veesum}
\begin{aligned}
&m_{\tilde f}(r, \!\bigvee_{\substack{I \subset J \\ |I| = \ell}} \!\sum_{i \in I} \widetilde{Z}_i) + T_{\tilde{f}_{[1]}}(r, \mathcal{O}_{\widetilde{X}_1}(\ell)) + \ell \mathcal{N}([Z_{\tilde{f}'}], r) \\
\leq \, &m_{\tilde{f}_{[1]}}\big(r, \!\bigvee_{\substack{I \subset J \\ |I| = \ell}} \!\sum_{i \in I} (\widetilde{Z}_i)_{\widetilde{\mathfrak S}}^{[1]}\big) - \ell m_{\tilde f}(r, Z_{\tilde{\mathfrak s}}) + S_{\tilde f}(r)
.\end{aligned}
\end{equation}
It follows from Lemma \ref{lem:(D,Z)_S^1_WD^k} and \eqref{eqn:Sfk} that \eqref{eqn:tildeZ_S^1veesum} is reformulated as
\begin{align*}
&m_{f_{[1]}}\big(r, \!\bigvee_{\substack{I \subset J \\ |I| = \ell}} \!\sum_{i \in I} (D, Z_i)_{\mathfrak S}^{[1]}\big) + T_{f_{[2]}}(r, \mathcal{O}_{X_2}(\ell)) + \ell \mathcal{N}([Z_{f_{[1]}'}], r) \\
\leq \, &\ell m_f(r, D) + m_{f_{[2]}}\big(r, \!\bigvee_{\substack{I \subset J \\ |I| = \ell}} \!\sum_{i \in I} (D, Z_i)_{\mathfrak S}^{[2]}\big) - \ell m_{f_{[1]}}(r, Z_{\mathfrak{w}(X_1, D, \mathfrak{S})}) + S_f(r)
.\end{align*}
Taking the average of two sides for $D \in \mathfrak{S}$, by Lemma \ref{lem:(D,Z)_S^kavg}, we conclude that
\begin{align*}
&m_{f_{[1]}}(r, \!\bigvee_{\substack{I \subset J \\ |I| = \ell}} \!\sum_{i \in I} (Z_i)_{\mathfrak S}^{[1]}) + T_{f_{[2]}}(r, \mathcal{O}_{X_2}(\ell)) + \ell \mathcal{N}([Z_{f_{[1]}'}], r) \\
\leq \, &m_{f_{[2]}}(r, \!\bigvee_{\substack{I \subset J \\ |I| = \ell}} \!\sum_{i \in I} (Z_i)_{\mathfrak S}^{[2]}) - \ell m_{f_{[1]}}(r, Z_{\mathfrak{w}(X_1, \mathfrak{S})}) + \ell m_f(r, Z_{\mathfrak s}) + S_f(r)
.\end{align*}

Assume the theorem establishes for $k=j$.
Analogous to the reasoning in the case when $k=2$, we are going to verify the theorem for $k=j+1$.
Applying the theorem with $k=j$ to $\tilde{f}$ with respect to $\widetilde{\mathfrak S}$, we obtain
\begin{align*}
&m_{\tilde{f}_{[j-1]}}(r, \!\bigvee_{\substack{I \subset J \\ |I| = \ell}} \!\sum_{i \in I} (\widetilde{Z}_i)_{\widetilde{\mathfrak S}}^{[j-1]}) + T_{\tilde{f}_{[j]}}(r, \mathcal{O}_{\widetilde{X}_j}(\ell)) + \ell \mathcal{N}([Z_{\tilde{f}_{[j-1]}'}], r) \\
\leq \, &m_{\tilde{f}_{[j]}}(r, \!\bigvee_{\substack{I \subset J \\ |I| = \ell}} \!\sum_{i \in I} (\widetilde{Z}_i)_{\widetilde{\mathfrak S}}^{[j]}) - \ell m_{\tilde{f}_{[j-1]}}(r, Z_{\widebar{\mathfrak w}(\widetilde{X}_{j-1}, \widetilde{\mathfrak S})}) + \ell m_{\tilde{f}_{[j-2]}}(r, Z_{\widebar{\mathfrak w}(\widetilde{X}_{j-2}, \widetilde{\mathfrak S})}) + S_{\tilde f}(r)
,\end{align*}
which amounts to saying that
\begin{align*}
&m_{f_{[j]}}(r, \!\bigvee_{\substack{I \subset J \\ |I| = \ell}} \!\sum_{i \in I} (D, Z_i)_{\mathfrak S}^{[j]}) + T_{f_{[j+1]}}(r, \mathcal{O}_{X_{j+1}}(\ell)) + \ell \mathcal{N}([Z_{f_{[j]}'}], r) \\
\leq \, &m_{f_{[j+1]}}(r, \!\bigvee_{\substack{I \subset J \\ |I| = \ell}} \!\sum_{i \in I} (D, Z_i)_{\mathfrak S}^{[j+1]}) - \ell m_{f_{[j]}}(r, D_{\mathfrak S}^{[j]}) + \ell m_{f_{[j-1]}}(r, D_{\mathfrak S}^{[j-1]}) + S_f(r)
.\end{align*}
Taking the average of two sides for $D \in \mathfrak{S}$, by Lemma \ref{lem:(D,Z)_S^kavg}, we have
\begin{align*}
&m_{f_{[j]}}(r, \!\bigvee_{\substack{I \subset J \\ |I| = \ell}} \!\sum_{i \in I} (Z_i)_{\mathfrak S}^{[j]}) + T_{f_{[j+1]}}(r, \mathcal{O}_{X_{j+1}}(\ell)) + \ell \mathcal{N}([Z_{f_{[j]}'}], r) \\
\leq \, &m_{f_{[j+1]}}(r, \!\bigvee_{\substack{I \subset J \\ |I| = \ell}} \!\sum_{i \in I} (Z_i)_{\mathfrak S}^{[j+1]}) - \ell m_{f_{[j]}}(r, Z_{\widebar{\mathfrak w}(X_{j}, \mathfrak{S})}) + \ell m_{f_{[j-1]}}(r, Z_{\widebar{\mathfrak w}(X_{j-1}, \mathfrak{S})}) + S_f(r)
.\qedhere\end{align*}
\end{proof}

For arbitrary closed subschemes $Z_1, \dots, Z_q$ of $X$, set
\[
\ell_{k, \mathfrak{S}} \bydef \max \{ |I| \mid I \subset J, \, \bigcap_{i \in I} (Z_i)_{\mathfrak S}^{[k]} \supsetneqq Z_{\widebar{\mathfrak w}(X_k, \mathfrak{S})} \} 
.\]
We easily check that $\ell_{k, \mathfrak{S}}$ is a non-increasing sequence for $k$.
Let $k_{\mathfrak S}$ be the minimum of $k$ such that $\ell_{k, \mathfrak{S}} = 0$.
Since $\ell_{N, \mathfrak{S}} = 0$, then $k_{\mathfrak S}$ always exists and $k_{\mathfrak S} \leq N$.

Using the notation above, we get the following theorem.
\begin{theorem}\label{thm:Z_S^ksum}
Let $(X, V)$ be a directed projective manifold, and let $Z_1, \dots, Z_q$ be closed subschemes of $X$.
Choose a linear system $\mathfrak{S}$ on $X$ satisfying that $\mathcal{I}_{Z_i} = \mathfrak{s}_{\supset Z_i}$ for each $1 \leq i \leq q$.
Let $\ell_{k, \mathfrak{S}}$ and $k_{\mathfrak S}$ be as above.
Let $f \colon (\mathbb{C}, T_{\mathbb{C}}) \to (X, V)$ be a non-constant holomorphic curve such that $f_{[k_{\mathfrak S}]}(\mathbb{C}) \not\subset \Bs(\mathfrak{W}(X_{k_{\mathfrak S}}, \mathfrak{S}))$.
Suppose that $f(\mathbb{C}) \not\subset \Supp (Z_1 + \dots + Z_q)$.
Then,
\begin{align*}
&\sum_{i =1}^q m_f(r, Z_i) + T_{f_{[k_{\mathfrak S}]}}(r, \mathcal{O}_{X_{k_{\mathfrak S}}}(\bm{\ell}_{\mathfrak S})) + \sum_{j=0}^{k_{\mathfrak S}-1} \ell_{j, \mathfrak{S}} \mathcal{N}([Z_{f_{[j]}'}], r) \\
\leq \, & \sum_{j=0}^{k_{\mathfrak S}} (\ell_{j-1, \mathfrak{S}} - 2\ell_{j, \mathfrak{S}} + \ell_{j+1, \mathfrak{S}}) m_{f_{[j]}}(r, Z_{\widebar{\mathfrak w}(X_{j}, \mathfrak{S})}) + S_f(r)
,\end{align*}
where $\bm{\ell}_{\mathfrak S} = (\ell_{0, \mathfrak{S}}, \dots, \ell_{k_{\mathfrak S}-1, \mathfrak{S}}) \in \mathbb{Z}_{+}^{k_{\mathfrak S}}$ and we make the conventions that $\ell_{-1, \mathfrak{S}} = q$, and $\widebar{\mathfrak w}(X_0, \mathfrak{S}) = \mathfrak{s}$.
\end{theorem}

\begin{proof}
By definition, for each $\ell_{k, \mathfrak{S}} < \ell \leq q$, we have
\[
\mathop{\max}_{\substack{I \subset J \\ |I| = \ell}} \lambda_{\bigcap_{i \in I} (Z_i)_{\mathfrak S}^{[k]}} =_{\langle X \rangle} \lambda_{Z_{\widebar{\mathfrak w}(X_k, \mathfrak{S})}}
.\]
According to Lemma \ref{lem:maxsum},
\[
\sum_{i =1}^q m_f(r, Z_i) = m_f(r, \hspace{-1eM} \bigvee_{\substack{I \subset J \\ |I| = \ell_{-1, \mathfrak{S}}}} \hspace{-1eM} \sum_{i \in I} Z_i) = m_f(r, \hspace{-.5eM}\!\bigvee_{\substack{I \subset J \\ |I| = \ell_{k, \mathfrak{S}}}} \hspace{-.5eM}\!\sum_{i \in I} Z_i) + (\ell_{-1, \mathfrak{S}}-\ell_{0, \mathfrak{S}}) m_{f_{[k]}}(r, Z_{\mathfrak s}) + O(1)
,\]
\[
m_{f_{[k]}}(r, \hspace{-1eM} \bigvee_{\substack{I \subset J \\ |I| = \ell_{k-1, \mathfrak{S}}}} \hspace{-1eM} \sum_{i \in I} (Z_i)_{\mathfrak S}^{[k]}) = m_{f_{[k]}}(r, \hspace{-.5eM}\!\bigvee_{\substack{I \subset J \\ |I| = \ell_{k, \mathfrak{S}}}} \hspace{-.5eM}\!\sum_{i \in I} (Z_i)_{\mathfrak S}^{[k]}) + (\ell_{k-1, \mathfrak{S}}-\ell_{k, \mathfrak{S}}) m_{f_{[k]}}(r, Z_{\mathfrak{w}(X_k, \mathfrak{S})}) + O(1)
.\]
Applying Theorem \ref{thm:Z_S^kveesum} with $\ell = \ell_{k-1, \mathfrak{S}}$, we obtain
\begin{align*}
&m_{f_{[k-1]}}(r, \hspace{-1eM} \bigvee_{\substack{I \subset J \\ |I| = \ell_{k-1, \mathfrak{S}}}} \hspace{-1eM} \sum_{i \in I} (Z_i)_{\mathfrak S}^{[k-1]}) + T_{f_{[k]}}(r, \mathcal{O}_{X_k}(\ell_{k-1, \mathfrak{S}})) + \ell_{k-1, \mathfrak{S}} \mathcal{N}([Z_{f_{[k-1]}'}], r) \\
\leq \, &m_{f_{[k]}}(r, \hspace{-1eM} \bigvee_{\substack{I \subset J \\ |I| = \ell_{k-1, \mathfrak{S}}}} \hspace{-1eM} \sum_{i \in I} (Z_i)_{\mathfrak S}^{[k]}) - \ell_{k-1, \mathfrak{S}} m_{f_{[k-1]}}(r, Z_{\widebar{\mathfrak w}(X_{k-1}, \mathfrak{S})}) + \ell_{k-1, \mathfrak{S}} m_{f_{[k-2]}}(r, Z_{\widebar{\mathfrak w}(X_{k-2}, \mathfrak{S})}) + S_f(r) \\
\leq \, &m_{f_{[k]}}(r, \hspace{-.5eM}\!\bigvee_{\substack{I \subset J \\ |I| = \ell_{k, \mathfrak{S}}}} \hspace{-.5eM}\!\sum_{i \in I} (Z_i)_{\mathfrak S}^{[k]}) + (\ell_{k-1, \mathfrak{S}} - \ell_{k, \mathfrak{S}}) m_{f_{[k-1]}}(r, Z_{\widebar{\mathfrak w}(X_k, \mathfrak{S})}) \\
&- \ell_{k-1, \mathfrak{S}} m_{f_{[k-1]}}(r, Z_{\widebar{\mathfrak w}(X_{k-1}, \mathfrak{S})}) + \ell_{k-1, \mathfrak{S}} m_{f_{[k-2]}}(r, Z_{\widebar{\mathfrak w}(X_{k-2}, \mathfrak{S})}) + S_f(r)
.\end{align*}
Combining all the inequalities for $1 \leq k \leq k_{\mathfrak S}$, we get
\begin{align*}
&\sum_{i =1}^q m_f(r, Z_i) + T_{f_{[k_{\mathfrak S}]}}(r, \mathcal{O}_{X_{k_{\mathfrak S}}}(\bm{\ell}_{\mathfrak S})) + \sum_{j=0}^{k_{\mathfrak S}-1} \ell_{j, \mathfrak{S}} \mathcal{N}([Z_{f_{[j]}'}], r) \\
\leq \, & \sum_{j=0}^{k_{\mathfrak S}} (\ell_{j-1, \mathfrak{S}} - 2\ell_{j, \mathfrak{S}} + \ell_{j+1, \mathfrak{S}}) m_{f_{[j]}}(r, Z_{\widebar{\mathfrak w}(X_{j}, \mathfrak{S})}) + S_f(r)
.\qedhere\end{align*}
\end{proof}

\begin{remark}
It is evident that Theorem \ref{thm:Z_S^ksum} still establishes for arbitrary $\ell_{0, \mathfrak{S}}', \dots, \allowbreak \ell_{N, \mathfrak{S}}' \in \mathbb{N}$ such that $\ell_{k, \mathfrak{S}} \leq \ell_{k, \mathfrak{S}}'$ for each $0 \leq k \leq N$.
In practice, we usually consider closed subschemes $Z_1, \dots, Z_q$ of $X$ satisfying a certain condition.
Then we technically take $\ell_{k, \mathfrak{S}}$ to be the maximum of $\ell_{k, \mathfrak{S}}$ for all the finite sets of closed subschemes of $X$ satisfying such condition.
\end{remark}

\subsection{Second Main Theorem for some particular cases}

Now, let us investigate the particular case where $(X, V) = (\mathbb{P}^n, T_{\mathbb{P}^n})$ and $Z_1, \dots, Z_q$ are the hyperplanes $H_1, \dots, H_q$ of $\mathbb{P}^n$ in general position.
Then choose the complete linear system $\mathfrak{S} = |H_1|$.
Obviously, $\dim \mathbb{S} = n+1$ and $\dim \mathbb{S}_{\supset \bigcap_{i \in I} H_i} > n-k+1$ for all the subsets $I \subset J$ with $|I| > n-k$.
We see from the pigeonhole principle that $\bigcap_{i \in I} (H_i)_{\mathfrak S}^{[k]} = Z_{\widebar{\mathfrak w}(X_k, \mathfrak{S})}$ for all the subsets $I \subset J$ with $|I| > n-k$, namely, $\ell_{k, \mathfrak{S}} \leq n-k$.
Hence we conclude Cartan's second main theorem for holomorphic curves.
\begin{theorem}\label{thm:Cartan}[Cartan's SMT]
Let $H_1, \dots, H_q$ be the hyperplanes of $\mathbb{P}^n$ in general position.
For $f \colon \mathbb{C} \to \mathbb{P}^n$ be a linear non-degenerated holomorphic curve, we have
\[
\sum_{i=1}^q m_f(r, H_i) + \mathcal{N}([Z_{W(f)}], r) \leq T_f(r, \mathcal{O}_{\mathbb{P}^n}(n+1)) + S_f(r)
,\]
where the Wronskian $W(f) = F^n$ defined by \eqref{eqn:F^k}.
\end{theorem}

\begin{proof}
Choose the linear system $\mathfrak{S} = |H_1|$ as above.
Thus we can take $k_{\mathfrak S} = n$, and $\ell_{k, \mathfrak{S}} = n-k$.
Applying Theorem \ref{thm:Z_S^ksum}, we have
\begin{align*}
&\sum_{i =1}^q m_f(r, H_i) + T_{f_{[n]}}(r, \mathcal{O}_{X_{n}}(\bm{a}^n)) + \sum_{j=0}^{n-1} (n-j) \mathcal{N}([Z_{f_{[j]}'}], r) \\
\leq \, &(q-n-1) m_{f}(r, Z_{\mathfrak s}) + m_{f_{[n]}}(r, Z_{\widebar{\mathfrak w}(X_{n}, \mathfrak{S})}) + S_f(r)
.\end{align*}
Since $\mathfrak{S}$ is base point free, $Z_{\mathfrak s} = \emptyset$.
And $Z_{\widebar{\mathfrak w}(X_{n}, \mathfrak{S})}$ is indeed an effective divisor associated to the only $n$-th Wronskian global section $\overline{\omega}(\varsigma_0, \dots, \varsigma_n) \in H^0(X_n, \mathcal{O}_{X_n}(\bm{a}^n) \otimes \pi_{0, n}^{\ast} \mathcal{O}_{\mathbb{P}^n}(n+1))$.
That is to say,
\begin{equation}\label{eqn:TfZwn}
T_{f_{[n]}}(r, \mathcal{O}_{X_{n}}(\bm{a}^n)) + T_f(r, \mathcal{O}_{\mathbb{P}^n}(n+1)) = T_{f_{[n]}}(r, Z_{\widebar{\mathfrak w}(X_{n}, \mathfrak{S})}) + O(1)
.\end{equation}
Therefore,
\[
\sum_{i =1}^q m_f(r, H_i) + N_{f_{[n]}}(r, Z_{\widebar{\mathfrak w}(X_{n}, \mathfrak{S})}) + \sum_{j=0}^{n-1} (n-j) \mathcal{N}([Z_{f_{[j]}'}], r)
\leq T_f(r, \mathcal{O}_{\mathbb{P}^n}(n+1)) + S_f(r)
.\]
Observe that
\[
\mathcal{N}([Z_{W(f)}], r) = N_{f_{[n]}}(r, Z_{\widebar{\mathfrak w}(X_{n}, \mathfrak{S})}) + \sum_{j=0}^{n-1} (n-j) \mathcal{N}([Z_{f_{[j]}'}], r)
.\]
The theorem is verified.
\end{proof}

On the other hand, we easily see that $P^{(1)} = \emptyset$, if $P = \{p\}$ is the set of a point in $X$.
According to Theorem \ref{thm:Z^(1)vee}, we immediately obtain the following corollary.
\begin{corollary}\label{cor:pointsSMT}
Let $(X, V)$ be a directed projective manifold, and let $P_1, \dots, P_q$ be the distinct points in $X$.
Let $f \colon (\mathbb{C}, T_{\mathbb{C}}) \to (X, V)$ be a non-constant holomorphic curve.
Then we have
\[
\sum_{i=1}^q m_f(r, P_i) + T_{f_{[1]}}(r, \mathcal{O}_{X_1}(1)) + \mathcal{N}([Z_{f'}], r) \leq S_f(r)
.\]
\end{corollary}

More specifically, if $X = \mathbb{P}^n$ and $V = T_{\mathbb{P}^n}$, we get the second main theorem for holomorphic curves intersecting points in projective spaces.
\begin{theorem}
Let $P_1, \dots, P_q$ be the distinct points in $\mathbb{P}^n$.
For $f \colon \mathbb{C} \to \mathbb{P}^n$ be a linear non-degenerated holomorphic curve, we have
\[
n \sum_{i=1}^q m_f(r, P_i) + \frac{2}{n+1} \mathcal{N}([Z_{W(f)}], r) \leq T_f(r, \mathcal{O}_{\mathbb{P}^n}(2)) + S_f(r)
.\]
\end{theorem}

\begin{proof}
Applying Corollary \ref{cor:pointsSMT} to $f$, we obtain
\begin{equation}\label{eqn:npointsSMT}
n' \sum_{i=1}^q m_f(r, P_i) + T_{f_{[1]}}(r, \mathcal{O}_{X_1}(n')) + n' \mathcal{N}([Z_{f'}], r) \leq S_f(r)
,\end{equation}
where $n' \bydef 1 + \dots + n = \dfrac{n(n+1)}{2}$.

At the same time, it follows from \eqref{eqn:Gammafk-1'} that
\[
\mathcal{N}([Z_{f_{[j-1]}'}], r) = N_{f_{[j+1]}}(r, \varGamma_{j+1}) + \mathcal{N}([Z_{f_{[j]}'}], r)
\]
for each $1 \leq j \leq n-1$.
Substituting all these equalities for $j=1, \dots, n-1$ into \eqref{eqn:npointsSMT}, by \eqref{eqn:TfZwn}, we get
\[
n' \sum_{i =1}^q m_f(r, P_i) + N_{f_{[n]}}(r, Z_{\widebar{\mathfrak w}(X_{n}, \mathfrak{S})}) + \sum_{j=0}^{n-1} (n-j) \mathcal{N}([Z_{f_{[j]}'}], r)
\leq T_f(r, \mathcal{O}_{\mathbb{P}^n}(n+1)) + S_f(r)
.\]
Therefore,
\[
n' \sum_{i=1}^q m_f(r, P_i) + \mathcal{N}([Z_{W(f)}], r) \leq T_f(r, \mathcal{O}_{\mathbb{P}^n}(n+1)) + S_f(r)
.\qedhere\]
\end{proof}

\bibliographystyle{amsalpha}
\renewcommand\baselinestretch{1.4}
\bibliography{ref}

\end{document}